\def\forprint{1}
\def\dark{0}
\if\forprint0
	\documentclass[a5paper]{article}
	\usepackage[margin={-15pt, 10pt}, centering]{geometry}
\else
	\documentclass{amsart}
	\usepackage[margin={100pt, 100pt}, centering]{geometry}
\fi

\usepackage[utf8]{inputenc}
\usepackage[T1]{fontenc}
\usepackage{textcomp}
\usepackage{amsmath, amssymb}
\usepackage{amsthm}
\usepackage{mathrsfs}
\usepackage{esint}
\usepackage{xcolor}
\usepackage{import}
\usepackage{xifthen}
\usepackage{pdfpages}
\usepackage{transparent}
\usepackage[shortlabels]{enumitem}
\usepackage{hyperref}
\usepackage{subcaption}
\usepackage{tcolorbox}
\usepackage{subcaption}

\usepackage{mathtools}
\usepackage{amssymb}
\usepackage[shortlabels]{enumitem}

\usepackage{import}
\usepackage{xifthen}
\usepackage{pdfpages}
\usepackage{transparent}
\usepackage{subcaption}
\usepackage{cleveref}

\newtheorem{theorem}{Theorem}[section]

\newtheorem{proposition}[theorem]{Proposition}
\newtheorem{corollary}[theorem]{Corollary}

\newtheorem{hypothesis}[theorem]{Hypothesis}

\newtheorem{lemma}[theorem]{Lemma}
\newtheorem{remark}[theorem]{Remark}
\numberwithin{equation}{section}
\theoremstyle{definition} 
\newtheorem{definition}[theorem]{Definition}
\newtheorem{example}[theorem]{Example}

\begin{document}
\if\forprint0
\if\dark1
	\definecolor{background}{RGB}{40,44,52}
	\definecolor{normal_color}{RGB}{171,178,191}
\else
	\definecolor{background}{RGB}{251,241,199}
	\definecolor{normal_color}{RGB}{40,40,40}
\fi
\else
	\definecolor{background}{RGB}{255,255,255}
	\definecolor{normal_color}{RGB}{0,0,0}
\fi
\title{Allard-Type Regularity for Varifolds with Prescribed Contact Angle}
\author{Gaoming Wang}
\date{\today}
\address{Yau Mathematical Sciences Center, Tsinghua University, Beijing, China}
\email{gmwang@tsinghua.edu.cn}
\maketitle
\pagecolor{background}
\color{normal_color}
\begin{abstract}
Given a bounded $C^2$ domain in $\mathbb{R}^{n+1}$ and an integral $n$-rectifiable varifold $V$ with bounded first variation and bounded generalized mean curvature.
	Given a $C^1$ function $\theta$ defined on the boundary of the domain with range $(0,\pi)$, we assume $V$ has prescribed contact angle $\theta$ with $\partial \Omega$ and the tangent cone of $V$ at a point $X \in \partial \Omega$ is a half-hyperplane of density one.
Then we can show that the support of $V$ is a $C^{1,\gamma}$ hypersurface with boundary near $X$ for some $\gamma \in (0,1)$.
\end{abstract}

\section{Introduction}%
\label{sec:introduction}



In the field of fluid mechanics, capillary surfaces serve as mathematical models to describe the interfaces between two immiscible fluids.

Consider a container, denoted by $\Omega$, with its boundary $\partial \Omega$, containing two immiscible fluids. The equilibrium configuration of the interface separating these fluids can be depicted as a capillary surface. Drawing from Young's work, the mean curvature of this interface, along with the contact angle it forms with the container's boundary, is determined by the surface energy, wetting energy, potential energy, and a Lagrange multiplier which comes from the volume of fluids.

For clarity, let us assume that $\Omega$ represents a smoothly bounded, open domain within $\mathbb{R}^{n+1}$, with $E\subset \Omega$ indicating the region occupied by one of the fluids. We analyze the following free energy expression,
\[
	\mathcal{F}(E):=\int_{ \partial E \cap \Omega} d \mathcal{H}^n-\int_{ \partial E \cap \partial \Omega}\vartheta(x)d\mathcal{H}^n+\int_{ E} g(x)dx, 
\]
where $\mathcal{H}^{n}$ denotes the $n$-dimensional Hausdorff measure, $\sigma$ is a smooth function on $\partial \Omega$ with a range of $[-1,1]$, and $g$ is a smooth potential function over $\Omega$. 
If $E$ reaches a critical point for the above energy with the condition $\mathcal{H}^{n+1}(E)$ fixed, and $\Sigma:=\partial E \cap \Omega$ is a smooth hypersurface with boundary, then the following conditions are satisfied,
\[
	\begin{cases}
		\boldsymbol{H}_\Sigma=g(x)+\lambda,& \quad x \in \Sigma,\\
		\cos \theta(x)=\vartheta(x), &\quad x \in \partial \Sigma,
	\end{cases}
\]
where $\boldsymbol{H}_\Sigma$ represents the mean curvature of $\Sigma$, $\theta(x)$ the contact angle between $\Sigma$ and $\partial \Omega$, and $\lambda$ a constant determined by the volume constraint. 
The study of such capillary surfaces, regarding their existence and regularity, has been pursued by J. Taylor \cite{Taylor1977regularity} for $n=3$, and by De Phillippis, Maggi \cite{DePhilippisMaggi2014RegularityAnisotropic} for the anisotropic case across general $n$.
If we only focus on the case $\varphi\equiv 0$ (the free boundary case), there were numerous works like \cite{Gruter1986allard, Gruter1987optimalRegularityFB, Gruter1987regularityCurrentFB}, and \cite{Lewy1951minimalPartialFB,Hildebrandt1969boundaryFB,Hildebrandt1979minimalFB} and so on.
The min-max theory, recently receiving significant attention, presents a powerful method to affirm the existence of minimal hypersurfaces, as demonstrated by Li-Zhou-Zhu \cite{Li2021Capillary}, and De Masi-De Philippis \cite{DeMasi2021CapillaryMinMax} for capillary minimal hypersurfaces, and \cite{Li2021min_max} for the free boundary minimal hypersurface.

This paper concentrates on the Allard-type regularity of capillary surfaces. Prior to detailing our main theorem, it is necessary to define the concept of $n$-varifold with a specified contact angle, a notion previously introduced by Kagaya and Tonegawa \cite{Tonegawa2017fixAngle} for general varifolds, and further extended by De Masi-De Philippis \cite{DeMasi2021CapillaryMinMax}.
Our approach aligns with that of Kagaya and Tonegawa, as it sufficiently captures the intended contact angle upon achieving regularity under some natural conditions. (See Corollary \ref{cor_contactAngleEuclidean}.)

Although our main results are applicable in general Riemannian manifolds, for simplicity, we only discuss definitions and results in the context of Euclidean space in this section, directing readers to Section \ref{sec:main_theorem} for a comprehensive treatment.

Let's consider $\Omega$ as a bounded closed set in $\mathbb{R}^{n+1}$ with a $C^2$ boundary $\partial \Omega$, and $\theta$ as a $C^1$ function on $\partial \Omega$ ranging between $(0,\pi)$.
We will use the notations for $n$-varifolds in\cite{Allard1972} and \cite{simon1983lectures} (see \ref{sub:notations} for the list of notations we used in this paper).
Given an integral $n$-rectifiable varifold $V$ defined on $\mathbb{R}^{n+1}$ such that it supports in the closure of $\Omega$ and has bounded first variation and bounded generalized mean curvature.
We say that $V$ has \textit{prescribed contact angle $\theta$} if there exists a $\mathcal{H}^n$-measurable subset $U \subset \partial \Omega$ such that following condition holds for any compactly supported $C^1$ vector field $\varphi$ tangential to $\partial \Omega$ when restricted to $\partial \Omega$,
\[
	\delta V(\varphi)-\delta |U|(\cos \theta \varphi)=\int_{ } \left< \boldsymbol{H}, \varphi \right> d\|V\|,
\]
where $\delta V$ represents the first variation of $V$, $|U|$ the $n$-varifold associated with $U$, and $\boldsymbol{H}$ the generalized mean curvature of $V$.

\begin{theorem}
	Suppose $V$ is the integral $n$-rectifiable varifold described above with prescribed contact angle $\theta$, and suppose $\boldsymbol{H}$ is uniformly bounded.
	If at a point $X \in \mathrm{spt}\|V\| \cap \partial \Omega$, $V$ has a multiplicity-one tangent half-hyperplane at $X$, then the support of $\|V\|$ near $X$ is a $C^{1,\gamma}$ hypersurface with boundary for some $\gamma \in (0,1)$.
	\label{thm_mainEuclidean}
\end{theorem}

As a corollary, we can obtain the following result. (See Corollary \ref{cor_contactAngleG} for more precise statement.)
\begin{corollary}
	If $V$ is the varifold that meets the conditions of Theorem \ref{thm_mainEuclidean} at a point $X \in \mathrm{spt}\|V\|\cap \partial \Omega$, then for any $Y \in \mathrm{spt}\|V\|\cap \partial \Omega$ near $X$, the contact angle between $\mathrm{spt}\|V\|$ and $\partial \Omega$ is either $\theta(Y)$ or $\frac{\pi}{2}$ under the suitable choice of the normal direction of $\partial M$ and the orientation of $\mathrm{spt}\|V\|$ near $X$.

	Furthermore, if $\theta\neq \frac{\pi}{2}$ on the support of $\|V\|\cap \partial \Omega$ and the tangent space of $\|V\|$ at $X$ is not orthogonal to $\partial \Omega$, then the contact angle between $\Sigma$ and $\partial \Omega$ is exactly $\theta$.
	\label{cor_contactAngleEuclidean}
\end{corollary}

When considering $\theta$ approaching $\frac{\pi}{2}$ or $n=2$, the regularity of the varifold can be inferred solely from the density condition.
\begin{theorem}
	\label{thm_densityEuclidean}
	Given $\Lambda_0 \in (0,1)$, suppose $V$ is the integral $n$-rectifiable varifold described above with prescribed contact angle $\theta$ with $\theta \in [\Lambda_0,\pi-\Lambda_0]$, and suppose $\boldsymbol{H}$ is uniformly bounded.
	Then there exists a positive constant $\varepsilon=\varepsilon(n,\Lambda_0)>0$, such that the following conclusions hold,
	\begin{enumerate}[\normalfont(a)]
		\item The dimension $n$ equals 2, or the range of $\theta$ is within $(\frac{\pi}{2}-\varepsilon,\frac{\pi}{2}+\varepsilon)$.
		\item At a point $X \in \mathrm{spt}\|V\|\cap \partial \Omega$, we have
	\[
		\liminf_{\rho \rightarrow 0^+} \frac{\|V\|(B_\rho(X))}{\omega_n\rho^n}\le \frac{1}{2}+\varepsilon.
	\]
	\end{enumerate}
	Then, the support of $\|V\|$ near $X$ is a $C^{1,\gamma}$ hypersurface with boundary for some $\gamma \in (0,1)$.
Here, $\omega_n$ is the volume of $n$-dimensional unit ball in Euclidean space.
\end{theorem}

Due to the lack of a monotonicity formula for $\frac{\|V\|(B_\rho(X))}{\omega_n \rho^n}$ for $X \in \partial \Omega$, we are unable to determine whether the density of $\|V\|$ exists in general.
However, a similar monotonicity formula, which incorporates a boundary term, has been previously established in \cite{Tonegawa2017fixAngle}.
We reproof under the $C^1$-metric setting, and it plays a crucial role in our proof of the main result.

The core of our proof strategy is inspired by the regularity work of L. Simon \cite{Simon1993cylindrical}.
To illustrate, consider a scenario where $\theta$ equals $\frac{\pi}{3}$ and $\Omega$ is a half-space in $\mathbb{R}^{n+1}$.
By reflecting the varifold $V$ across $\Omega$'s boundary, we obtain a new varifold $V'$. 
Then, we can find an integral $n$-rectifiable varifold $W$ supported on $\Omega$'s boundary such that varifold $V+V'+W$ is a stationary varifold in $\mathbb{R}^{n+1}$.
The application of Simon's regularity theory tells us that near $X$, $V$ exhibits regularity if it has a multiplicity one tangent half-hyperplane at $X$.

\begin{figure}[ht]
    \centering
	\begingroup
	\fontsize{7pt}{12pt}
	\def\svgwidth{0.4\columnwidth}
\begingroup%
  \makeatletter%
  \providecommand\color[2][]{%
    \errmessage{(Inkscape) Color is used for the text in Inkscape, but the package 'color.sty' is not loaded}%
    \renewcommand\color[2][]{}%
  }%
  \providecommand\transparent[1]{%
    \errmessage{(Inkscape) Transparency is used (non-zero) for the text in Inkscape, but the package 'transparent.sty' is not loaded}%
    \renewcommand\transparent[1]{}%
  }%
  \providecommand\rotatebox[2]{#2}%
  \newcommand*\fsize{\dimexpr\f@size pt\relax}%
  \newcommand*\lineheight[1]{\fontsize{\fsize}{#1\fsize}\selectfont}%
  \ifx\svgwidth\undefined%
    \setlength{\unitlength}{680.31496063bp}%
    \ifx\svgscale\undefined%
      \relax%
    \else%
      \setlength{\unitlength}{\unitlength * \real{\svgscale}}%
    \fi%
  \else%
    \setlength{\unitlength}{\svgwidth}%
  \fi%
  \global\let\svgwidth\undefined%
  \global\let\svgscale\undefined%
  \makeatother%
  \begin{picture}(1,0.5)%
    \lineheight{1}%
    \setlength\tabcolsep{0pt}%
    \put(0,0){\includegraphics[width=\unitlength,page=1]{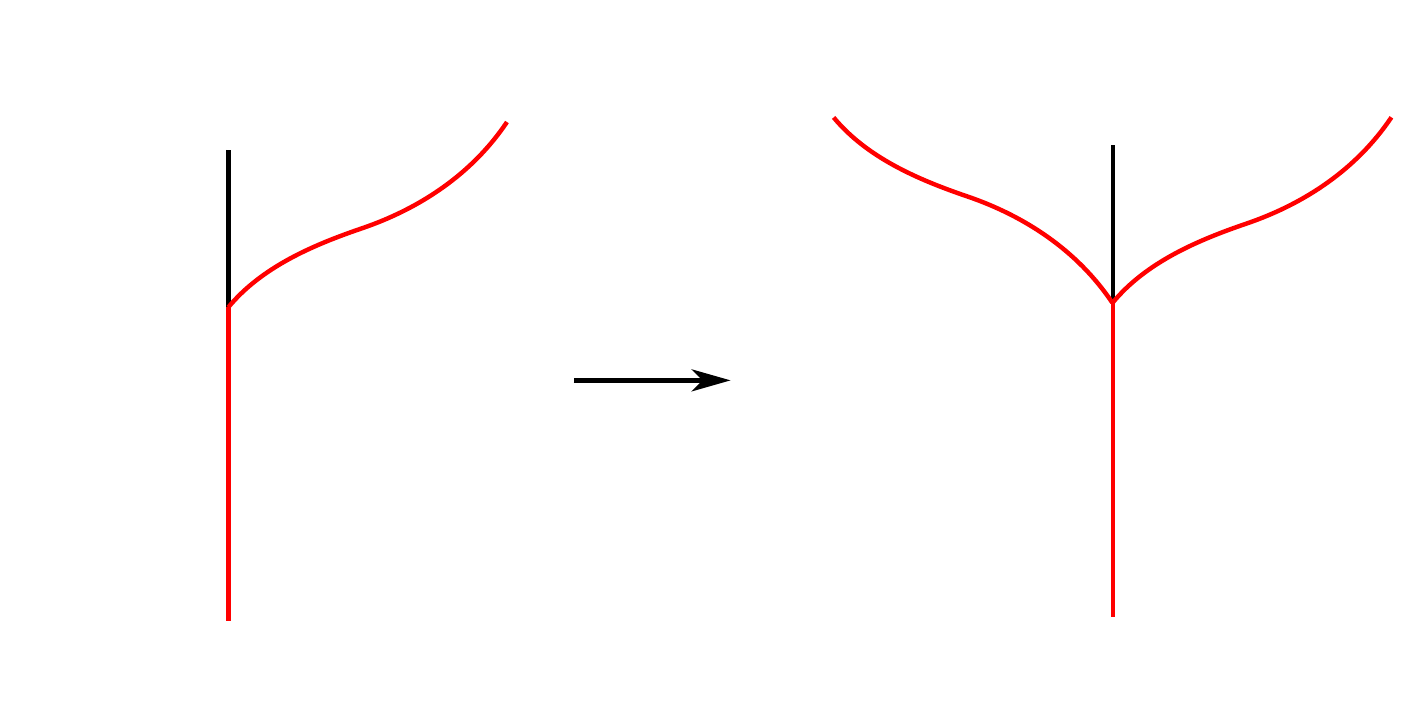}}%
    \put(0.2426391,0.29051453){\color[rgb]{0,0,0}\makebox(0,0)[lt]{\lineheight{0}\smash{\begin{tabular}[t]{l}$V$\end{tabular}}}}%
    \put(0.16664031,0.17208189){\color[rgb]{0,0,0}\makebox(0,0)[lt]{\lineheight{0}\smash{\begin{tabular}[t]{l}$W$\end{tabular}}}}%
    \put(0.87421351,0.29152974){\color[rgb]{0,0,0}\makebox(0,0)[lt]{\lineheight{0}\smash{\begin{tabular}[t]{l}$\Sigma$\end{tabular}}}}%
    \put(0.7906162,0.17208189){\color[rgb]{0,0,0}\makebox(0,0)[lt]{\lineheight{0}\smash{\begin{tabular}[t]{l}$W$\end{tabular}}}}%
    \put(0.63285003,0.30826532){\color[rgb]{0,0,0}\makebox(0,0)[lt]{\lineheight{0}\smash{\begin{tabular}[t]{l}$V'$\end{tabular}}}}%
  \end{picture}%
\endgroup%

	\endgroup

    \caption{The stationary varifold after reflection}
    \label{fig:reflection}
\end{figure}
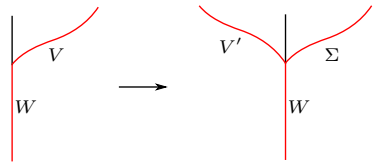

In general, when we want to prove the Allard-type regularity in a general Riemannian manifold $M$, we may consider isometrically embedding $M$ into a higher dimensional Euclidean space by Nash embedding theorem.
However, we need to work on the codimension one to ensure the contact angle condition is well-defined.
Our approach involves the development of a varifold theory under a $C^1$ metric, focusing particularly on the standard half-space $\mathbb{H}^{n+1}$ equipped with this metric.
Note that the metric can be chosen to be sufficiently close to the Euclidean metric, after an appropriate scaling.
This approach mirrors the method employed by Schoen-Simon \cite{SchoenSimon1981Regularity} to prove the regularity of the stable minimal hypersurfaces. In particular, the monotonicity formula remains valid under the $C^1$ metric.
This formula leads to the derivation of a crucial $L^2$-estimate (Theorem \ref{thm__l_2_estimate}) for the varifold $V$ with a prescribed contact angle, thereby enabling us to achieve the decay of the $L^2$-excess (Proposition \ref{prop_ImproveExcessRho}) and, consequently, to finish the proof of our main theorem.

The paper is organized as follows.
Section \ref{sec:main_theorem} sets the stage by defining the necessary notation and presenting preliminary results concerning varifolds under the $C^1$ metric.
Following this, we state our main theorem in the context of a general metric $g$.
The notion of a $\mu$-stationary quadruple, introduced in Section \ref{sec:_mu_stationary_quadruple}, assumes the metric $g$ is close to the Euclidean metric and the contact angle $\theta$ is close to a constant.
We prove that various fundamental results applicable to stationary $n$-varifolds also pertain to the $\mu$-stationary quadruple, including the monotonicity formula and compactness theorem.
The subsequent sections, from Section \ref{sec:l2_estimate} through Section \ref{sec:iterations}, concentrate on scenarios where $\theta$ is greater than $\frac{\pi}{2}+\Lambda_0$ for a certain $\Lambda_0$ in the interval (0,1).
To make our proof clear, we follow the strategy of Wickramasekera \cite{Wickramasekera2014Regularity} and divide the proof into several steps.
At first, we prove the validity of the $L^2$-estimate for the $\mu$-stationary quadruple in Section \ref{sec:l2_estimate}, which then allows us to define a blow-up sequence and show that the associated blow-up is at least a $C^2$ function in Section \ref{sec:blow_up_argument}.
This will lead to the $L^2$-excess decay in Section \ref{sec:iterations} and conclusively prove the theorem for cases where $\theta$ is greater than $\frac{\pi}{2}+\Lambda_0$.
Section \ref{sec:free_boundary_case} addresses situations where $\theta$ closely to $\frac{\pi}{2}$, highlighting that only minor adjustments are necessary compared to the previous sections.
Lastly, we collect some basic facts about the varifold theory under $C^1$ metric in Appendix \ref{sec:varifolds} and an elementary result in spherical geometry in Appendix \ref{sec:append_stationaryNetwork} in order to prove Theorem \ref{thm_densityEuclidean}.

\section{Main Theorem}%
\label{sec:main_theorem}
\subsection{Basic Notation}%
\label{sub:notations}

It is necessary to define some notations for our discussion,

\begin{itemize}
	\item We denote points in $\mathbb{R}^{n+1}$ using capital letters $X, Y, Z$. Specifically, we write $X = (x_1, x_2, \ldots, x_{n+1})$. For each $i = 1, 2, \ldots, n+1$, we represent $e_i = \frac{\partial}{\partial x_i}$. Commonly, we express $X$ as $(x, y)$ or $(\xi, \eta)$, where $x, \xi \in \mathbb{R}^2$ and $y, \eta \in \mathbb{R}^{n-1}$. Additionally, $Y = (0, y)$ for $y \in \mathbb{R}^{n-1}$, and $Z = (0, \zeta, \eta)$ for $\zeta \in \mathbb{R}$, $\eta \in \mathbb{R}^{n-1}$.
	\item For any point $X = (\xi, \eta) \in \mathbb{R}^{n+1}$, we define $r(X) = |\xi|$.
	\item $B^N_r(X)$ denotes the standard Euclidean ball in $\mathbb{R}^N$. When $N = n+1$, we simplify this to $B_r(X)$ instead of $B_r^{n+1}(X)$. If $X = 0$, we further abbreviate it to $B_r^N$ instead of $B_r^N(0)$.
	\item We define $\mathbb{H}^{n+1}$ as the set $\{ X = (x_1, x_2, y) \in \mathbb{R}^2 \times \mathbb{R}^{n-1} : x_1 \ge 0 \}$.
		\item $\omega_n$ represents the volume of the $n$-dimensional unit ball.

	\item $\mathcal{B}_r(X)$ refers to the intersection of $B_r(X)$ with $\mathbb{H}^{n+1}$, representing an open Euclidean ball in $\mathbb{H}^{n+1}$. Here, $X$ can be any point in $\mathbb{R}^{n+1}$.
		
	\item For any point (or vector) $X \in \mathbb{R}^{n+1}$, with $X = (x_1, x_2, y)$, we define $\tilde{X} = (-x_1, x_2, y)$ as the reflection point (or vector) across the boundary plane $\partial \mathbb{H}^{n+1}$.
	\item For any (relatively) open subset $U \subset \mathbb{R}^{n+1}$ or $\mathbb{H}^{n+1}$, we introduce,
	\begin{itemize}
		\item $\mathfrak{X}^1(U)$ as the set of $C^1$ class vector fields on $U$, 
		\item $\mathfrak{X}^1_c(U)$ as the subset of $\mathfrak{X}^1(U)$ with vector fields that have compact support within $U$.
	\end{itemize}
	When $U$ is an open set within $\mathbb{H}^{n+1}$, we specifically define $\mathfrak{X}^1_{\tan}(U)$ and $\mathfrak{X}^1_{c,\tan}(U)$ to account for vector fields tangent to the boundary of $\mathbb{H}^{n+1}$.

	\item If $H$ is a half-hyperplane in $\mathbb{H}^{n+1}$, $\vec{n}_H$ denotes the unit normal vector to $H$ pointing upward (i.e., $\vec{n}_H \cdot e_2\ge 0$).
		Note that $\vec{n}_H$ is unique if $e_2$ is not parallel to $H$.
	\item For any two $n$-dimensional subspace $P,P'$ of $\mathbb{R}^{n+1}$, we define $|P-P'|$ as the Hausdorff distance between $P \cap B_1$ and $P' \cap B_1$.
		Similarly, for any two $n$-dimensional half-space $H, H'$ in $\mathbb{H}^{n+1}$ (the half-hyperplanes containing the origin), we define $|H - H'|$ as the Hausdorff distance between $H\cap \mathcal{B}_1$ and $H'\cap \mathcal{B}_1$.
	\item For any $n$-plane $P$ in $\mathbb{R}^{n+1}$ and a point $X \in \mathbb{R}^{n+1}$, we write $X^{\top_S}$ as the orthogonal projection of $X$ onto $P$ and $X^{\bot_S}$ as the orthogonal projection of $X$ onto $P^{\bot}$.
	\item $\eta_{X,r}(Y)=\frac{Y-X}{r}$, $\tau_X(Y)=Y-X$.
	\item We define $H^\theta=\left\{ (r \cos \theta,r \sin \theta,y): r \in \mathbb{R}^+, y \in \mathbb{R}^{n-1}\right\}$ for $\theta \in (0,\pi)$.
		In particular, $H^0=\left\{ x_2<0,x_1=0 \right\}$.
	\item 
		$N(H):=\left\{ X=(r\cos \theta',r\sin \theta',y) \in \mathbb{H}^{n+1}: r \in \mathbb{R}^+,|\theta'-\theta|<\frac{\pi}{8}  \right\}$ if $H=H^\theta$ for some $\theta \in (0,\pi)$.
	\item For any metric $g$ (or a general symmetric 2-tensor) defined on $U\subset \mathbb{R}^{n+1}$ or $\mathbb{H}^{n+1}$, and for vectors $\varphi,\varphi' \in T_XU$ at a point $X \in U$, we define $\left< \varphi,\varphi' \right> _g$ as the inner product of these two vectors under metric $g$.
		Sometimes, this is alternatively expressed as $\left< \varphi,\varphi' \right> _{g(X)}$ to emphasize the evaluation of the inner product at the specific point $X$.
		On local coordinate $\left\{ e_i \right\}_{i=1}^{n+1}$, we write 
		\[
			g_{ij}(X)=\left< e_i,e_j \right>_{g(X)}.
		\]
		In particular, we denote by $\delta = (\delta_{ij})_{i,j=1}^{n+1}$ the standard Euclidean metric on $\mathbb{R}^{n+1}$, and we use the notation $\varphi \cdot \varphi' := \langle \varphi, \varphi' \rangle_{\delta}$.
		The norm of a vector $\varphi$ under metric $g$ is represented as $|\varphi|_g = \sqrt{\langle \varphi, \varphi \rangle_g}$.
	\item For any $N$ by $N$ matrix $L=(l_{ij})_{i=1}^N$, we denote by $|L|=\sqrt{\sum_{i,j=1}^N l _{ij}^2}$ the $L^2$ norm of $L$.
		Hence, for any symmetric two-tensor $h$ defined on $U\subset \mathbb{R}^{n+1}$ or $\mathbb{H}^{n+1}$, we define the norm of $|h|$ at $X$ as $|h|(X)=\sqrt{\sum_{i,j=1}^{n+1} h_{ij}^2(X)}$.
\end{itemize}

\subsection{Measures and Varifolds}%
\label{sub:measures_and_varifolds}

In this subsection, we revisit some basic facts of measures and varifolds, directing readers to \cite{Allard1972} and \cite{simon1983lectures} for comprehensive details.

For any Radon measure $\mu$ defined on a metric space $M$, its support, denoted $\mathrm{spt}\mu$, is defined as
\[
	\mathrm{spt}\mu=\left\{ X \in M:\mu(U)>0 \text{ for any open neighborhood }U \text{ of }X \right\}.
\]
When $\mu$ is a (signed) Radon measure on $\mathbb{R}^{n+1}$, the $k$-dimensional density of $\mu$ at a point $X$, represented as $\Theta^k(\mu,X)$, is defined by
\[
	\Theta^k(\mu,X):=\lim_{\rho \rightarrow 0^+} 
	\frac{\mu(B_\rho(X))}{\omega_k\rho^k}
\]
provided the limit exists. For simplicity, we use $\Theta(\mu,X) = \Theta^n(\mu,X)$ in this paper.

Let $G(n+1,n)$ represent the space of all $n$-dimensional linear subspaces in $\mathbb{R}^{n+1}$ (the Grassmann Manifold). For an open subset $U$ of $\mathbb{R}^{n+1}$ or $\mathbb{H}^{n+1}$, we denote $G(n,U) = U \times G(n+1,n)$.
We say $V$ is an $n$-varifold on $U$ if $V$ is a Radon measure on $G(n,U)$.
The weight measure of $V$, denoted $\|V\|$, is defined by
\[
	\|V\|(K):=\int_{K} dV(X,S),
\]
for any Borel subset $K\subset U$.
For an $\mathcal{H}^n$-measurable, countably $n$-rectifiable set $M \subset U$, and a positive, locally $\mathcal{H}^n$-integrable function $\vartheta$ on $M$, there exists a natural $n$-varifold $|(M,\vartheta)|$ defined by
\[
	|(M,\vartheta)|(K):=\int_{(X,S) \in K, x \in M, S=T_XM} \vartheta(X)dV(X,S)
\]
where $T_XM$ is the approximate tangent space of $M$ at point $X$.
We say $V$ is rectifiable if it can be expressed as $V = |(M,\vartheta)|$ for the above-defined $M$ and $\vartheta$, and integral rectifiable if, in addition, $\vartheta$ is integer-valued $\mathcal{H}^n$-a.e.
Typically, $|M|$ denotes $|(M,1)|$.

Given an $n$-varifold $V$ on $U$, the tangent cone space of $V$ at a point $X \in \mathrm{spt}\|V\|$ is defined by
\[
	\mathrm{VarTan}(V,X):=\{ V' \in G(n,\mathbb{R}^{n+1}): V'=\lim_{i \rightarrow +\infty} (\eta_{X,\rho_i})_{\#}(V) \text{ for some sequence }\rho_i \rightarrow 0^+ \}.
\]
Here, $f_{\#}(V)$ denotes the pushforward of the varifold $V$ under a mapping $f$.
We say that $V$ has a tangent cone $V'$ at $X$ if $V' \in \mathrm{VarTan}(V,X)$.
Specifically, we say that $V$ has a multiplicity-one tangent hyperplane (half-hyperplane, respectively) at $X$ if $V'=|L(P)|$ for a certain $L \in \mathrm{SO}(n+1)$ and $P=\left\{ x_1=0 \right\}$ ($P=\left\{ x_1=0,x_2\ge 0 \right\}$, respectively).
The first variation $\delta V$ of $V$ is defined by
\[
	\delta V(\varphi):= \int_{ } \mathrm{div}_S \varphi(X) dV(X,S),\quad \forall \varphi \in \mathfrak{X}^1_c(U).
\]

\subsection{Main Results}%
\label{sub:main_results}
We consider a bounded closed subset $\Omega \subset \mathbb{R}^{n+1}$ with a $C^2$ boundary, indicating that $\partial \Omega$ is an $C^2$ hypersurface in $\mathbb{R}^{n+1}$.
Let $g$ be a $C^1$ metric on $\Omega$ and $\theta$ a $C^1$ function defined on $\partial \Omega$ with values in the range $(0,\pi)$.
We extend $g$ to ensure it remains a $C^1$ metric in a neighborhood of $\Omega$.
Note that the extensions of $g$ do not affect the concepts defined subsequently.

In the context of varifold theory under the metric $g$, we outline some fundamental results and direct the reader to Appendix \ref{sub:theory_of_varifolds_under_lipschitz_metrics}, as these are analogous to established varifold theory.

We say $V$ is an $n$-varifold defined on $\Omega$ if $V$ is an $n$-varifold in $\mathbb{R}^{n+1}$ with $\mathrm{spt}\|V\|\subset \Omega$.

Given an $n$-varifold defined on $\Omega$, the weight measure $\|V\|_g$ of $V$ under metric $g$ defined by
\[
	\|V\|_g(K):=\int_{ K} \sqrt{\mathrm{det}g_S(X)}dV(X,S),
\]
where $\mathrm{det}g_S(X):=\mathrm{det}(\left< \tau_i,\tau_j \right>_{g(X)} )_{i,j=1}^n$ and $\left\{ \tau_i \right\}_{i=1}^n$ forms an orthonormal basis of $S$ under the Euclidean metric.

For any $C^1$ vector field $\varphi$ on $\Omega$, the first variation $\delta^gV$ of $V$ under $g$ is defined by
\[
	\delta^gV(\varphi):=\int_{ } \mathrm{div}_S^g \varphi(X) \sqrt{\mathrm{det}g_S(X)} dV(X,S).
\]
Here, $\mathrm{div}_S^g \varphi(X)$ is the divergence of $\varphi$ under metric $g$ at $X$.

We say $V$ has \textit{bounded first variation} under the metric $g$ if there exists a constant $C$ such that
\begin{equation}
	\left| \delta^gV(\varphi) \right| \le C\sup _{}|\varphi|_g,\quad \text{ for any vector field } \varphi \text{ supported in the neighborhood of }\Omega.
	\label{eq:defBounded1stVar}
\end{equation}
With $V$ having bounded first variation under metric $g$, we can define the \textit{generalized mean curvature} $\boldsymbol{H}_g$, the \textit{generalized boundary measure} $\sigma_V^g$, and the \textit{generalized unit conormal} $\nu_V^g$ of $V$ under metric $g$ satisfying
\[
	\delta^g V(\varphi)= \int_{ } \left< \boldsymbol{H}_g,\varphi \right> _g d\|V\|_g + \int_{ } \left< \nu_V^g,\varphi \right> _g d\sigma_V^g,
\]
where $\nu_V^g$ satisfies $|\nu_V^g|_g=1$ for $\sigma_V^g$-a.e. $X \in \mathbb{R}^{n+1}$.

We now introduce the concept of a prescribed contact angle varifold on $\Omega$.

\begin{definition}
	\label{def_Prescribed}
	An $n$-varifold $V$ defined on $\Omega$ is said to have \textit{prescribed contact angle $\theta$ under metric $g$} if there exists an $\mathcal{H}^n$-measurable subset $U \subset \partial \Omega$ such that, for any $\varphi \in \mathfrak{X}^1(\mathbb{R}^{n+1})$ with compact support and tangent to $\partial \Omega$ at any point $X \in \partial \Omega$, the following holds,
			\begin{equation}
				\delta^gV(\varphi)-\delta^g |U|(\cos \theta \varphi)=\int_{ } \left< \boldsymbol{H}_g, \varphi \right> _g d\|V\|_g,
				\label{eq:defPrescribedAngle}
			\end{equation}
		where $\boldsymbol{H}_g$ denotes the generalized mean curvature vector of $V$ under metric $g$, and $\sigma_V^g$ is singular with respect to $\|V\|_g$.
\end{definition}

The following example demonstrates that if there exists a $C^2$ hypersurface with a prescribed contact angle $\theta$ under metric $g$ in the classical sense, then the corresponding varifold also has a prescribed contact angle $\theta$ under metric $g$.
\begin{example}
Consider $\Sigma$, a $C^2$ hypersurface with boundary in $\Omega$, and let $U \subset \partial \Omega$ be an open subset with a $C^2$ boundary such that $\partial U = \partial \Sigma$. 
We define the contact angle $\measuredangle_g(\Sigma,U)$ at any point $X \in \partial \Sigma$ in the traditional manner.
Specifically, let $\tau_{\Sigma}^g$ (respectively, $\tau_U^g$) denote the unit conormal vector of $\partial \Sigma$ in $\Sigma$ (respectively, $\partial U$ in $\partial \Omega$), pointing into $\Sigma$ (respectively, $U$) at $X$ under metric $g$.
The contact angle $\measuredangle_g(\Sigma,U)$ is defined as the angle between $\nu_{\Sigma}^g$ and $\nu_U^g$ under metric $g$.
If $\measuredangle_g(\Sigma,U) = \theta(X)$ for any $X \in \partial \Sigma$, then we can varify that the varifold $V = |\Sigma|$ has the prescribed contact angle $\theta$ under metric $g$.
\end{example}

\begin{theorem}
	\label{thm_mainG}
	Suppose $V$ is an integral $n$-rectifiable varifold with prescribed contact angle $\theta$ under metric $g$, having bounded first variation, and the mean curvature vector $\boldsymbol{H}_g$ is uniformly bounded.
	If at a point $X \in \mathrm{spt}\|V\| \cap \partial \Omega$, $V$ has a multiplicity-one tangent half-hyperplane at $X$.
	Then, $\mathrm{spt}\|V\|$ is a $C^{1,\gamma}$ hypersurface with boundary near $X$ for some $\gamma \in (0,1)$.
\end{theorem}

Now, we restate Corollary \ref{cor_contactAngleEuclidean} and Theorem \ref{thm_densityEuclidean} in a general metric context.

\begin{corollary}
	If $V$ is a varifold satisfying the conditions of Theorem \ref{sub:main_results} at a point $X \in \mathrm{spt}\|V\| \cap \Omega$, and $\Sigma$ denotes the support of $\|V\| \cap B_\rho(X)$ for $\rho$ such that $\Sigma$ is a $C^{1,\gamma}$ hypersurface, then we can find an open subset $U \subset \partial \Omega$ such that $\partial \Sigma = \partial U$ in $B_\rho(X)$ and the contact angle $\measuredangle _g(\Sigma,U)=\theta(Y)$ or $\frac{\pi}{2}$ for any $Y \in \partial \Sigma \cap B_\rho(X)$.

	Furthermore, if $\theta\neq \frac{\pi}{2}$ on $\mathrm{spt}\|V\|\cap \partial \Omega$ and the tangent cone of $\|V\|$ at $X$ is not orthogonal to $\partial \Omega$, then we actually have $\measuredangle _g(\Sigma,U)=\theta$.
	\label{cor_contactAngleG}
\end{corollary}

\begin{theorem}
	\label{thm_densityG}
	Given $\Lambda_0 \in (0,1)$, suppose $V$ is an integral $n$-rectifiable varifold described above with prescribed contact angle $\theta$ with $\theta \in [\Lambda_0,\pi-\Lambda_0]$ and suppose $\boldsymbol{H}_g$ is uniformly bounded.
	Then there exists a positive constant $\varepsilon>0$, dependent on $n$ and $\Lambda_0$, such that the following conclusions hold,
	\begin{enumerate}[\normalfont(a)]
		\item The dimension $n$ equals 2, or the range of $\theta$ is within $(\frac{\pi}{2}-\varepsilon,\frac{\pi}{2}+\varepsilon)$.
		\item At a point $X \in \mathrm{spt}\|V\|\cap \partial \Omega$, we have
	\[
		\liminf_{\rho \rightarrow 0^+} \frac{\|V\|(B_\rho(X))}{\omega_n\rho^n}\le \frac{1}{2}+\varepsilon.
	\]
	\end{enumerate}
	Then, the support of $\|V\|$ near $X$ is a $C^{1,\gamma}$ hypersurface with boundary for some $\gamma \in (0,1)$.
\end{theorem}

\section{$\mu$-stationary quadruple}%
\label{sec:_mu_stationary_quadruple}

\subsection{$\mu$-flat metric}%
\label{sub:_mu_flat_metric}

Upto a linear transformation and scaling, we can always assume $g$ is sufficiently close to the Euclidean metric in a neighborhood of $X$.

In this subsection, we introduce the concept of a $\mu$-flat metric and discuss some of its fundamental properties. 
We fix an open convex domain $U\subset \mathbb{R}^{n+1}$ or $\mathbb{H}^{n+1}$.

\begin{definition}
	\label{def_muFlat}
	A metric $g$ on $U$ is said \textit{$\mu$-flat} if for any $X \in U$,
	\[
		\left|g_{ij}(X)-\delta_{ij}\right|\le \mu,\quad \text{ and }|Dg(X)|\le \mu \text{ for any }X \in U.
	\]
	Here, $Dg$ denotes the derivative of $g$ under the Euclidean metric and $|Dg(X)|$ is defined as 
	\[
		|Dg(X)|=\sqrt{\sum_{i,j,k=1}^{n+1} \left( \frac{\partial g_{ij}}{\partial x_k}(X) \right)^2}.
	\]
	
\end{definition}

Utilizing the properties of Cholesky decomposition, we can establish the following proposition.
\begin{proposition}
	\label{prop_Pi2Delta}
	There exists $\mu_0=\mu_0(n)\in (0,1)$ sufficiently small such that for any $\mu$-flat metric with $\mu \in (0,\mu_0)$ and $X \in U$, we can find a unique upper triangular matrix $L^{g}_X \in \mathrm{GL}(n+1)$ such that
	\begin{equation}
		|L^g_{X} - \mathrm{Id}|\le C\mu,\quad 
		\left|(L_X^g)^{-1}-\mathrm{Id}\right|\le C\mu,
		\label{eq:propLcloseG}
	\end{equation}
	and $(L^g_X)_{*} g=\delta$ at point $L_X^g(X)$.
	The constant $C=C(n)$ is independent of $\mu$, and $(L^g_X)_*g$ denotes the pushforward of $g$ by $L^g_X$.

	Further more, the map $X\rightarrow L_X^g$ is a $C^{1}$ map.
\end{proposition}

\begin{remark}
	\label{rmk:LXg}
	Since $L_X^g$ is a non-singular upper triangular matrix, we have
	\[
		L_X^g(\mathrm{span}(e_k,e_{k+1},\cdots ,e_{n+1}))=\mathrm{span}(e_k,e_{k+1},\cdots ,e_{n+1}), \forall k=1,2,\cdots ,n+1,
	\]
	where $\mathrm{span}(e_k,e_{k+1},\cdots ,e_{n+1})$ denotes the subspace spanned by $\left\{ e_k,e_{k+1},\cdots ,e_{n+1} \right\}$.

	In particular, $L_X^g(\mathbb{H}^{n+1})=\mathbb{H}^{n+1}$.
\end{remark}

\begin{remark}
	Proposition \ref{prop_Pi2Delta} implies
	\begin{equation}
		B_{(1-C\mu)\rho}(X)\subset L_X^g(B_\rho(X))\subset B_{(1+C\mu)\rho}(X),
		\label{eq:ballCompLXg}
	\end{equation}
	for some constant $C=C(n)$ if $\mu$ is small enough.
\end{remark}

\begin{definition}
	\label{def_Gball}
	For any $X,Y \in U$, the \textit{distance function} $\mathrm{dist}_g(X,Y)$ between $X$ and $Y$ is defined as the infimum length of any $C^1$ curve joining $X$ and $Y$ in $U$.

	Accordingly, the geodesic ball $B^g_r(X)$ is defined as
	\[
		B^g_r(X):=\left\{ X' \in U:\mathrm{dist}^g(X,X')<r \right\}.
	\]
\end{definition}

When assuming $g$ is sufficiently close to the Euclidean metric, we expect that the geodesic ball $B^g_r(X)$ should be close to $B_r(X)$.
The detailed relation is given in the following proposition.

\begin{proposition}
	\label{prop_compBall}
	There exist $\mu_0=\mu_0(n,U)\in (0,1)$ such that for any $\mu$-flat metric $g$ with $\mu \in (0,\mu_0)$,
	\begin{equation}
		B_{(1-C\mu)\rho}(X)\subset B^g_\rho(X)\subset B_{(1+C\mu)\rho}(X).
		\label{eq:ballComp}
	\end{equation}
	and if $g(X)=\delta$, a more precise relation is
	\begin{align}
		B_{(1-C\mu \rho)\rho}(X)\subset B^g_\rho(X)\subset B_{(1+C\mu\rho)\rho}(X),
		\label{eq:ballCompFine}
	\end{align}
	where $C=C(n,U)$.
\end{proposition}
\begin{proof}
	Comparing the lengths of any $C^1$ curve $\gamma$ in $U$ under different metrics and integrating along $\gamma$ yields
	\[
		(1-C\mu)L^\delta(\gamma)\le L^g(\gamma)\le (1+C\mu)L^\delta(\gamma),
	\]
	where $L^g(\gamma)$ denotes the length of $\gamma$ under metric $g$.
	Then, we can obtain \eqref{eq:ballComp} by the definition of geodesic balls and convexity of $U$.

	Moreover, if $g(X)=\delta$, since $|Dg|(X)\le \mu$, we have $|g(X')-\delta|\le C \mu|X-X'|$.
	This will lead to \eqref{eq:ballCompFine}.
\end{proof}

\subsection{Varifolds with prescribed contact angle}%
\label{sub:VPA}

Note that Theorem \ref{thm_mainG} is purely a local result.
By using Fermi coordinates, this theorem can be reduced to the scenario where $X=0$, $g=\delta$ at $0$, and $\Omega\cap B_\rho$ is a half-ball $B_\rho \cap \left\{ x_1\ge 0 \right\}$ for some $\rho>0$.
Up to a scaling, we set $\rho=1$.
We denote by $U$ the $\mathcal{H}^n$-measurable subset of $\partial\Omega$ such that $V$ and $U$ satisfies the conditions specified in Definition \ref{def_Prescribed} and represent $W=|U\cap B_1|$.

Consequently, the quadruple $(V,W,\theta,g)$ fulfills the following conditions.
\begin{enumerate}[\normalfont(a)]
	\item $\theta$ is a $C^1$ function defined on $\mathcal{B}_1\cap \partial \mathbb{H}^{n+1}$ with values in $(0,\pi)$, $g$ is a $C^1$ metric defined on $\mathcal{B}_1\cap \mathbb{H}^{n+1}$.
		\label{it:thetaG}
	\item 
		$V$ is an integral $n$-rectifiable varifold on $\mathcal{B}_1$, and $W=|\tilde{U}|$ for an $\mathcal{H}^n$-measurable subset of $\mathcal{B}_1\cap \partial\mathbb{H}^{n+1}$. \label{it:varifold}
\item $V$ has bounded first variation under metric $g$.
		\label{it:1stBounded}
	
	\item The generalized mean curvature of $V$ under metric $g$ in $\mathcal{B}_1$, denoted by $\boldsymbol{H}_g$, satisfies for any $\varphi \in \mathfrak{X}^1_{c,\tan}(\mathcal{B}_1)$,
		\begin{equation}
			\delta^gV(\varphi)-\delta^gW(\cos \theta \varphi)=\int_{ } \left< \boldsymbol{H}_g,\varphi \right> _g d\|V\|_g.
			\label{eq:1stVarG}
		\end{equation}
		\label{it:PrescribedAngle}
	\item $\boldsymbol{H}_g$ satisfies $\sup_{X \in \mathrm{spt}\|V\|}|\boldsymbol{H}_g(X)|_g<\infty$.
		\label{it:HBounded}
	\item $V$ has a multiplicity-one tangent half-hyperplane at $0$.
		\label{it:TangentHalfPlane}
\end{enumerate}
	
\begin{definition}
	\label{def_VPA}
	We define a \textit{varifold with prescribed contact angle quadruple}, denoted by $(V,W,\theta,g)$, to be a quadruple satisfies \ref{it:thetaG}, \ref{it:varifold}, and \ref{it:PrescribedAngle} as outlined above.

	For convenience, we represent this quadruple as $\mathcal{V}$ and refer to it as a \textit{VPCA-quadruple} when it exemplifies a \underline{v}arifold with a \underline{p}rescribed \underline{c}ontact \underline{a}ngle quadruple.

	We say a sequence of VPCA-quadruples $\left\{ \mathcal{V}_i \right\}_{i=1}^\infty$ converges to a VPCA-quadruple $\mathcal{V}$ if $V_i\rightarrow V,W_i\rightarrow W$ in the sense of varifolds, $\theta_i \rightarrow \theta$, $g_i\rightarrow g$ in $C^1$ sense.
\end{definition}

From now on, we will assume $\theta$ is a $C^1$ function defined on $\mathcal{B}_1\cap \partial \mathbb{H}^{n+1}$ with values in $(0,\pi)$, and $g$ is a $C^1$ metric defined on $\mathcal{B}_1$.

Now, let us give the precise definitions of $\mu$-stationary VPCA-quadruples.

\begin{definition}
	\label{def_muStat}
	We say a VPCA-quadruple $\mathcal{V}$ is a \textit{$\mu$-stationary quadruple} if $\mathcal{V}$ satisfies the following conditions,
	\begin{enumerate}[\normalfont(a)]
		
	\item The generalized mean curvature $|\boldsymbol{H}_g|_g\le \mu$.
		\label{it:muMeanCur}
	\item The angle function $\theta$ satisfies $|D\theta|\le \mu$.
		\label{it:muAngle}
	\item The metric $g$ satisfies $|g-\delta|\le \mu,|Dg|\le \mu$.
		\label{it:muMetric}
\end{enumerate}
Here, $|D \theta|(X):=\sqrt{\sum_{i=2}^{n+1} \left( \frac{\partial \theta(X)}{\partial x_i} \right)^2}$ since $\theta$ is a function defined on $\mathcal{B}_1\cap \partial \mathbb{H}^{n+1}$.
In particular, we say $\mathcal{V}$ is \textit{stationary} if it is $0$-stationary.
We say $\mathcal{V}$ is \textit{integral rectifiable} if each $V,W$ is an integral rectifiable varifold.
Moreover, we say $\mathcal{V}$ is \textit{restricted} if $V$ and $W$ all have bounded first variation under metric $g$, and
$W=|U|$ for some $\mathcal{H}^n$-measurable set $U\subset \mathcal{B}_1\cap \partial \mathbb{H}^{n+1}$.
We denote the class of all \textit{restricted integral rectifiable $\mu$-stationary quadruples} by $\mathcal{RIV}(\mu)$.

Additionally, we say $\mathcal{V}$ is \textit{$(\Lambda_0,\mu)$-stationary} if $\theta$ also satisfies $\theta \in [\frac{\pi}{2}+\Lambda_0,\pi-\Lambda_0]$, and $\mathcal{V}$ is \textit{$\mu$-stationary} in free boundary sense if $\theta$ also satisfies $\theta \in [\frac{\pi}{2}-\mu,\frac{\pi}{2}+\mu]$.

Specifically, we say an $n$-varifold $V$ defined on $\mathcal{B}_1$ is \textit{$\mu$-stationary under metric $g$ in free boundary sense} if $(V,0,\frac{\pi}{2},g)$ is a $\mu$-stationary quadruple and $V$ is a \textit{stationary varifold in free boundary sense} if $(V,0,\frac{\pi}{2},\delta)$ is a stationary quadruple.
\end{definition}

\begin{remark}
	\label{rmk_stat}
	If $\mathcal{V}=(V,W,\theta,g)$ is stationary, then we directly know $\theta$ is a constant, $g=\delta$, and $V-\cos \theta W$ is a stationary varifold in free boundary sense.
\end{remark}

\begin{remark}
	\label{rmk_changeAngle}
	For any $\mathcal{V}=(V,W,\theta,g) \in \mathcal{RIV}(\mu)$, we can check that $(V,|\mathcal{B}_1\cap \mathbb{H}^{n+1}|-W,\frac{\pi}{2}-\theta,g) \in \mathcal{RIV}(\mu)$, according to the definition of $\mathcal{RIV}(\mu)$.
\end{remark}

\begin{remark}
	If $V$ is a stationary varifold in free boundary sense in $\mathcal{B}_1$, then we can do a reflection across $\left\{ x_1=0 \right\}$ to obtain a stationary varifold in $\mathcal{B}_1$.
	\label{rmk_reflectionFB}
\end{remark}

For simplicity, given a VPCA-quadruple $\mathcal{V}$, we introduce the following notations,
\[
	V^\theta := V-\cos \theta W
\]
to be the signed Radon measure on $\mathcal{B}_1 \times G(n+1,n)$ and define
\[
	\|V^\theta\|_g:=\|V\|_g(U)-\int_{ U} \cos \theta d\|W\|_g.
\]
to be the signed weight measure for $V^\theta$.

Given $\mu$ is sufficiently small, we can rewrite condition \eqref{eq:1stVarG} under Euclidean metric.
\begin{proposition}
	\label{prop_1stVar}
	There exists $\mu_0=\mu_0(n)$ such that the following holds.
	Suppose $\mathcal{V}$ is a $\mu$-stationary quadruple with $\mu \in (0,\mu_0)$ and $X_0 \in \mathcal{B}_1$, then for any vector field $\varphi \in \mathfrak{X}^1_{c,\tan}(\mathcal{B}_1)$, the following inequality holds,
	\begin{equation}
	\left|\int_{ } \mathrm{div}_S(\varphi)dV(X,S)-\int_{ }  \cos \theta \mathrm{div}_S(\varphi)dW(X,S)\right|	\le
	C\mu \int_{ } |\varphi|+|D\varphi|d\|V+W\|
	\label{eq:1stVarFormulaGeneral}
	\end{equation}
	Moreover, if $g(X_0)=\delta$, the above inequality can be improved by
	\begin{equation}
	\left|\int_{ } \mathrm{div}_S(\varphi)dV(X,S)-\int_{ }  \cos \theta \mathrm{div}_S(\varphi)dW(X,S)\right|	\le
	C\mu \int_{ } |\varphi|+|X-X_0||D\varphi|d\|V+W\|
	\label{eq:1stVarFormula}
	\end{equation}
\end{proposition}
Here, $|D\varphi|(X)=\sqrt{\sum_{i,j=1}^{n+1} \left( \frac{\partial \varphi^i}{\partial x_j}(X) \right)^2}$ with $\varphi$ expressed as $\sum_{i =1}^{n+1}\varphi^ie_i$.
\begin{proof}
	Note that $|g(X)-\delta|\le \mu$ from \ref{it:muMetric} in Definition \ref{def_muStat}, we can establish the following inequality,
	\begin{equation}
		\left|\mathrm{div}_S^g \varphi \sqrt{\mathrm{det} g_S}-\mathrm{div}_S \varphi\right|\le
		C \mu ( |\varphi|+|D\varphi|)
		\label{eq:pf1stVarCompG}
	\end{equation}
	for any $X \in \mathcal{B}_1$ and $n$-plane $S \in G(n+1,n)$, provided $\mu\in (0,\mu_0)$ for $\mu_0=\mu_0(n)$ small enough.
	Then, using \eqref{eq:pf1stVarCompG} and the condition \ref{it:muMeanCur} in Definition \ref{def_muStat}, we can find
	\[
		\left|\int_{ } \mathrm{div}_S\varphi dV(X,S)-\int_{ } \mathrm{div}_S(\cos \theta \varphi)dW(X,S)\right|\le
		C\mu \int_{ } |\varphi|+|\cos \theta \varphi|+|D\varphi|+|D(\cos \theta \varphi)|d\|V+W\|.
	\]
	Now, with condition \ref{it:muAngle}, we can verify that \eqref{prop_compBall} holds.
	
	To derive the second inequality, we note that $|Dg|\le \mu$ and $g(X_0)=\delta$ imply $|g(X)-\delta|\le \mu |X-X_0|$, and then we can apply the same argument as above.
\end{proof}

\begin{remark}
	\label{rmk_generalDomain}
	While our discussions are based on the domain $\mathcal{B}_1$, it is feasible to consider a general domain $U\subset \mathbb{H}^{n+1}$, where the constant $C$ may depend on the diameter of $U$.
\end{remark}

The key ingredient in proving Allard regularity is the application of the monotonicity formula.

In our context, we can establish the following monotonicity formula.

\begin{theorem}
	[Monotonicity Formula]
	There exists $\mu_0=\mu_0(n)$ such that the following holds.
	Suppose $\mathcal{V}$ is a $\mu$-stationary quadruple with $\mu \in (0,\mu_0)$ and we also assume $\mathcal{V}$ is restricted. 
	Then, the following monotonicity holds for any point $X_0 \in \mathcal{B}_1$,
	\begin{enumerate}[\normalfont(a)]
		
\item When $X_0 \in \partial \mathbb{H}^{n+1}\cap B_1$, and $g=\delta$ at $X_0$, we can establish following inequality for any $\sigma,\rho$ with $0<\sigma<\rho<1-|X_0|$,
	\begin{align}
		\frac{2\|V^\theta\|(\mathcal{B}_\sigma(X_0))}{\sigma^n}+2\int_{ \mathcal{B}_\rho(X_0)\backslash \mathcal{B}_\sigma(X_0)} \frac{|(X-X_0)^{\bot _S}|^2}{|X-X_0|^{n+2}}dV(X,S)&\nonumber \\
		\le (1+C\mu\rho)\frac{2\|V^\theta\|(\mathcal{B}_\rho(X_0))}{\rho^n}+{}&{}C\mu \rho,
		\label{eq:MonotonicityFormulaBoundary}
	\end{align}
	for some $C=C(n)$.
\item When $\theta \ge \frac{\pi}{2}-\mu$, then for any $\sigma,\rho$ with $0<\sigma<\rho<1-|X_0|$,  we have a weak version of \eqref{eq:MonotonicityFormulaBoundary} by
		\begin{align}
		&\frac{ \|V^\theta\|(\mathcal{B}_\sigma(X_0))+\|V^\theta\|(\mathcal{B}_\sigma(\tilde{X}_0))}{\sigma^n}\le (1+C\mu)\frac{\|V^\theta\|(\mathcal{B}_\rho(X_0))+\|V^\theta\|(\mathcal{B}_\rho(\tilde{X}_0))}{\rho^n}+C\mu,
		\label{eq:MonotonicityFormulaInteriorG}
	\end{align}
	for some $C=C(n)$.
	\end{enumerate}

	\label{thm_monotonicityFormula}
\end{theorem}

\begin{proof}
	For any $X_0 \in B_1$, we introduce the vector field $\varphi_{X_0}$ and $\varphi$ as follows,
	\begin{align}
		\varphi_{X_0}(X):={} & 
		\begin{cases}
		\left( \frac{1}{\max\left\{ |X-X_0|,\sigma \right\}^n}-\frac{1}{\rho^n} \right)(X-X_0), & \text{ if }X \in \mathcal{B}_\rho(X_0)\\
		0, & \text{ otherwise.}
		\end{cases}
		\label{eq:pfDefPhiX0}\\
		\varphi(X)={}&\varphi_{X_0}(X)+\varphi_{\tilde{X} _0}(X).
		\label{eq:defPhiTangential}
	\end{align}
	This definition ensures that $\varphi(X)$ is tangential to $\partial \mathbb{H}^{n+1}$ for any points $X \in \partial \mathbb{H}^{n+1}\cap B_1$.
	Let us denote by $h(\rho)$ the expression
	\[
		h(\rho)=\frac{1}{\rho^n}\left( \|V\|(\mathcal{B}_\rho(X_0))+\|V\|(\mathcal{B}_\rho(\tilde{X} )) \right)-\frac{2}{\rho^n}\int_{ \mathcal{B}_\rho(X_0)}\cos \theta d\|W\|.
	\]
	By applying inequality \eqref{eq:1stVarFormula}, alongside a similar argument as for \eqref{eq:pfMonoWithMu}, we obtain
	\begin{align}
		(1-C\mu \sigma) h(\sigma)-(1+C\mu\rho)h(\rho)+\int_{ \mathcal{B}_\rho(X_0)\backslash \mathcal{B}_\sigma(X_0)} \frac{|(X-X_0)^{\bot _S}|^2}{|X-X_0|^{n+2}}dV^\theta(X,S)\nonumber \\
		+\int_{ \mathcal{B}_\rho(\tilde{X}_0)\backslash \mathcal{B}_\sigma(\tilde{X}_0)} \frac{|(X-\tilde{X}_0)^{\bot _S}|^2}{|X-\tilde{X}_0|^{n+2}}dV^\theta(X,S)\le{} C\mu \int_{\sigma}^{\rho} h(\tau)d\tau+{}&{}C\mu \rho,
		\label{eq:pfMonCapillary}
	\end{align}
	for some $\mu_0$ small enough.
	Here, we used the fact $\|W\|(\mathcal{B}_\rho)\le \omega_n\rho^n$ to estimate the right hand side of \eqref{eq:1stVarFormula}.
	(Note that we can take $\Lambda=0$ in \eqref{eq:pfMonoWithMu} and here, we assume $\mu_0$ small enough instead of $\rho$.)

	We also need to estimate term $\mathrm{I}$ defined by
	\begin{align*}
		\mathrm{I}:={}
		-\int_{ \mathcal{B}_\rho(X_0)\backslash \mathcal{B}_\sigma(X_0)}&\frac{|(X-X_0)^{\bot _S}|^2\cos \theta}{|X-X_0|^{n+2}}  dW(X,S)\\
					  &-\int_{ \mathcal{B}_\rho(\tilde{X}_0)\backslash \mathcal{B}_\sigma(\tilde{X}_0)} \frac{|(X-\tilde{X}_0)^{\bot _S}|^2\cos \theta}{|X-\tilde{X}_0|^{n+2}}dW(X,S)
	\end{align*}
	since $-\cos \theta$ could be negative.
	For the case $X_0 \in \partial \mathbb{H}^{n+1}\cap \mathcal{B}_1$, we have $(X-X_0)^{\bot _S}=0$ and hence $\mathrm{I}=0$.
	Then, similar with the argument leading to \eqref{eq:pfMonoLocallyBoundedFirst}, we have
	\[
		h(\sigma)\le (1+C\mu \rho)h(\rho)+C\mu\rho,
	\]
	which, in turn, implies \eqref{eq:MonotonicityFormulaBoundary} by inequality \eqref{eq:pfMonCapillary}.

	For the second case, we use \eqref{eq:1stVarFormulaGeneral} instead of \eqref{eq:1stVarFormula} to deduce
	\[
		(1-C\mu) h(\sigma)-(1+C\mu)h(\rho)+\mathrm{I}\le C\mu \int_{\sigma}^{\rho} h(\tau)d\tau+C\mu.
	\]
	
	To estimate $\mathrm{I}$, we use the fact that $|(X-\tilde{X} _0)^{\bot _S}|=|(X-X_0)^{\bot _S}|=\mathrm{dist}(X_0,\partial \mathbb{H}^{n+1})=:l$ for any $X \in \mathrm{spt}\|W\|$ and $S=\left\{ x_1=0 \right\}$.
	Consequently, with $\theta \ge \frac{\pi}{2}-\mu$ implying $\cos \theta\le \mu$, we can obtain
	\begin{align*}
		{} & \int_{ \mathcal{B}_\rho(X_0)\backslash \mathcal{B}_\sigma(X_0)} \frac{|(X-X_0)^{\bot _S}|^2 \cos \theta}{|X-X_0|^{n+2}}dW(X,S) \\
		\le{} & \mu l^2\int_{ B^n_{\sqrt{\rho^2-l^2}}(0)} \frac{1}{(|X|^2+l^2)^{\frac{n+2}{2}}}d\mathcal{H}^n(X)\\
		={}& \mu l^2 \int_{0}^{\sqrt{\rho^2-l^2}} \frac{n \omega_n\tau ^{n-1}}{(\tau^2+l^2)^{\frac{n+2}{2}}}d\tau 
		= \mu \frac{\omega_n (\rho^2-l^2)^{\frac{n}{2}}}{\rho^{n}}\le \mu \omega_n,
	\end{align*}
	indicating $\mathrm{I}\ge -C\mu$.
	This analysis establishes
	\[
		(1-C\mu \sigma)h(\sigma)-(1+C\mu \rho)h(\rho)\le C\mu \int_{\sigma}^{\rho} h(\tau)d\tau+C\mu,
	\]
	which implies \eqref{eq:MonotonicityFormulaInteriorG}.
\end{proof}
\begin{remark} 
	The only place we use the restriction condition is to control the term $\mathrm{I}$.
	Notably, \eqref{eq:MonotonicityFormulaBoundary} remains valid with $\theta\ge 0$ in place of the restriction condition.
\end{remark}

We introduce the density ratio $\mathcal{I}_{\mathcal{V}}(X,\rho)$ defined by,
\[
	\mathcal{I}_{\mathcal{V}}(X,\rho):=
	\frac{2}{\omega_n\rho^n} \|V\|_g(B_\rho^g(X)) -\frac{2}{\omega_n\rho^n}\int_{ B_\rho^g(X)} 
	\cos \theta d\|W\|_g,
\]
for any $X \in \mathcal{B}_1\cap \partial \mathbb{H}^{n+1}$.
Then the density ratio satisfies the following monotonicity formula.

\begin{corollary}
	\label{cor_monoG}
	Suppose $\mathcal{V} \in \mathcal{RIV}(\mu)$ for some $\mu=\mu(n)$ small enough.
	Then, for any $X_0 \in \mathcal{B}_1\cap \partial \mathbb{H}^{n+1}$, and $0<2\sigma\le \rho <\mathrm{dist}_g(X_0,\partial B_1\cap \mathbb{H}^{n+1})$, we have
	\begin{equation}
		\mathcal{I}_{\mathcal{V}}(X_0,\sigma)\le (1+C\mu \rho)\mathcal{I}_{\mathcal{V}}(X_0,\rho)+C\mu \rho.
		\label{eq:MonotonicityFormulaG}
	\end{equation}
	for some $C=C(n)$.

	Consequently, the limit $\lim_{\rho \rightarrow 0^+} \mathcal{I}_{\mathcal{V}}(X_0,\rho)$ exists for any $X \in \mathcal{B}_1\cap \partial \mathbb{H}^{n+1}$.
\end{corollary}
\begin{proof}
	This result directly follows from Theorem \ref{thm_monotonicity_formula} by applying a linear transformation and together with \eqref{eq:ballCompFine}.
	(See the proof of Corollary \ref{cor_MonoFormula} for a similar proof.)
	Moreover, the existence of the limit readily follows from \eqref{eq:MonotonicityFormulaG}.
\end{proof}

From Corollary \ref{cor_monoG}, we can define the density of a $\mu$-stationary quadruple $\mathcal{V}$ at $X_0 \in \mathcal{B}_1\cap \partial \mathbb{H}^{n+1}$ by
\[
	\Theta(\mathcal{V},X_0):=\lim_{\rho \rightarrow 0^+} \mathcal{I}_{\mathcal{V}}(X_0,\rho).
\]

\begin{remark}
	\label{rmk_negativeDensity}
	It should be noted that $\Theta(\mathcal{V},X_0)$ could be a negative number.
	For instance, we consider setting $\theta$ as a constant in $(0,\frac{\pi}{2})$, and choosing $V=0$, $W=\partial \mathbb{H}^{n+1}\cap \mathcal{B}_1$, $g=\delta$. Then $\Theta(\mathcal{V},0)=-2\cos \theta<0$.
\end{remark}

We now establish the upper semi-continuity of density.

\begin{theorem}

	Consider a sequences $\left\{ \mathcal{V}_i=(V_i,W_i,\theta_i,g_i) \right\}\subset \mathcal{RIV}(\mu_i)$, each a $\mu_i$-stationary quadruples with $\mu_i \in (0,\mu_0)$ for some $\mu_0=\mu_0(n)$ small enough.
	We assume $\mathcal{V}_i \rightarrow \mathcal{V}$ for some VPCA-quadruple $\mathcal{V}=(V,W,\theta,g)$.
	Then $\mathcal{V}$ is $\mu$-stationary with $\mu=\liminf_{i\rightarrow +\infty} \mu_i$ and
\[
	\Theta(\mathcal{V},X)\ge \limsup_{i\rightarrow +\infty} \Theta(\mathcal{V}_i,X_i).
\]
where $\left\{ X_i \right\}\subset \partial \mathbb{H}^{n+1}\cap \mathcal{B}_1$ is a sequence of points such that $\lim_{i\rightarrow +\infty} X_i=X$ and $X \in \partial\mathbb{H}^{n+1}\cap \mathcal{B}_1$
	\label{thm_theta__v_w_theta_g_x_ge__theta__v_i_w_i_theta_i_g_i_x}
\end{theorem}
\begin{remark}
	Note that we do not know whether $\mathcal{V} \in \mathcal{RIV}(\mu)$ here.
\end{remark}

\begin{proof}
	[Proof of Theorem \ref{thm_theta__v_w_theta_g_x_ge__theta__v_i_w_i_theta_i_g_i_x}]
	First, it's straightforward to establish that $\mathcal{V}$ is a $\mu$-stationary quadruple.

	WLOG, we assume $X=0$ and $g_i=\delta$ at 0 (after a possible affine transformation).
	We introduce the following density ratio for $\mathcal{V}_i$,
	\begin{align*}
		h_i(\rho)={}&\frac{2}{\rho^n}
		 \|V_i\|(\mathcal{B}_\rho(0)) -\frac{2}{\rho^n}\int_{ \mathcal{B}_\rho(0)} \cos \theta_i d\|W_i\|,\\
		h(\rho)={}&\frac{2}{\rho^n}\|V\|(\mathcal{B}_\rho(0))-\frac{2}{\rho^n}\int_{ \mathcal{B}_\rho(0)}\cos \theta d\|W\|.
	\end{align*}
	Note that $\lim_{\rho\rightarrow 0^+} h(\rho)$ still exists in view of the proof of Theorem \ref{thm_monotonicityFormula}, despite the possibility that $W$ might not correspond to a Caccioppoli set.

	From this, it follows that
	\[
		\Theta(\mathcal{V}_i,0)=\lim_{\rho\rightarrow 0^+} \omega_n^{-1}h_i(\rho),\quad 
		\Theta(\mathcal{V},0)=\lim_{\rho\rightarrow 0^+} \omega_n^{-1}h(\rho).
	\]

	For any $\varepsilon>0$, we choose a continuous function $f$ supported on $\mathcal{B}_\rho$ and $f=1$ on $B_{\rho-\varepsilon}$ with $0\le f\le 1$ in $\mathcal{B}_\rho$. 
	Using $V_i \rightarrow V$ in the sense of varifolds, we obtain
	\begin{align*}
		\frac{\rho^n}{2}h(\rho)\ge{} & \int_{ } f d\|V\|-\int_{ } f \cos \theta d\|W\|-\int_{ \mathcal{B}_\rho\backslash B_{\rho-\varepsilon}} |\cos \theta| (1-f)d\|W\| \\
		\ge{} & \limsup_{i\rightarrow +\infty} \left( \int_{ } f d\|V_i\|-\int_{ } f \cos \theta_i d\|W_i\|-\|W_i\|(\mathcal{B}_\rho\backslash \mathcal{B}_{\rho-\varepsilon}) \right) \\
		\ge{}& \limsup_{i\rightarrow +\infty} \left( \|V_i\|(\mathcal{B}_{\rho-\varepsilon})-\int_{ \mathcal{B}_{\rho-\varepsilon}} \theta_i d\|W_i\| -2\|W_i\|(\mathcal{B}_{\rho}\backslash \mathcal{B}_{\rho-\varepsilon})\right) \\
		\ge{}& \limsup_{i\rightarrow +\infty} \frac{(\rho-\varepsilon)^n}{2}h_i(\rho-\varepsilon)-2\omega_n(\rho^n-(\rho-\varepsilon)^n).
	\end{align*}

	Recall that the monotonicity formula \eqref{eq:MonotonicityFormulaBoundary} implies
	\[
		h_i(\rho-\varepsilon)\ge
		\frac{ \omega_n\Theta(\mathcal{V}_i,0)-C\mu (\rho-\varepsilon) }{1+C\mu(\rho-\varepsilon)}.
	\]

	Define $\Theta_0=\limsup_{i\rightarrow +\infty} \Theta(\mathcal{V}_i,0)$.
	Then, we have
	\[
		h(\rho)\ge \frac{(\rho-\varepsilon)^n}{\rho^n}\left( \frac{\omega_n\Theta_0-C\mu(\rho-\varepsilon)}{1+C\mu(\rho-\varepsilon)} \right)-4\frac{\rho^n -(\rho-\varepsilon)^n}{\rho^n}\omega_n.
	\]

	As $\varepsilon \rightarrow 0^+$, it follows that
	\[
		h(\rho)\ge \frac{\omega_n\Theta_0-C\mu\rho}{1+C\mu \rho}.
	\]

	Proceeding with $\rho \rightarrow 0^+$ yields
	\[
		\Theta(\mathcal{V},0)\ge \Theta_0.
	\]

	This is what we want.
\end{proof}

To describe the concept of the tangent cone of a quadruple, we need to discuss the rescaling of a quadruple first.
Given $\mathcal{V} \in \mathcal{RIV}(\mu)$, for any $X_0 \in \partial \mathbb{H}^{n+1}\cap \mathcal{B}_1$ and $\rho \in (0,1-|X_0|)$, we define the rescaled quadruple $\mathcal{V}_{X,\rho}$ of $\mathcal{V}$ by
\[
	\mathcal{V}_{X,\rho}=(\Pi_{X_0,\rho}^g)_{\#}\mathcal{V}:=\left( (\Pi^g_{X_0,\rho})_{\#}V,(\Pi^g_{X_0,\rho})_{\#}W,\theta \circ \Pi^g_{X_0,\rho}, \frac{1}{\rho^2}(\Pi^g_{X_0,\rho})_*g \right).
\]
This definition presumes that $\mathcal{B}_1 \subset \Pi^g_{X_0,\rho}(\mathcal{B}_1)$.
However, given $\gamma \in (0,1)$, by selecting $\mu=\mu(n, \gamma)$ small enough, the rescaled quadruple $\mathcal{V}_{X_0,\rho}$ is guaranteed to be well-defined for any $\rho \in (0,\gamma(1-|X_0|))$.

\begin{remark}
	For each $R>1$, we can choose $\rho$ small enough to ensure $\mathcal{V}_{X,\rho}$ is well-defined on $\mathcal{B}_R\cap \partial \mathbb{H}^{n+1}$.
	Typically, we only require they are well-defined for $X \in \mathcal{B}_1\cap \partial \mathbb{H}^{n+1}$.
	\label{rmk_extendQuad}
\end{remark}

Based on the definition of $\mu$-stationary quadruple, the following rescaling property is straightforward.
\begin{proposition}
	
	Suppose $\mathcal{V} \in \mathcal{RIV}(\mu)$ for some $\mu=\mu(n)$ small enough.
	Then, for any $X_0 \in \partial \mathbb{H}^{n+1}\cap \mathcal{B}_1$ and $\rho \in (0,\frac{15}{16}(1-|X_0|))$, we have $\mathcal{V}_{X,\rho} \in \mathcal{RIV}(C\mu \rho)$ for some $C=C(n)$.
	\label{prop_rescaling}
\end{proposition}

\begin{definition}
	\label{def_tangent_cone}
	We denote the class of \textit{tangent cone} of a $\mu$-stationary quadruple $\mathcal{V}=(V,W,\theta,g)$ at $X_0 \in \partial \mathbb{H}^{n+1}\cap \mathcal{B}_1$ by 
	\[
		\mathrm{VarTan}(\mathcal{V},X_0):=\{ \mathcal{V}':\mathcal{V}'=\lim_{i\rightarrow +\infty}\mathcal{V}_{X_0,\rho_i} \text{ for some }\rho_i \rightarrow 0^+  \}.
	\]
\end{definition}

Before discussing the property of tangent cone, we present a compactness theorem as follows.

\begin{theorem}
	\label{thm_compactnessRIV}
	Suppose $\left\{ \mathcal{V}_i \right\}$ is a sequence of VPCA-quadruples with $\mathcal{V}_i \in \mathcal{RIV}(\mu_i)$ and $\mu_i \rightarrow 0^+$.
	Assuming that $\sup_{i}\|V_i\|(\mathcal{B}_1)<+\infty$, and $\theta_i \in [\Lambda_0,\pi-\Lambda_0]$ for some $\Lambda_0 \in (0,1)$,
	then, after extracting a subsequence, we have $\mathcal{V}_i \rightarrow \mathcal{V}=(V,W,\theta,\delta)$ for some stationary quadruple $\mathcal{V}$.
	Additionally, $V^\theta$ is a stationary rectifiable varifold in free boundary sense in $\mathcal{B}_1(0)$ and its density satisfies
	\begin{equation}
		\Theta(\mathcal{V},X)\ge \limsup_{i\rightarrow +\infty} \Theta(\mathcal{V}_i,X_i)
		\label{eq:thmUpperDensityStat}
	\end{equation}
	for any $X_i \rightarrow X$ where $\left\{ X_i \right\} \subset \partial \mathbb{H}^{n+1}\cap \mathcal{B}_1$ and $X \in \partial\mathbb{H}^{n+1}\cap \mathcal{B}_1$.

	Furthermore, we can establish
	\begin{equation}
		\Theta(\|V^\theta\|,X) \in \bigcup_{k=0}^{+\infty}\left\{ k,k-\cos \theta \right\}\quad \text{for $\mathcal{H}^n$-a.e. $X \in \mathcal{B}_1 \cap \partial \mathbb{H}^{n+1}$}.
		\label{eq:thmDensityRange}
	\end{equation}
\end{theorem}

For convenience, given a constant $\theta$, we abbreviate $Q_\theta=\cup _{k=0}^{+\infty}\left\{ k,k-\cos \theta \right\}$.
Prior to proving Theorem \ref{thm_compactnessRIV}, we need to establish the following lemma to estimate the lower bounds of density.
\begin{lemma}
		\label{lem_lowerBoundDensity}
		For any VPCA-quadruple $\mathcal{V}=(V,W,\theta,g) \in \mathcal{RIV}(\mu)$, it holds that
		\begin{align}
			\frac{1}{2}\Theta(\mathcal{V},X) \in{}& Q_{\theta(X)} \quad \text{for $\mathcal{H}^n$-a.e. $X \in \mathcal{B}_1$},\label{eq:lemRangeDensity}\\
			\frac{1}{2}\Theta(\mathcal{V},X)\ge{}& \min \left\{ 1-\cos \theta(X),1 \right\} \quad \text{for $\|V\|$-a.e. $X \in \mathcal{B}_1$}\cap \partial \mathbb{H}^{n+1}.
			\label{eq:lemDensityLower}
		\end{align}
		If $\theta(X)> \frac{\pi}{2}$ for every $X \in \mathcal{B}_1\cap \partial \mathbb{H}^{n+1}$, then we also have
		\begin{equation}
			\frac{1}{2}\Theta(\mathcal{V},X)\ge \cos(X) \quad \text{for $\|V+W\|$-a.e. $X \in \mathcal{B}_1\cap \partial \mathbb{H}^{n+1}$}.
			\label{eq:lemDensityLowerW}
		\end{equation}
\end{lemma}
\begin{proof}
	The validity of \eqref{eq:lemRangeDensity} is trivial since $\Theta(\|V\|,X)$ is a integer for $\mathcal{H}^n$-a.e. $X \in \mathcal{B}_1$ and limit
	\[
		\lim_{\rho\rightarrow 0^+} \frac{1}{\omega_n \rho^n}\int_{ \mathcal{B}_\rho(X)} \cos \theta d\|W\|(X)=\cos \theta(X)
	\]
	holds for $\mathcal{H}^n$-a.e. $X \in \mathcal{B}_1\cap \partial\mathbb{H}^{n+1}$.
	
		Since $W=|U|$ for some $\mathcal{H}^n$-measurable set $U$ in $\mathcal{B}_1\cap \partial \mathbb{H}^{n+1}$, it follows that
		\[
			\limsup_{\rho\rightarrow 0^+} \frac{1}{\omega_n \rho^n}\int_{ \mathcal{B}_\rho(X)} \cos \theta d\|W\|
			=\limsup_{\rho\rightarrow 0^+}\frac{\cos \theta(X) }{\omega_n \rho^n}\|W\|(\mathcal{B}_\rho(X))\le \max \left\{ 0, \cos \theta(X) \right\},
		\]
		for every $X \in \partial \mathbb{H}^{n+1}\cap \mathcal{B}_1$.
		Together with the fact $\Theta(\|V\|,X)\ge 1$ for $\|V\|$-a.e. $X \in \mathcal{B}_1$ and the definition of $\Theta(\mathcal{V},X)$, we obtain \eqref{eq:lemDensityLower}.
		
		In particular, if $\theta(X)>\frac{\pi}{2}$ for every $X \in \mathcal{B}_1\cap \partial \mathbb{H}^{n+1}$, then \eqref{eq:lemDensityLower} directly implies $\Theta(\mathcal{V},X)\ge 2$ for $\|V\|$-a.e. $X \in \mathcal{B}_1\cap \partial \mathbb{H}^{n+1}$.
		Additionally,
		\[
			\lim_{\rho\rightarrow 0^+} \frac{1}{\omega_n\rho^n}\int_{ B_\rho(X)} \cos \theta d\|W\|
			= \cos(X) \Theta(\|W\|,X)=1\quad \text{ for $\|W\|$-a.e., $X \in \mathcal{B}_1\cap \partial \mathbb{H}^{n+1}$}. 
		\]
		Together with the definition of $\Theta(\mathcal{V},X)$, we have \eqref{eq:lemDensityLowerW}.
\end{proof}
\begin{proof}[Proof of Theorem \ref{thm_compactnessRIV}]
	The existence of $\mathcal{V}$ is deduced from the compactness theorem for varifolds.
	Consequently, Theorem \ref{thm_theta__v_w_theta_g_x_ge__theta__v_i_w_i_theta_i_g_i_x} implies $\mathcal{V}$ is stationary and the upper semi-continuity of density \eqref{eq:thmUpperDensityStat}.

	Now, we employ Lemma \ref{lem_lowerBoundDensity} to show the rectifiability of $V^\theta$.
	
	For the case $\theta>\frac{\pi}{2}$, we can see that $\Theta(\|V^\theta\|,X)>0$ for $\|V+W\|$-a.e. $X \in \mathcal{B}_1$ by \eqref{eq:thmUpperDensityStat} and Lemma \ref{lem_lowerBoundDensity}.
	Hence, $V^\theta$ is rectifiable by the Rectifiability Theorem (i.e., Theorem \ref{thm_rectifiability}) together with a reflection argument.

	For $\theta <\frac{\pi}{2}$, Remark \ref{rmk_changeAngle} provides a method to adjust the angle to $\pi-\theta$, after which the preceding argument is applied.

	For $\theta=\frac{\pi}{2}$, we know $V$ is a stationary varifold in free boundary sense in $B_1(0)\cap \mathbb{H}^{n+1}$.
	Then Lemma \ref{lem_lowerBoundDensity} and \eqref{eq:thmUpperDensityStat} imply $\Theta(\|V\|,X)>0$ for $\|V\|$-a.e. $X \in \mathcal{B}_1$.
	Again, $V$ is rectifiable by the Rectifiability Theorem.

	Lastly, we aim to establish \eqref{eq:thmDensityRange}.
	This effort primarily concerns cases where $\theta\ge \frac{\pi}{2}$ by Remark \ref{rmk_changeAngle}.

	This part is an adaptation of the argument used in proving the integral compactness theorem in \cite[Theorem 42.7]{simon1983lectures}.
	Thus, we only provide an outline of the proof.

	Given that $V^\theta$ is rectifiable, we know tangent space of $V^\theta$ exists for $\mathcal{H}^n$-a.e. $X \in \mathrm{spt}\|V^\theta\|\cap \mathcal{B}_1\cap \partial \mathbb{H}^{n+1}$.
	Moreover, the support of such tangent space has to be $P_0$, which we define as $P_0:=\left\{ x_1=0 \right\}$.
	Let $X$ be such a point.
	This implies that $(\eta_{X,\rho})_{\#}V^\theta \rightarrow \Theta(\|V^\theta\|,X)|P_0|$ as $\rho \rightarrow 0^+$ in the sense of varifolds.

	Since $\mathcal{V}_i \rightarrow \mathcal{V}$, we can find a sequence $\rho_i\rightarrow 0^+$ such that
	\begin{equation}
		V_i'-\cos \theta_i'W_i'\rightarrow \Theta(\|V^\theta\|,X)|P_0 \cap \mathcal{B}_1|
		\label{eq:pfCovergeToPlane}
	\end{equation}
	as Radon measures.
	Here, $\mathcal{V}'_i=(V'_i,W_i',\theta_i',g_i')=\mathcal{V}_{X,\rho_i}$.
	We denote $\overline{V}_i:=V_i'-\cos \theta_i'W_i'$.
	If we do an orthogonal projection from $B_1(0)$ to $P_0$, we obtain
	\begin{equation}
		(P_0)_{\#}(\overline{V}_i \lfloor(\mathcal{B}_1\cap \left\{ x_1<\varepsilon \right\}))\rightarrow \Theta(\|V^\theta\|,X)|P_0\cap B_1|
		\label{eq:pfProjConvergence}
	\end{equation}
	for any $\varepsilon>0$.
	Here, we regard $P_0$ as an orthogonal projection from $\mathbb{R}^{n+1}$ to $P_0$.
	To proceed with our argument, we need the following lemma.
	\begin{lemma}
		\label{lem_closeToPlane}
		Given $\delta>0$ and a positive integer $N$, there exists an $\varepsilon=\varepsilon(\delta,N,n) \in (0,\delta)$ such that following statement holds.

		For each $i$, we can find a subset $A_i \subset \mathcal{B}_{\frac{2}{3}}$ such that $\overline{V}_i \lfloor(\mathcal{B}_{\frac{2}{3}}\backslash A_i)\rightarrow 0^+$ and for each $Z \in \mathcal{B}_{\frac{1}{3}}\cap \partial \mathbb{H}^{n+1}$, consider $\mathcal{X}$ as a subset of $P_0^{-1}(Z)\cap \left\{ x_1<\varepsilon \right\}\cap \mathrm{spt}\|\overline{V}_i\|\cap A_i$ with cardinality $\# \mathcal{X}\le N$.
		Then,
		\[
			\sum_{X \in \mathcal{X}}\Theta(\|\overline{V}_i\|,X)\le
			(1+\delta) \frac{\|\overline{V}_i\|(B_{\frac{1}{3}}(Z))}{\omega_n (\frac{1}{3})^n}+\delta,
		\]
		for $i$ large enough.
	\end{lemma}
	\begin{proof}
		We use a similar argument as in \cite[Lemma 42.9]{simon1983lectures}.
		Hence, we only give a sketch of the proof here.
		Define $A_i\subset B_{\frac{2}{3}}\cap \left\{ x_1<\varepsilon \right\}$ be the set for which
		\begin{equation}
			\int_{ B_\rho(X)} \|P_0-S\|d\overline{V}_i(X,S)\le \varepsilon \rho^n\quad \forall \rho \in (0,\frac{1}{3})
			\label{eq:pfTiltExcess}
		\end{equation}
		where the $\varepsilon>0$ will be chosen later.
		Then, the Besicovitch Covering Theorem and monotonicity formula yield
		\[
			\|\overline{V}_i\|(\mathcal{B}_{\frac{2}{3}}\backslash A_i)\le \frac{C}{\varepsilon} \int_{ B_1} \|P_0-S\|d\overline{V}_i(X,S).
		\]
		Further, \eqref{eq:pfCovergeToPlane} implies $\|\mathcal{V}_i\|(\mathcal{B}_{\frac{2}{3}}\backslash A_i)\rightarrow 0^+$ as $i\rightarrow +\infty$,
		taking into account that $\left|\cos \theta_i \right| \rightarrow \cos \theta$ in $C^0$ sense.

		The next key ingredient is the following two inequalities for any $X_0 \in \mathcal{B}_{\frac{1}{3}}\cap \left\{ x_1<\varepsilon \right\}\cap  A_i$ with $X_0 \in \partial \mathbb{H}^{n+1}$ or $2\tau\ge \mathrm{dist}(X_0,\partial \mathbb{H}^{n+1})$,
		\begin{align}
			\Theta(\|\overline{V}_i\|,X_0)\le{} & (1+C\mu_i) \frac{\|\overline{V}_i\|(U^{2\tau}_\rho(X_0))}{\omega_n \rho^n}+ \frac{C \varepsilon \rho}{\tau}+C\mu_i,
			\quad \forall 0<\rho\le \frac{1}{3}
			\label{eq:pfMonCylinderDensity}\\
			\frac{\|\overline{V}_i\|(U^{\tau}_{\sigma}(X_0))}{\omega_n \sigma^n}\le{} & (1+C\mu_i)
			\frac{\|\overline{V}_i\|(U^{2\tau}_\rho(X_0))}{\omega_n \rho^n}+\frac{C\varepsilon \rho}{\tau}+C\mu_i,\quad \forall 0<\sigma<\rho\le \frac{1}{3},
			\label{eq:pfMonCylinder}
		\end{align}
		for $i$ large enough.
		Here, $U^{\tau}_\rho(X_0)=\mathcal{B}_\rho(X_0)\cap \left\{ X:|(X-X_0)\cdot e_1|<\tau \right\}$.

		Let $f_{X_0}(X):=\eta(|(X-X_0)\cdot e_1|)$ and $\eta(r)$ a $C^1$ function such that $\eta(r)=1$ when $r<\tau$, $\eta(r)=0$ when $r>2\tau$, and $|\eta'|\le \frac{2}{\tau}$.
		To establish the above inequalities, we choose $\varphi_{X_0}$ to be the vector field defined in \eqref{eq:pfDefPhiX0} and set $\varphi=\varphi_{X_0}f_{X_0}+\varphi_{\tilde{X} _0}f_{\tilde{X} _0}$.
		Applying \eqref{eq:1stVarFormulaGeneral} with $\mathcal{V}_i$ in place of $\mathcal{V}$, and after standard computations, we find
		\begin{align*}
			&\int_{ \mathcal{B}_\sigma(X_0)} f d\|\overline{V}_i\|+\int_{ \mathcal{B}_{\sigma}(\tilde{X} _0)} f d\|\overline{V}_i\|\\
			\le{}& (1+C\mu_i)\left(  \int_{ \mathcal{B}_{\rho}(X_0)} f d\|\overline{V}_i\|+ \int_{ \mathcal{B}_{\rho}(\tilde{X} _0)} f d\|\overline{V}_i\|\right)+C\frac{\varepsilon \rho}{\tau}+C\mu_i.
		\end{align*}
		for $i$ large enough.
		Here, we have used the fact that $|D_Sf|\le \frac{C\|P_0-S\|}{\tau}$ and \eqref{eq:pfTiltExcess}.
		Note that
		\[
			\int_{ \mathcal{B}_{\rho}(\tilde{X} _0)}fd\|\overline{V}_i\|=
			\begin{cases}
			0, & \mathrm{dist}(X_0,\partial \mathbb{H}^{n+1})\le 2\tau,\\
			\displaystyle\int_{ B_\rho(X_0)}f d\|\overline{V}_i\| , & X_0 \in \partial \mathbb{H}^{n+1}.
			\end{cases}
		\]
		Hence, \eqref{eq:pfMonCylinder} follows by the definition of $f$ and $U^{\tau}_\rho$, and \eqref{eq:pfMonCylinderDensity} follows by taking $\sigma \rightarrow 0^+$.

		Now, similar to the argument in \cite[Lemma 42.9]{simon1983lectures}, an inductive argument allows us to obtain,
		\[
			\sum_{X \in \mathcal{X} }^{}\Theta(\|\overline{V}_i\|,X)\le (1+C\mu) \frac{\|\overline{V}_i\|(\mathcal{B}_{\frac{1}{3}}(Z))}{\omega_n (\frac{1}{3})^n}+C\varepsilon+C\mu,
		\]
		where the constant $C=C(n,N)$ is independent of $\varepsilon$.
		The only difference is, we need to ensure the center of $U^\tau_{\rho}(X)$, which is $X$, should satisfy $X \in \partial \mathbb{H}^{n+1}$ or $\mathrm{dist}(X,\partial \mathbb{H}^{n+1})\ge \tau$ in the proof.
		We select $\varepsilon$ small enough such that $C\mu+C\varepsilon<\delta$ to conclude the proof.
		\end{proof}
	
	Returning to the proof of Theorem \ref{thm_compactnessRIV}.
	Arguing by contradiction, we assume $\Theta(\|V^\theta\|,X) \notin Q_\theta$.

	We define the following three constants by
	\begin{align*}
		\overline{\theta}_1={} & \max \left\{ a  \in Q_\theta: a\le \Theta(\|V^\theta\|,X) \right\},\\
		\overline{\theta}_2={} & \min \left\{ a \in Q_\theta:a>\Theta(\|V^\theta\|,X) \right\},\\
		N={}& \min \left\{ a \in \mathbb{Z}:a>\overline{\theta}_2+2 \right\}.
	\end{align*}
	It is clear that $\overline{\theta}_1<\Theta(\|V^\theta\|,X)$ by our assumption.
	
	Based on Lemma \ref{lem_closeToPlane} and \eqref{eq:pfCovergeToPlane}, we can find a constant $\varepsilon$ small enough such that the following holds.

	For each $i$, there exists a subset $A_i \subset \mathcal{B}_{\frac{2}{3}}$ such that $\|\overline{V}_i\| \lfloor(\mathcal{B}_{\frac{2}{3}}\backslash A_i)\rightarrow 0^+$ and for each $Z \in \mathcal{B}_{\frac{1}{3}}\cap \partial \mathbb{H}^{n+1}$, let $\mathcal{X}$ be a subset of
	\begin{equation}
		\mathcal{X}_{Z}^i:=P_0^{-1}(Z)\cap \left\{ x_1<\varepsilon \right\}\cap \mathrm{spt}\|\overline{V}_i\|\cap A_i
		\label{eq:pfDefXi}
	\end{equation}
	with $\# \mathcal{X}\le N$.
	Then, we find
	\begin{equation}
		\sum_{X \in \mathcal{X}}\Theta(\|\overline{V}_i\|,X)<
		\overline{\theta}_2,
		\label{eq:pfSumDensity}
	\end{equation}
	for $i$ large enough.
	Notably, for $\mathcal{H}^n$-a.e. $Z \in \mathcal{B}_{\frac{1}{3}}\cap \partial \mathbb{H}^{n+1}$ and $X \in \mathcal{X}_Z$, it holds that
	\begin{equation}
		\Theta(\|\overline{V}_i\|,X)\ge
		\begin{cases}
		-\cos \theta_i'(X), & X \in \partial \mathbb{H}^{n+1}\cap \mathrm{spt}\|\overline{V}_i\|,\\
		1, & X \in \mathrm{spt}\|\overline{V}_i\|\backslash \partial \mathbb{H}^{n+1},
		\end{cases}
		\label{eq:pfLowerBoundDensity}
	\end{equation}
	by Lemma \ref{lem_lowerBoundDensity}.
	Together with \eqref{eq:pfSumDensity} and $\overline{\theta}_2<N-2$, we have $\# \mathcal{X}\le N-1$.
	Therefore, we have $\# \mathcal{X}_{Z}^i\le N-1$ for $\mathcal{H}^n$-a.e. $Z \in \mathcal{B}_{\frac{1}{3}}\cap \partial \mathbb{H}^{n+1}$ for $i$ large enough.

	Define $\psi_i$ as
	\begin{equation}
		\psi_i(Z)=\sum_{X \in \mathcal{X}_Z^i}\Theta(\|\overline{V}_i\|,X).
		\label{eq:pfDefPsi_i}
	\end{equation}
	Since $\Theta(\|\overline{V}_i\|,X) \in Q_{\theta_i(X)}$ for $\|\overline{V}_i\|$-a.e. $X \in \mathcal{B}_1$, along with \eqref{eq:pfSumDensity} and \eqref{eq:pfLowerBoundDensity}, then, for any fixed $\varepsilon'>0$, we have
	\begin{equation}
		\psi_{i}(Z)\le \overline{\theta}_1+\varepsilon', \text{ for }\mathcal{H}^n \text{-a.e. }Z \in \mathcal{B}_{\frac{1}{3}}\cap \partial \mathbb{H}^{n+1},
		\label{eq:pfPsiUpper}
	\end{equation}
	for $i$ large enough.
	Here, we also use the fact $\theta_i' \rightarrow \theta$.

	On the other hand, by \eqref{eq:pfProjConvergence} and $\|\overline{V}_i\|(\mathcal{B}_{\frac{2}{3}}\backslash A_i)\rightarrow 0$, we obtain
	\begin{equation}
		\int_{ } f \psi_i d \mathcal{H}^n \rightarrow \Theta(\|V^\theta\|,X) \int_{ } f d \mathcal{H}^n, \forall f \in C_c^0(B_{\frac{2}{3}}^{n}(0)).
		\label{eq:pfIntPsiConvergence}
	\end{equation}
	resulting a contradiction with \eqref{eq:pfPsiUpper} if we choose $\varepsilon'<\Theta(\|V^\theta\|,X)-\overline{\theta}_1$.
	Therefore, this establishes \eqref{eq:thmDensityRange}
\end{proof}

\begin{proposition}
	\label{prop_tangentCone}
	Suppose $\mathcal{V} \in \mathcal{RIV}(\mu)$ for some $\mu=\mu(n)$ small enough and $X_0 \in \mathcal{B}_1 \cap \partial \mathbb{H}^{n+1}$.
	For any $\mathcal{V}'=(V',W',\theta',\delta) \in \mathrm{VarTan}(\mathcal{V},X_0)$, we know $\mathcal{V}'$ is a stationary quadruple.
	Furthermore, we can obtain $V'-\cos \theta'W'$ is a rectifiable stationary cone in free boundary sense in $\mathcal{B}_1\cap \mathbb{H}^{n+1}$ together with following density estimate
	\begin{equation}
		\frac{1}{\omega_n}\|V'\|(\mathcal{B}_1)-\frac{\cos \theta'}{\omega_n}\|W'\|(\mathcal{B}_1)=2\Theta(\|V'-\cos \theta' W'\|,X)= \Theta(\mathcal{V},X_0)
		\label{eq:propConeDensity}
	\end{equation}
\end{proposition}
Here, we say a varifold $V$ defined on $\mathcal{B}_1$ is a cone if $(\eta_{0,\rho})_{\#}V\lfloor(\mathcal{B}_1)=V$ for any $\rho \in (0,1)$.


\begin{proof}
	Suppose $\mathcal{V}_i=(V_i,W_i,\theta_i,g_i)=\mathcal{V}_{X_0,\rho_i}$ such that $\mathcal{V}'=\lim_{i\rightarrow +\infty} \mathcal{V}_i$ with $\rho_i\rightarrow 0^+$.
	From Theorem \ref{thm_compactnessRIV}, it is clear that $\mathcal{V}'$ is a stationary quadruple and $V'-\cos \theta' W'$ is rectifiable.

	For the density estimate, we have $2\Theta(\|V'-\cos \theta' W'\|,0)\ge \Theta(\mathcal{V},X_0)$ easily.
	On the other hand, lower semi-continuity of mass implies
	$\|V'\|(\mathcal{B}_1)\le \liminf_{i\rightarrow +\infty} \|V_i\|(\mathcal{B}_1)$.
	But note that for $n$-varifolds corresponding to $\mathcal{H}^n$-measurable sets, $W_i \rightarrow W'$ in the sense of varifolds implies
	\[
		\lim_{i\rightarrow +\infty} \int_{ \mathcal{B}_1} \cos \theta_i d\|W_i\|=\int_{\mathcal{B}_1 } \cos \theta' d\|W'\|.
	\]
	Therefore, for any $\varepsilon$ small enough, we can show that
	\[
		\|V'\|(\mathcal{B}_1)- \int_{ \mathcal{B}_1} \cos \theta' d\|W'\|\le \Theta(\mathcal{V},X_0)+\varepsilon.
	\]
	By choosing $\varepsilon \rightarrow 0^+$, we obtain \eqref{eq:propConeDensity}.
	Furthermore, it implies $V'-\cos \theta' W'$ is a cone from the standard result in geometric measure theory.
\end{proof}
\begin{remark}
	\label{rmk_maxDensity}
	For any $\mathcal{V}' \in \mathrm{VarTan}(\mathcal{V},X_0)$, it holds that
	\[
		\Theta(\mathcal{V}',X)\le \Theta(\mathcal{V}',0)
	\]
	for any $X \in \mathcal{B}_1\cap \partial \mathbb{H}^{n+1}$ by the property of tangent cone and the upper semi-continuity of density.
\end{remark}

\begin{remark}
	In general, we do not know whether $\mathrm{VarTan}(\mathcal{V},X_0) \subset \mathcal{RIV}(0)$.
\end{remark}

\section{Estimation of $L^2$ excess}%
\label{sec:l2_estimate}

From this section to Section \ref{sec:iterations}, we fix a constant $\Lambda_0 \in (0,\frac{1}{2})$ and concentrate on the regularity theory for $\theta \in [\frac{\pi}{2}+\Lambda_0,\pi-\Lambda_0]$ (See Theorem \ref{thm_l2positiveAngle}).

\begin{theorem}
	\label{thm:smallDensity}
	Suppose $\mathcal{V}=(V,W,\theta,g) \in \mathcal{RIV}(\mu)$ with $\mu$ small enough.
	If $\Theta(\mathcal{V},X)<1-\cos \theta(X)$ for every $X \in \mathcal{B}_1\cap \partial \mathbb{H}^{n+1}$, then $V$ is a $\mu$-stationary varifold under metric $g$ in free boundary sense and for any connected component $U_0$ of $\mathcal{B}_1 \cap \partial \mathbb{H}^{n+1}\backslash \{ \theta \neq \frac{\pi}{2} \}$, we have $W\lfloor(U_0)=m|U_0|$ for $m=0$ or 1, and $\mathrm{spt}\|V\|\cap \partial\mathbb{H}^{n+1}\cap \mathcal{B}_1\cap \{ \theta=\frac{\pi}{2} \}=\emptyset$.
\end{theorem}
Before the proof of this theorem, we need to establish the following lemma to estimate the lower bound of the density for stationary varifolds in free boundary sense by dimension reduction argument.
\begin{lemma}
	\label{lem_dimReduction}
	Given a real number $\vartheta \in (0,1), \Lambda \in (0,+\infty)$, suppose $V$ is an $k$-rectifiable varifolds in $\mathbb{H}^{k+1}$ satisfies following conditions,
	\begin{enumerate}[\normalfont(a)]
		\item $V$ is a cone. \label{it:Vcone}
		\item The mass $\|V\|(B^k_1)\le \Lambda$.
			\label{it:VvolumeBound}
		\item $V+\vartheta |\{ (x_1,\cdots ,x_{k+1}) \in \mathbb{H}^{k+1}:x_1=0,x_2\le 0 \}|$ is stationary in free boundary sense.
			\label{it:statFB}
		\item $V$ has integer multiplicity $\|V\|$-a.e. on $\mathbb{H}^{k+1}\backslash \left\{ x_1=0,x_2\le 0 \right\}$. \label{it:VdensityInteger}
	\end{enumerate}
	Then, the density satisfies $\Theta^k(\|V\|,0)\ge \frac{1}{2}$.

Here, a varifold $V$ is said stationary in free boundary sense in $U\subset \mathbb{H}^{k+1}$ if for any $\varphi \in \mathfrak{X}^1_{c,\tan}(U)$, we have $\delta V(\varphi)=0$.
\end{lemma}
\begin{remark}
	 The density $\Theta^k(\|V\|,X)$ exists for every $X \in \mathbb{H}^{k+1}$ by \ref{it:statFB}.
\end{remark}

\begin{proof}
	[Proof of Lemma \ref{lem_dimReduction}]
	For simplicity, we write $W_k=|\{ (x_1,\cdots ,x_{k+1}) \in \mathbb{H}^{k+1}:x_1=0,x_2\le 0 \}|$.
	The proof proceeds by induction on $k$.
	For $k=1$, $V$ can be written as a sum of weighted rays,
	\[
		V=\sum_{i =0}^{m} a_i|\left\{ (r \sin\theta_i, -r\cos\theta_i):r\ge 0 \right\}|
	\]
	with $m \in \mathbb{N}$, $a_0\ge 0$, $\theta_0=0$, $a_i \in \mathbb{N}^+$, and $\theta_i \in (0,\pi]$ for $1\le i\le m$.
	Given that $V+\vartheta |\left\{ x_1=0,x_2\le 0 \right\}|$ is stationary in free boundary sense, it implies $m$ can not be 0.
	Hence, $\Theta^1(\|V\|,0)\ge \frac{1}{2}$ is clear.

	For the induction step, we assume the lemma holds for $k-1$.
	We choose a point $Y=(0,0,\cdots ,0,1) \in \partial\mathbb{H}^{k+1}$ and consider the tangent cone $V_1$ of $V$ at $Y$.
	By \ref{it:statFB}, we know $V_1$ exists and $V_1+\vartheta|W_k|$ is a stationary rectifiable varifold in free boundary sense.
	Moreover, since $\Theta^k(\|V\|,rY)=\Theta^k(\|V\|,Y)$ for each $r>0$ by \ref{it:Vcone}, we have $\Theta^k(\|V_1\|, rY)=\Theta^k(\|V_1\|,0)$ for every $r \in \mathbb{R}$.
	Therefore, $V_1$ is translation invariant along direction $Y$ by monotonicity formula (e.g., Theorem \ref{thm_monotonicityFormula} with $\mu=0$).
	Then, we can write $V_1=V_2\times \mathbb{R}$ where $V_2$ is an $(k-1)$-dimensional rectifiable varifold in $\mathbb{H}^{k}$ such that \ref{it:Vcone}, \ref{it:statFB} are satisfies with $V_2$ in place of $V$, $k-1$ in place of $k$.
	It is straightforward to check \ref{it:VvolumeBound} holds with $V_2$ in place of $V$, $k-1$ in place of $k$.

	At last, we suppose $(\eta_{Y,\rho_i})_{\#}V\rightarrow V_1$ for some $\rho_i\rightarrow 0^+$.
	Since each $(\eta_{Y,\rho_i})_{\#}V$ is stationary in free boundary sense in $\mathbb{H}^{k+1}\backslash \left\{ x_1=0,x_2\le 0 \right\}$, we know $V_2$ is integral rectifiable on $\mathbb{H}^{k}\backslash \left\{ x_1=0,x_2\le 0 \right\}$, which shows \ref{it:VdensityInteger} holds with $V_2$ in place of $V$, $k-1$ in place of $k$.

	By the induction hypothesis, we have $\Theta^{k-1}(\|V_2\|,0)\ge \frac{1}{2}$.
	Again, monotonicity formula implies $\Theta^k(\|V\|,0)\ge \Theta^k(\|V\|,Y)=\Theta ^{k-1}(\|V_2\|,0)\ge \frac{1}{2}$, thus concluding the proof.
\end{proof}

\begin{proof}
	[Proof of Theorem \ref{thm:smallDensity}]
	We write $W=|U|$ for some $\mathcal{H}^n$-measurable set $U \subset \mathcal{B}_1\cap \partial \mathbb{H}^{n+1}$.
	We claim that $\delta W=0$ in $\mathcal{B}_1\cap \partial \mathbb{H}^{n+1}\cap \{ \theta\neq \frac{\pi}{2} \}$, which equivalently means $\delta^g W(\varphi)=0$ for any $\varphi \in \mathfrak{X}^1_c(\mathcal{B}_1\cap \partial \mathbb{H}^{n+1}\cap \{ \theta\neq\frac{\pi}{2} \})$ by Proposition \ref{prop_equivCacci}.

	Indeed, we only need to show $\delta W=0$ in $\mathcal{B}_1 \cap \partial \mathbb{H}^{n+1}\cap \left\{ \theta>\frac{\pi}{2} \right\}$.
	This is because, by Remark \ref{rmk_changeAngle}, we know $(V,|\mathcal{B}_1\cap \partial \mathbb{H}^{n+1}|-W,\frac{\pi}{2}-\theta,g)$ satisfies the same theorem assumption as $\mathcal{V}$ and $\delta W= - \delta(|\mathcal{B}_1\cap \partial \mathbb{H}^{n+1}|-W)$.

	Suppose there exists a point $X \in \mathcal{B}_1\cap \partial \mathbb{H}^{n+1}$ and a positive number $\rho \in (0,1-|X|)$ such that $\theta>\frac{\pi}{2}$ on $\mathcal{B}_\rho(X)\cap \partial \mathbb{H}^{n+1}$, and $\delta W\neq 0$ in $\mathcal{B}_\rho(X)\cap \partial\mathbb{H}^{n+1}$
	We assume $|\cos \theta|\ge \delta$ on $\mathcal{B}_\rho(X)\cap \partial \mathbb{H}^{n+1}$ for some $\delta>0$ small enough.

	Now, we proceed to show that $U\cap B_\rho(X)$ is a Caccioppoli set in $\mathcal{B}_\rho(X)\cap \partial \mathbb{H}^{n+1}$.

	Let $\varphi \in \mathfrak{X}^1_{c}(\mathcal{B}_\rho(X)\cap \partial \mathbb{H}^{n+1})$ and extend it to be a vector field in $\mathfrak{X}^1_{c,\tan}(\mathcal{B}_1)$ with $\sup_{X \in \mathcal{B}_1}|\varphi|_g\le 2 \sup_{X \in \mathcal{B}_1\cap \partial \mathbb{H}^{n+1}}|\varphi|_g$.

	Then by \eqref{eq:defPrescribedAngle} and the definition of bounded first variation \eqref{eq:defBounded1stVar}, we have
	\begin{equation}
		|\delta^gW(\cos \theta\varphi)|\le \mu \int_{} |\varphi|_gd\|V\|_g+|\delta^gV(\varphi)|\le C \sup _{X \in \mathcal{B}_1\cap \partial \mathbb{H}^{n+1}}|\varphi|_g
		\label{eq:pf1stVarW}
	\end{equation}
	for any $\varphi \in \mathfrak{X}^1_{c,\tan}(\mathcal{B}_1)$.
	Additionally, for $S=\left\{ x_1=0 \right\}$, we find that $\mathrm{div}_S^g(\cos \theta \varphi)\ge |\cos \theta|\mathrm{div}_S^g(\varphi)-C\sin \theta |D \theta| |\varphi|_g$.
	Together with $|D \theta|\le \mu$, $|g-\delta|\le \mu$, and $|\cos \theta|\ge \delta$ in $\mathcal{B}_\rho(X)\cap \partial \mathbb{H}^{n+1}$, it follows that
	\[
		|\delta^g W(\varphi)|\le
		C \sup_{X \in \mathcal{B}_\rho(X)\cap \partial \mathbb{H}^{n+1}}|\varphi|_g,
	\]
	for some $C(n,\delta)$.
	Then, Proposition \ref{prop_equivCacci} implies that $U$ is a Caccioppoli set in $\mathcal{B}_\rho(X)\cap \partial \mathbb{H}^{n+1}$.
	Next, we choose a point $Z \in \mathcal{B}_\rho(Z)\cap \partial \mathbb{H}^{n+1}$ such that the approximate tangent space of $\partial^*U$ exists.
	Then, it follows that the tangent cone of $W$ at $Z$ is a half-space of $\mathbb{R}^n=\partial \mathbb{H}^{n+1}$.
	WLOG, we assume the tangent cone of $W$ at $Z$ to be $|\left\{ x_2\le 0 ,x_1=0\right\}|$.
	Considering a sequence of $\mathcal{V}_i=(V_i,W_i,\theta_i,g_i):=\mathcal{V}_{Z,\rho_i}$ for some $\rho_i\rightarrow 0^+$, we assume $\mathcal{V}'=(V',W',\theta',\delta)=\lim_{i\rightarrow +\infty} \mathcal{V}_i$ where $\theta'=\theta(Z)$.
	It implies $V'-\cos \theta' W'$ is a rectifiable stationary cone in free boundary sense in $\mathcal{B}_1$, and
	\begin{equation}
		\Theta(\mathcal{V}',0)=\Theta(\mathcal{V},Z)<1-\cos \theta'.
		\label{eq:pfDensityUpper}
	\end{equation}
	Note that since $W'=|\left\{ x_2\le 0,x_1=0 \right\}\cap \mathcal{B}_1|$, we have $\Theta(V',0)<\frac{1}{2}$.

	Now, our aim is to apply Lemma \ref{lem_dimReduction} to $V'$ and get a contradiction.
	We extend $V'$ to be a rectifiable varifold defined on $\mathbb{H}^{n+1}$ while ensuring it is a cone.
	We already know \ref{it:Vcone}, \ref{it:VvolumeBound}, and \ref{it:statFB} hold with $V'$ in place of $V$, $-\cos\theta'$ in place of $\vartheta$, and $n$ in place of $k$.
	Let $U\subset \mathcal{B}_1$ be any (relatively) open set in $\mathcal{B}_1$ such that $\overline{U}\cap \left\{ x_1=0,x_2\le 0 \right\}=\emptyset$.
	Then, we find $\mathrm{spt}\|W_i\|\cap U=\emptyset$ for $i$ large enough since $\lim_{i\rightarrow +\infty}W_i=|\left\{ x_2\le 0, x_1=0 \right\}\cap \mathcal{B}_1|$.
	Consequently, $V_i$ is a stationary varifold in free boundary sense in $U$ for $i$ large enough, which leads that $V'$ is integral rectifiable in $U$ by compactness theorem (e.g., Theorem \ref{thm_compactnessRIV} and Theorem \ref{thm_compactness_theorem_for_n_varifolds_under_lipschitz_metric}).
	By the choice of $U$, we can conclude that $V'$ has integer multiplicity $\|V'\|$-a.e. on $\mathbb{H}^{n+1}\backslash \left\{ x_1=0,x_2\le 0 \right\}$.
	Thus, Lemma \ref{lem_dimReduction} implies $\Theta(V',0)\ge \frac{1}{2}$, which contradicts \eqref{eq:pfDensityUpper}.

	Therefore, $\delta W=0$ in $\mathcal{B}_1\cap \partial \mathbb{H}^{n+1}\cap \{ \theta\neq \frac{\pi}{2} \}$.
	In particular, by Constancy Theorem, we know $W\lfloor(U_0)=m|U_0|$ for some $m=0$ or 1 if $U_0$ is a connected component of $\mathcal{B}_1\cap \partial \mathbb{H}^{n+1}\backslash \{ \theta \neq \frac{\pi}{2}\}$.

	To prove that $\mathrm{spt}\|V\|\cap \mathcal{B}_1\cap \partial \mathbb{H}^{n+1}\cap \{ \theta=\frac{\pi}{2} \}=\emptyset$, we assume the contrary that
	there exists $X \in\mathrm{spt}\|V\|\cap \mathcal{B}_1\cap \partial \mathbb{H}^{n+1}\cap \{ \theta=\frac{\pi}{2} \}$.
	Let $\mathcal{V}'=(V',W',\frac{\pi}{2},\delta) \in \mathrm{VarTan}(\mathcal{V},X)$ and then, $V'$ is a stationary varifold in free boundary sense in $\mathcal{B}_1$.
	By compactness theorem (Theorem \ref{thm_compactnessRIV}), $V'$ is integral rectifiable in $\mathcal{B}_1$.
	A reflection argument then leads to $\Theta(\|V\|,X)\ge \frac{1}{2}$, which contradict our assumption $\Theta(\mathcal{V},X)<1-\cos \theta(X)=1$.
	This contradiction proves $\mathrm{spt}\|V\|\cap \mathcal{B}_1\cap \partial \mathbb{H}^{n+1}\cap \{ \theta=\frac{\pi}{2} \}=\emptyset$.

	Finally, we need to show $V$ is a $\mu$-stationary varifold under metric $g$ in free boundary sense.
	For any $\varphi \in \mathfrak{X}^1_{c,\tan}(\mathcal{B}_1)$, we choose another $\varphi' \in \mathfrak{X}^1_{c,\tan}(\mathcal{B}_1)$ such that $\varphi=\varphi'$ in a neighborhood of $\mathrm{spt}\|V\|$ and $\varphi'=0$ in a neighborhood of $\mathcal{B}_1\cap \partial \mathbb{H}^{n+1}\cap \{ \theta=\frac{\pi}{2} \}$.
	Then, it follows that
	\[
		\delta^gV(\varphi)=
		\delta^g V(\varphi')=
		\delta^g V(\varphi')-\delta^gW(\cos \theta \varphi')=
		\int_{ } \left< \boldsymbol{H}_g,\varphi' \right> _g d\|V\|_g=
		\int_{ } \left< \boldsymbol{H}_g,\varphi \right> _g d\|V\|_g
	\]
	demonstrating that $V$ is indeed $\mu$-stationary under metric $g$ in free boundary sense.
\end{proof}

For convenience, for any $\mathcal{V} \in \mathcal{RIV}(\mu)$, we define
\[
	\mu(\mathcal{V})=\mu(V,W,\theta,g):=\inf\left\{ \mu\ge 0: \mathcal{V} \in \mathcal{RIV}(\mu) \right\}.
\]

For any half hyperplane $H=H^{\theta'}$ for some $\theta' \in (0,\pi)$, we define the excess of $\mathcal{V}$ with respect to $H$ as
\[
	\mathcal{E}(\mathcal{V},H):=\sqrt{\int_{ B_1(0)} \mathrm{dist}^2(X,\mathrm{spt}\|\boldsymbol{C}_H\|)d\|V^\theta\|(X)+\mu(\mathcal{V})}.
\]
Here, $\boldsymbol{C}_H:=|\left\{ X \in \mathbb{H}^{n+1}:X \in H \text{ or }x_1=0, x_2\le 0 \right\}|$.
For the sake of simplicity, we may refer to $\boldsymbol{C}$ as $\boldsymbol{C}_H$ when the context clearly identifies $H$.

\begin{hypothesis}
	Given $\varepsilon>0$, a quadruple $\mathcal{V} \in \mathcal{RIV}(\mu)$ for $\mu$ small enough, and a half hyperplane $H$, we say $(\mathcal{V},H)$ satisfies \textit{$(\Lambda_0,\varepsilon)$-Hypothesis} if the following conditions hold.
	\begin{enumerate}[\normalfont(a)]
		\item $\mathcal{V}$ is $(\Lambda_0,\mu)$-stationary with $\mu<\varepsilon$ and $g(0)=\delta$.
		\item $\Theta(\mathcal{V},0)\ge (1-\cos \theta(0))$ and $\|V^\theta\|(B_1(0))<\frac{3}{4}(1-\cos \theta(0))\omega_n$,
			\label{hyp_l2_2}
		\item $H$ can be written as $H=H^{\theta'}$ with $|\theta'-\theta(0)|<\varepsilon$.
			\label{hyp_l2_3}
		\item The half hyperplane $H$ satisfies $\mathcal{E}^2(\mathcal{V},H)<\varepsilon$.
			\label{hyp_l2_4}
	\end{enumerate}
	\label{hyp_l2}
\end{hypothesis}

We now present the main regularity theorem under Hypothesis \ref{hyp_l2}.

\begin{theorem}
	\label{thm_l2positiveAngle}
	There exists a constant $\varepsilon=\varepsilon(n,\Lambda_0)\in (0,1)$ such that if $(\mathcal{V},H)$ satisfies $(\Lambda_0,\varepsilon)$-Hypothesis, then $\mathrm{spt}\|V\|\cap B_{\frac{1}{32}}=\Sigma$ where $\Sigma$ is a $C^{1,\gamma}$-hypersurface with boundary in $B_{\frac{1}{32}}$ such that $\partial \Sigma \cap B_{\frac{1}{32}}\subset \partial\mathbb{H}^{n+1}$ and $\measuredangle _g(\Sigma,U)=\theta$ along $\partial\Sigma$ where $U$ is the connected component of $\mathcal{B}_{\frac{1}{32}}\cap \partial \mathbb{H}^{n+1}\backslash \Sigma$ positioned below the $\partial\Sigma$.
\end{theorem}
\begin{remark}
	It is noteworthy that $\partial\Sigma$ can be written as a graph of a function defined on a region in $\left\{ x_1=x_2=0 \right\}$ and it makes sense to talk about the region situated below $\partial \Sigma$. See the proof of Theorem \ref{thm_l2positiveAngle} in Chapter \ref{sec:iterations}.
\end{remark}

\begin{lemma}
	\label{lem_graph_over_cone}
	Let $\tau \in(0,\frac{1}{8})$. For any $\delta>0$, there exist constants $\varepsilon=\varepsilon(n,\tau,\delta,\Lambda_0)\in (0,1)$ and $\beta=\beta(n) \in (0,1)$ such that if $(\mathcal{V},H)$ satisfies $(\Lambda_0,\varepsilon)$-Hypothesis,
	then we can find $u \in C^{1,\beta}(B_{\frac{15}{16}}(0)\cap H\backslash \left\{ r<\frac{\tau}{2} \right\},H^\bot)$ such that
	\[
		V\lfloor(B_{\frac{7}{8}}(0)\backslash \left\{ r<\tau \right\})\subset \mathrm{graph}u, \text{ with }\|u\|_{C^{1,\beta}(B_{\frac{15}{16}}(0)\cap H\backslash \{ r<\frac{\tau}{2} \},H\bot )}\le \delta,
	\]
	and
	\[
		W \lfloor(B_{\frac{7}{8}}(0)\backslash \left\{ r<\tau \right\})=|B_{\frac{7}{8}}\cap \left\{ r<\tau \right\}\cap \left\{ x_1<0 \right\}|.
	\]
\end{lemma}

\begin{proof}
	Suppose we have a sequence of VPCA-quadruples $\left\{ \mathcal{V}_i \right\}$, half planes $H_i$, and $\varepsilon_i\rightarrow 0^+$ such that $(\mathcal{V}_i,H_i)$ satisfies $(\Lambda_0,\varepsilon_i)$-hypothesis.

	By Theorem \ref{thm_compactnessRIV}, up to a subsequence, we suppose $\mathcal{V}_i \rightarrow \mathcal{V}=(V,W,\theta,\delta)$ with $V_\theta$ being a stationary varifold in free boundary sense.
	Note that $\theta$ is a constant in $[\frac{\pi}{2}+\Lambda_0,\pi-\Lambda_0]$.
	We also have $|H_i-H|\rightarrow 0$ as $i\rightarrow +\infty$ where $H=H^\theta$.

	By \ref{hyp_l2_4} in Hypothesis \ref{hyp_l2}, monotonicity formulas (cf. Corollary \ref{cor_monoG}, Corollary \ref{cor_MonoFormula}), alongside the lower bounds of density (cf. Lemma \ref{lem_lowerBoundDensity}), we obtain
	\begin{equation}
		\sup_{X \in \mathrm{spt}\mathcal{V}_i\cap B_{\frac{31}{32}}(0)}\mathrm{dist}(X,\mathrm{spt}\|\boldsymbol{C}\|\cap B_{\frac{31}{32}}(0))\rightarrow 0
		\text{ as }i\rightarrow +\infty.
		\label{eqLemHauD}
	\end{equation}
	This implies $\mathrm{spt}\|V^\theta\|\subset \mathrm{spt}\|\boldsymbol{C}\|$.

	By the constancy theorem, we have $V^\theta=a_1\left|H\right|+a_2\left|H^0\right|$ where $a_1$ is an non-negative integer by Compactness Theorem \ref{thm_compactness_theorem_for_n_varifolds_under_lipschitz_metric}.
	Recall that $H^0$ is defined as $\left\{ x_1=0,x_2\le 0 \right\}$.
	With $V^\theta$ being stationary in free boundary sense, it follows that $a_2=-a_1 \cos \theta$.
	The lower semi-continuity of mass informs us that $\|V^\theta\|(B_1(0))\le \frac{3}{4}(1-\cos \theta)$, implying $a_1\le 1$.
	On the other hand, the upper-semi continuity of density (Theorem \ref{thm_compactnessRIV}) gives us $\Theta(\|V^\theta\|,0)\ge \frac{1}{2}(1-\cos \theta)$, leading $a_1=1$, and $a_2=-\cos \theta$.

	By \eqref{eqLemHauD}, we can decompose $V_i\lfloor(B_{\frac{{31}}{32}}\backslash \left\{ r<\frac{\tau}{4} \right\})=V_{i,1}+V_{i,2}$ where
	\begin{align*}
		V_{i,1}={} & V_i\lfloor(N(H)\cap B_{\frac{31}{32}}\cap \{ r<\frac{\tau}{4} \}) \\
		V_{i,2}={} & V_i\lfloor(N(H^0)\cap B_{\frac{31}{32}}\cap \{ r<\frac{\tau}{4} \})
	\end{align*}
	Hence, $V_{i,1}$ is a $\mu_i$-stationary integral rectifiable varifold with $V_{i,1}\rightarrow |H\cap B_{\frac{{31}}{32}}\backslash \{ r<\frac{\tau}{4} \}|$ and 
	\[
		\int_{ N(H)} \mathrm{dist}^2(X,H)d\|V_{i,1}\|(X)\rightarrow 0 \text{ as }i\rightarrow +\infty.
	\]

	Applying Theorem \ref{thm_allard} with $\delta=\frac{1}{2}, \gamma =\frac{1}{2}$ allows us to assert
	\[
		V_{i,1} \lfloor(B_{\frac{7}{8}}(0)\backslash \{ r<\tau \})\subset \mathrm{graph}u_i
	\]
	for a sequence of functions $u_i \in C^{1,\beta}(B_{\frac{15}{16}}(0)\cap H\backslash \{ r<\frac{\tau}{2} \},H^\bot )$ with 
	\[
		\|u_i\|_{C^{1,\beta}(B_{\frac{15}{16}}(0)\cap H\backslash \{ r<\frac{\tau}{2} \},H^\bot )}\rightarrow 0 \text{ as }i\rightarrow +\infty.
	\]
	for some $\beta=\beta(n) \in (0,1)$.

	Now, we have to demonstrate that $V_{i,2}\lfloor(B_{\frac{7}{8}}\backslash \{ r<\tau \})=0$ for $i$ large enough.

	We claim that
	\[
		\Theta(\mathcal{V}_i,X)<1-\cos(\theta_i(X))
	\]
	for any $X \in \partial \mathbb{H}^{n+1}\cap B_{\frac{15}{16}}(0) \backslash \{ r<\frac{\tau}{2} \}$ for $i$ sufficiently large.
	Otherwise, up to a subsequence, we can find a sequence of $X_i \in \partial \mathbb{H}^{n+1}\cap B_{\frac{15}{16}}(0)\backslash \{ r<\frac{\tau}{2} \}$ such that $\Theta(\mathcal{V}_i,X_i)\ge 1-\cos\theta_i(X_i)$.
	Again, up to a subsequence, we may assume $X_i \rightarrow X_0 \in \partial \mathbb{H}^{n+1}\cap B_{\frac{31}{32}}(0)\backslash \{ r<\frac{\tau}{4} \}$.
	Then the upper semi-continuity of density implies $2a_2\ge 1-\cos \theta$,
	which contradicts with the fact that $a_2=-\cos \theta$.
	Thus, our claim stands.

	With Theorem \ref{thm:smallDensity}, $V_{i,2}$ is a $\mu_i$-stationary varifold in free boundary sense in $B_{\frac{{31}}{32}}\backslash \{ r<\frac{\tau}{4} \}$,
	and $W_i \lfloor(B_{\frac{15}{16}}\backslash \{ r<\frac{\tau}{2} \})=k'|\partial \mathbb{H}^{n+1}\cap B_{\frac{15}{16}}\backslash \{ r<\frac{\tau}{2} \}|$ for some $k'=0$ or $1$.
	By the compactness of integral varifold in free boundary sense, we find $V_{i,2}\rightarrow k|\partial \mathbb{H}^{n+1}\cap B_{\frac{15}{16}}\backslash \left\{ r<\frac{\tau}{2} \right\}|$ for some non-negative integer $k$.
	If $k\ge 1$, it would, by the definition of $V^\theta$ and the upper semi-continuity of density, suggest $a_2\ge 1$, which contradicts our earlier finding that $a_2 = -\cos \theta$.
	Hence, the only possibility is $k=0$.
	This conclusion also implies $k'=1$.
	Again, by a contradiction argument together with a monotonicity formula, we can get $V_{i,2}\lfloor(B_{\frac{7}{8}}\backslash \{ r<\tau \})=0$ for $i$ large enough.

	This concludes the proof.
\end{proof}

For each $\kappa, \rho \in (0,1], \zeta \in \mathbb{R}^{n-1}$, we define the torus $T_{\rho,\kappa}(\zeta)$ by
	\[
		T_{\rho,\kappa}(\zeta):=\left\{ (x,y) \in (\mathbb{R}^+\times \mathbb{R})\times \mathbb{R}^{n-1}:(\left|x\right|-\rho)^2+\left|y-\zeta\right|^2\le \frac{\kappa^2 \rho^2}{64} \right\}.
	\]

\begin{lemma}
	\label{lem_torus}
	Let $\alpha\in (0,1)$ be a constant. 
	There exist constants $\delta=\delta(n,\Lambda_0,\alpha) \in (0,1)$, $\varepsilon=\varepsilon(n,\Lambda_0,\alpha) \in (0,1)$, and $\beta=\beta(n,\Lambda_0) \in (0,1)$ such that following holds.
	Assume $(\mathcal{V},H)$ satisfies $(\Lambda_0,\varepsilon)$-Hypothesis.
	Then for any $(\xi,\eta) \in \mathrm{spt}\|\boldsymbol{C}\|\cap B_{\frac{7}{8}}(0)\cap \left\{ r< \frac{1}{16} \right\}$, if
	\begin{equation}
		\frac{1}{|\xi|^{n+2}}\int_{ T_{|\xi|,1}(\eta)} \mathrm{dist}^2(X,\mathrm{spt}\|\boldsymbol{C}\|)d\|V^\theta\|(X)+\mu(\mathcal{V})|\xi|<\delta,
		\label{eq:lem_graph_torus}
	\end{equation}
	then there exist a function $u ^{\left|\xi\right|,\eta} \in C^{1,\beta}(T_{|\xi|,\frac{3}{4}}(\eta)\cap H,H^\bot )$ with
	\begin{equation}
		\mathrm{spt}\|V\|\lfloor(T_{\left|\xi\right|,\frac{1}{2}}(\eta))\subset \mathrm{graph} u ^{\left|\xi\right|,\eta}\subset \mathrm{spt}\|V\|
		\label{eq:lemGraphTorus}
	\end{equation}
	and
	\begin{equation}
		\frac{1}{\left|\xi\right|}\sup_{H \cap T_{\left|\xi\right|,\frac{1}{2}}(\eta)}\left|u ^{\left|\xi\right|,\eta}\right|+
		\sup_{H \cap T_{\left|\xi\right|,\frac{1}{2}}}\left|D u ^{\left|\xi\right|,\eta}\right|\le \frac{\alpha}{2}.
		\label{eq:lemGradientGraph}
	\end{equation}
\end{lemma}

\begin{proof}
	Considering \eqref{eq:lem_graph_torus} alongside lower density bounds \eqref{eq:lemDensityLower}, we find the Hausdorff distance between $\mathrm{spt}\|V^\theta\|\cap T_{\left|\xi\right|,\frac{7}{8}}$ and $\mathrm{spt}\|\boldsymbol{C}\|\cap T_{\left|\xi\right|,\frac{7}{8}}$ can be arbitrary small when we choose $\delta=(n,\Lambda_0)$ small enough.
	Consequently, we decompose $V\lfloor(T_{|\xi|,\frac{7}{8}})=V_1+V_2$ where
	\begin{align*}
		V_1={} & V\lfloor(T_{|\xi|,\frac{7}{8}}\cap N(H)) \\
		V_2={} & V\lfloor(T_{|\xi|,\frac{7}{8}}\cap N(H^0))
	\end{align*}
	with $V_1$ being $\mu$-stationary in $T_{|\xi|,\frac{7}{8}}\cap N(H)$.
	
	By the monotonicity formula \eqref{eq:MonotonicityFormulaG} and employing the covering argument, we can find
	\begin{equation}
		\|V^\theta\|(\left\{ r<\tau \right\}\cap B_{\frac{15}{16}}(0))\le C\tau,
		\label{eq:pfAreaCylinder}
	\end{equation}
	for some constant $C=C(n,\Lambda_0)$.

	For each $X_0=(\xi,\eta) \in H\cap B_{\frac{13}{16}}(0)\cap \left\{ r<\frac{1}{16} \right\}$, we can choose $\tau$ and $\varepsilon$ small enough, use monotonicity formula \eqref{eq:MonotonicityFormulaInteriorG} with $(\frac{|\xi|}{4},\frac{1}{16})$ in place of $(\sigma,\rho)$, we have
	\[
		\frac{4^n}{\omega_n |\xi|^n}
		 \|V^\theta\|(B_{\frac{|\xi|}{4}}(X_0)) \le (1+C\mu) \frac{\|V^\theta\|(\mathcal{B}_{\frac{1}{16}}(X_0))+\|V^\theta\|(\mathcal{B}_{\frac{1}{16}})(\tilde{X} _0)}{\omega_n\left( \frac{1}{16} \right)^n}+C\mu
	\]
	for $\mu=\mu(n)$ small enough and a constant $C=C(n)$.
	Note that for any $\delta'>0$, we can make sure
	\[
		\frac{\|V^\theta\|(\mathcal{B}_{\frac{1}{16}}(X_0))+\|V^\theta\|(\mathcal{B}_{\frac{1}{16}})(\tilde{X} _0)}{\omega_n\left( \frac{1}{16} \right)^n}
	\]
	is smaller than $1-\cos(\pi-\Lambda_0)+\delta'$ provided that  $\varepsilon$ and $\tau$ are chosen to be small enough in view of Lemma \ref{lem_graph_over_cone} and \eqref{eq:pfAreaCylinder}.
	This suggests that 
	\begin{equation}
		\frac{4^n}{\omega_n |\xi|^n}
		 \|V_1\|(B_{\frac{|\xi|}{4}}(X_0))\frac{4^n}{\omega_n |\xi|^n}
		 =\|V^\theta\|(B_{\frac{|\xi|}{4}}(X_0)) \le 1-\cos \left(\pi-\frac{\Lambda_0}{2}\right)
		 \label{eq:pfVolumeUpper}
	\end{equation}
	by choosing $\varepsilon=\varepsilon(n,\Lambda_0)$ small enough.
	
	Now, we apply Theorem \ref{thm_allard} by choosing $\varepsilon=\varepsilon(n,\Lambda_0)$ to ensure the existence of $u ^{|\xi|,\eta} \in C^{1,\beta}(T_{|\xi|,\frac{3}{4}}(\eta)\cap H,H^\bot)$ satisfying \eqref{eq:lemGraphTorus} and \eqref{eq:lemGradientGraph} for some $\beta=\beta(n,\Lambda_0)$.

	We need to demonstrate that $V_2\lfloor(T_{|\xi|,\frac{7}{8}})=0$.
	Following a similar argument to \eqref{eq:pfVolumeUpper}, we establish
	\begin{equation}
		\frac{1}{\omega_n \rho^n} \|V^\theta\|(B_{\rho}(X_0))\le \frac{1}{2}\left( 1- \cos \left( \pi - \frac{\Lambda_0}{2} \right) \right),
		\label{eq:pfVolXiBall}
	\end{equation}
	for any $X_0 \in \left\{ x_1=0,0<x_2<\frac{1}{16} \right\}\cap B_{\frac{13}{16}}$ and $\rho \in (0,\frac{1}{16})$ if we choose $\varepsilon$ small enough.
	
	We proceed with a contradiction argument.
	Suppose we can find a sequence of $\delta_i,\varepsilon_i\rightarrow 0^+$, $\mathcal{V}_i$, and $H_i$ for which $(\mathcal{V}_i,H_i)$ satisfies $(\Lambda_0,\varepsilon_i)$-hypothesis and $\xi_i,\eta_i\in \mathrm{spt}\|\boldsymbol{C}\|\cap B_{\frac{7}{8}}(0)\cap \{ r<\tau \}$ such that \eqref{eq:lem_graph_torus} holds with $\xi_i$ in place of $\xi$, $\eta_i$ in place of $\eta$, $\delta_i$ in place of $\delta$. 
	Additionally, $\mathrm{spt}\|V_{i,2}\|\cap T_{\left|\xi_i\right|,\frac{7}{8}}\neq \emptyset$, where $V_{i,2}:=V_i \lfloor(T_{\left|\xi\right|,\frac{7}{8}}\cap N(H^0))$.
	Now, we consider $\mathcal{V}_i'=(V_i',W_i',\theta_i',g_i'):=\mathcal{V}_{(-|\xi_i|,0,\eta_i),\frac{|\xi_i|}{8}}$.
	We note that \eqref{eq:pfVolXiBall} implies
	\[
		\frac{1}{\omega_n}\|(V'_i)^{\theta'_i}\|(\mathcal{B}_1)\le \frac{1}{2}\left( 1- \cos \left( \pi - \frac{\Lambda_0}{3} \right) \right).
	\]
	for $i$ large enough.
	Hence, up to a subsequence, we may assume $\mathcal{V}_i'\rightarrow \mathcal{V}'$ for some $\mathcal{V}'=(V',W',\theta',\delta)$ with
	\[
		\|(V')^{\theta'}\|(\mathcal{B}_1)\le \frac{1}{2}\left( 1- \cos \left( \pi - \frac{\Lambda_0}{3} \right) \right).
	\]
	Condition \eqref{eq:lem_graph_torus} implies $\mathrm{spt}\|(V')^{\theta'}\|\subset \partial \mathbb{H}^{n+1}\cap \mathcal{B}_1$ and hence $(V')^{\theta'}=a|\partial \mathbb{H}^{n+1}\cap \mathcal{B}_1|$ for some $0\le a \le \frac{1}{2}\left( 1- \cos \left( \pi - \frac{\Lambda_0}{3} \right) \right)$ by constancy theorem.
	Again, by Theorem \ref{thm_compactnessRIV}, we know the only possibility values for $a$ are 0 or $-\cos \theta'$.
		
	By the upper semi-continuity of density \eqref{eq:thmUpperDensityStat}, and Theorem \ref{thm:smallDensity}, we see that $V_i'$ is a $\mu_i'$-stationary varifold in free boundary sense in $B_1(0)$.
	Applying Theorem \ref{thm_compactnessRIV} once more with $(V_i', 0, \frac{\pi}{2},g_i')$ in place of $\mathcal{V}_i$ and considering $(V'_i)^{\theta_i'}\rightarrow a|\partial \mathbb{H}^{n+1}\cap \mathcal{B}_1|$ for $a<1$, we have $V_i'\rightarrow 0$.
	This convergence gives that $\mathrm{spt}\|V_i'\|\cap \mathcal{B}_{\frac{15}{16}}=0$ for $i$ large enough.
	This contradicts with the fact $\mathrm{spt}\|V_{i,2}\|\cap T_{\left|\xi_i\right|,\frac{7}{8}}\neq \emptyset$ and definition of $V_i'$.
\end{proof}

\begin{lemma}
	\label{lem_graph}
	Suppose $\tau \in (0,\frac{1}{80})$ and $\alpha \in (0,1)$, there exists $\varepsilon=\varepsilon(n,\tau,\alpha,\Lambda_0) \in (0,\mu_0)$, $\beta=\beta(n,\Lambda_0) \in (0,1)$ such that, if $(\mathcal{V},H)$ satisfies $(\Lambda_0,\varepsilon)$-Hypothesis,
	then we can find two open subsets $U_V \subset B_1(0)\cap H, U_W \in B_1(0)\cap \partial \mathbb{H}^{n+1}\cap \left\{ x_1<0 \right\}$ and a $C^{1,\beta}$ function $u \in C^{1,\beta}(U_V,H^\bot )$ with the following properties,
	\begin{enumerate}[\normalfont(a)]
		\item $B_{\frac{7}{8}}\cap \left\{ r(X)>\tau \right\}\cap H\subset U_V$, $B_{\frac{7}{8}}\cap \left\{ r(X)>\tau \right\}\cap \partial \mathbb{H}^{n+1}\cap \left\{ x_1<0 \right\}\subset U_W$.
		\item $\mathrm{spt}\|V\|\lfloor(B_{\frac{7}{8}}\cap \left\{ r> \tau \right\})\subset \mathrm{graph}u \subset \mathrm{spt}\|V\|$, $\mathrm{spt}\|W\|\cap B_{\frac{7}{8}}\cap \left\{ r>\tau \right\}\subset U_W$.
		\item $\sup_{U_V}\frac{\left|u\right|^2}{r^2}+\sup_{U_V}\left|D u\right|^2\le \alpha^2$.
		\item $\int_{ \mathrm{spt}\|W\|\backslash U_W}r^2(X)d\|W\| +\int_{ B_{\frac{7}{8}}(0)\backslash \mathrm{graph}u} r^2(X)d\|V\|(X)+\int_{ U_V \cap B_{\frac{7}{8}}(0)} r^2\left|D u\right|^2d\mathcal{H}^n(X)\le C\mathcal{E}^2(\mathcal{V},H)$ where $C=C(n,\alpha,\Lambda_0)$.
	\end{enumerate}
\end{lemma}

\begin{proof}
Our proof is largely following from \cite[Lemma 2.6]{Simon1993cylindrical}.

	Let $U_V$ be the union of all $T_{\left|\xi\right|,\frac{1}{2}}(\eta)\cap H$ over all $(\xi, \eta) \in H\cap B_{\frac{7}{8}}(0)$, where there exists $u ^{\left|\xi\right|,\eta} \in C^{1,\beta}(T_{\left|\xi\right|,\frac{3}{4}}(\eta),H^\bot )$ with
	\[
		\mathrm{spt}\|V\|\lfloor(T_{\left|\xi\right|,\frac{1}{2}}(\eta))\subset \mathrm{graph} u ^{\left|\xi\right|,\eta}\subset \mathrm{spt}\|V\|
	\]
	and
	\[
		\frac{1}{\left|\xi\right|}\sup_{H \cap T_{\left|\xi\right|,\frac{1}{2}}(\eta)}\left|u ^{\left|\xi\right|,\eta}\right|+
		\sup_{H \cap T_{\left|\xi\right|,\frac{1}{2}}}\left|D u ^{\left|\xi\right|,\eta}\right|\le \frac{\alpha}{2},
	\]
	where $\beta$ is a constant for which Lemma \ref{lem_torus} is applicable.
	
	By the unique continuation of solutions to elliptic operators, we can define $u \in C^{1,\beta}(U_V,H^\bot )$ by
	\[
		u\lfloor(T_{\left|\xi\right|,\frac{1}{2}}\cap H)=
		u ^{\left|\xi\right|,\eta}\lfloor(T_{\left|\xi\right|,\frac{1}{2}}\cap H).
	\]

	Now, according to Theorem \ref{thm_allard}, we claim that if we choose $\varepsilon=\varepsilon(n,\tau,\alpha)$ small enough, we have $(B_{\frac{7}{8}}(0)\backslash \left\{ r(X)<\tau \right\})\subset U_V$ and $\mathrm{spt}\|V\|\lfloor(B_{\frac{7}{8}}\cap\{r>\tau\})\subset \mathrm{graph}u \subset \mathrm{spt}\|V\|$.
	This follows from the compactness theorem, constance theorem, and upper semi-continuity of density function by a contradiction argument.
	
	Moreover, for any $(\xi,\eta) \in \partial U\cap B_{\frac{7}{8}}(0)\cap H$, we have that
	\[
		\left|\xi\right|^{n+2}\le C\left( \int_{T_{\left|\xi\right|,1}(\eta)} \mathrm{dist}^2(X,\mathrm{spt}\|\boldsymbol{C}\|)d\|V^\theta\|(X)+\mu(\mathcal{V})|\xi| 
 \right)
\]
for some $C=C(n,\Lambda_0,\alpha)$ based on Lemma \ref{lem_torus}.

Adopting Simon's covering argument from \cite[Lemma 2.6]{Simon1993cylindrical} we can establish the following estimate.
\[
	\int_{ B_{\frac{7}{8}}(0)\backslash \mathrm{graph}u} r^2(X)d\|V\|(X)+\int_{ U_V \cap B_{\frac{7}{8}}(0)} r^2\left|D u\right|^2d\mathcal{H}^n(X)\le C\mathcal{E}^2(\mathcal{V},H),
\]
for some $C=C(n,\alpha,\Lambda_0)$.

	For the choice of $U_W$, 
	we choose $U_W\subset B_1(0)\cap \partial \mathbb{H}^{n+1}\cap \left\{ x_2<0 \right\}$ such that $B_{\frac{7}{8}}\cap\{r>\tau\}\subset U_W$ and $\mathcal{H}^n(\mathrm{spt}\|W\|\cap \left\{ x_2<0 \right\}\cap \left\{ r\le \tau \right\}\backslash U_W)<\mathcal{E}^2(\mathcal{V},H)$.
	Note that by Lemma \ref{lem_graph_over_cone}, we know $B_{\frac{7}{8}}\cap\{r>\tau\}\cap \left\{ x_2<0 \right\}\subset  \mathrm{spt}\|W\|$ and hence,
	\begin{equation}
		\mathcal{H}^n(\mathrm{spt}\|W\|\cap \left\{ x_2<0 \right\}\backslash U_W)<\mathcal{E}^2(\mathcal{V},H).
		\label{eq:pfHnDiff}
	\end{equation}
	It is clear that for $X \in \mathrm{spt}\|W\|\cap \left\{ x_2\ge 0 \right\}$, we have
	\[
		r(X)\le C\mathrm{dist}(X,\mathrm{spt}\|\boldsymbol{C}\|)
	\]
	for some constant $C=C(\Lambda_0)$. Together with \eqref{eq:pfHnDiff}, we know
	\[
		\int_{ \mathrm{spt}\|W\|\backslash U_W} r^2(X)d\|W\|(X)\le C \mathcal{E}^2(\mathcal{V},H).
	\]
\end{proof}

\begin{lemma}
	\label{lem_change_center}
	Let $r\in (0,\frac{1}{3})$. 
	For any $\varepsilon \in (0,1)$, there exists $\varepsilon_0=\varepsilon_0(\varepsilon,n, r,\Lambda_0) \in (0,1)$ such that if $(\mathcal{V},H)$ satisfies $(\Lambda_0,\varepsilon_0)$-hypothesis, then for any $X \in B_{\frac{5}{8}}(0)\cap \partial \mathbb{H}^{n+1}$ with $\Theta(\|\mathcal{V}\|,X)\ge 1-\cos \theta(X)$,
	$(\mathcal{V}_{X,r}, L_X^g(H))$ satisfies $(\Lambda_0,\varepsilon)$-Hypothesis.
\end{lemma}

\begin{proof}
	To prove the lemma, we examine the four conditions outlined in Hypothesis \ref{hyp_l2}.
	The first one is a direct consequence of Proposition \ref{prop_rescaling}.

	The approach for the second condition is similar to the proof for \eqref{eq:pfVolumeUpper}.

	Regarding the third criterion, we denote $H=H^{\theta'}$ for some $\theta_1' \in (0,\pi)$, and in line with Remark \ref{rmk:LXg}, express $L_X^g(H)=H^{\theta_2'}$.
	Proposition \ref{prop_Pi2Delta} tells us $|L_X^g(H)-H|<C\mu(\mathcal{V})\le C \varepsilon_0$, leading to $\left|\theta_1'-\theta_2'\right|\le C \varepsilon_0$.
	By the definition of $(\Lambda_0,\mu)$-stationary, we deduce $\left|\theta(X)-\theta(0)\right|\le \varepsilon_0$ and $\left|\theta_1'-\theta(0)\right|\le \varepsilon_0$.
	Consequently, $\left|\theta_2'-\theta(X)\right|\le C\varepsilon_0<\varepsilon$, assuming $\varepsilon_0$ is adequately small.

	For the last condition, we write $\mathcal{V}'=(V',W',\theta',g')$ and demonstrate that
	\[
		\int_{ B_1(0)} \mathrm{dist}^2(Y,\mathrm{spt}\|\boldsymbol{C}_H\|)d\|(V')^{\theta'}\|(Y)\le \varepsilon,
	\]
	by choosing $\varepsilon_0$ small enough.

	This is equivalent to
	\[
		\int_{ (\Pi^g_{X,\rho})^{-1}(B_1(0))} \mathrm{dist}^2(\Pi^g_{X,\rho}(Y),\mathrm{spt}\|\boldsymbol{C}_H\|)|J_S\Pi^g_{X,\rho}|d\|V^\theta\|(Y)\le \varepsilon,
	\]
	where $J_S\Pi^g_{X,\rho}$ denotes the Jacobian of the map $\Pi^g_{X,\rho}$ restricted to the $S$-plane.
	
	We observe that
	\begin{align*}
		\mathrm{dist}(\Pi^g_{X,\rho}(Y),\mathrm{spt}\|\boldsymbol{C}_H\|)\le{} & \mathrm{dist}(\eta_{X,\rho}(Y),\mathrm{spt}\|\boldsymbol{C}_H\|)+\mathrm{dist}((L_X^g-\mathrm{Id})\circ \eta_{X,\rho}(Y),\mathrm{spt}\|\boldsymbol{C}_H\|),\\
		\le{} & \frac{1}{\rho}\mathrm{dist}(Y,\mathrm{spt}\|\boldsymbol{C}_H\|)+\frac{1}{\rho}r(X)+C\frac{\mu}{\rho},\\
		|J_S \Pi^g_{X,\rho}|\le{}& \frac{1+C\mu }{\rho^n}.
	\end{align*}

	Consequently, we establish
	\begin{align*}
		{}&\int_{ (\Pi^g_{X,\rho})^{-1}(B_1(0))} \mathrm{dist}^2(\Pi^g_{X,\rho}(Y),\mathrm{spt}\|\boldsymbol{C}_H\|)|J_S\Pi^g_{X,\rho}|d\|V^\theta\|(Y)\\
		\le{}& \frac{1}{\rho^{n+2}}\int_{ B_1(0)} \left( \mathrm{dist}^2(Y,\mathrm{spt}\|\boldsymbol{C}_H\|)+r^2(X) \right)d\|V^\theta\|(Y)+\frac{C\mu}{\rho^{n+2}}.
	\end{align*}
	It can be sufficiently small provided $\varepsilon_0$ is chosen to be small enough and ensure $r^2(X)$ is also small enough based on Lemma \ref{lem_graph}.
	Lastly, we can conclude the proof by considering $|L_X^g(H)-H|\le C\mu$.
\end{proof}

\begin{theorem}
	[$L^2$-estimate]
	For any $\tau \in (0,\frac{1}{8})$, $\omega \in (0,1)$, there exist constants $\varepsilon_0=\varepsilon_0(n,\tau,\Lambda_0) \in (0,\frac{1}{2})$ and $\beta=\beta(n,\Lambda_0) \in (0,1)$ such that if $(\mathcal{V},H)$ satisfies $(\Lambda_0,\varepsilon_0)$-hypothesis. Then, the following conclusions hold,
	\begin{enumerate}[\normalfont(a)]
		\item $V\lfloor(B_{\frac{13}{16}}\backslash \{ r<\tau \})=\left|\mathrm{graph}u \cap B_{\frac{13}{16}}\backslash \left\{ r<\tau \right\}\right|$ where $u \in C^{1,\beta}(B_{\frac{13}{16}}\cap H\backslash \{ r<\frac{\tau}{2} \},H^{\bot} )$, and $\mathrm{dist}(X+u(X),\mathrm{spt}\|\boldsymbol{C}\|)=|u(X)|$ for $X \in B_{\frac{13}{16}}\cap H\backslash \{ r<\tau \}$.
		\label{it:L2Est1}
	\item $W\lfloor(B_{\frac{13}{16}}\backslash \left\{ r<\tau \right\})=|B_{\frac{13}{16}}\cap \left\{ x_1<-\tau \right\}\cap \partial \mathbb{H}^{n+1}|$.
			\label{it:Wsupport}
		\item $\int_{ B_{\frac{3}{4}}} \frac{|X^{\bot_S} |^2}{|X|^{n+2}}dV(X,S)\le C \mathcal{E}^2(\mathcal{V},H)$.
		\item $\sum_{j =3}^{n+1}\int_{ B_{\frac{3}{4}}} |e_j^{\bot_S} |^2dV(X,S)\le C \mathcal{E}^2(\mathcal{V},H)$.
			\label{it:L2EsteiNormal}
		\item $\int_{ B_{\frac{3}{4}}} \frac{\mathrm{dist}^2(X,\mathrm{spt}\|\boldsymbol{C}\|)}{|X|^{n+2-\omega}}d\|V^\theta\|(X)\le C_1 \mathcal{E}^2(\mathcal{V},H)$.
	\end{enumerate}
	Here, $C=C(n,\Lambda_0)$ and $C_1=C_1(n,\Lambda_0,\omega)$.
	\label{thm__l_2_estimate}
\end{theorem}

\begin{proof}
	The first two conclusions are the consequences of Lemma \ref{lem_graph}.

	The subsequent analyses rely on the application of the monotonicity formula.
	At first, we claim we can estimate $\|V^\theta\|(\mathcal{B}_\rho)$ by
	\begin{equation}
		\|V^\theta\|(\mathcal{B}_\rho)\le
		\frac{(1+C\mu)\rho}{n}\frac{d}{d\rho}\|V^\theta\|(B_\rho)+C\mu \rho^n,\quad \text{ for a.e. } \rho \in (0,1)
		\label{eq:pfEstVolumeBall}
	\end{equation}
	for some $C=C(n)$ provided $\varepsilon_0$ small enough.
	
	Given $\rho \in (0,1), \delta \in (0,1-\rho)$, let $\phi$ be a cut-off function defined by
	\[
		\phi(t)=
		\begin{cases}
		1, & \text{ if }t< \rho,\\
		\frac{\rho+\delta-t}{\delta}, & \text{ if }\rho \le t < \rho+\delta,\\
		0, & \text{ otherwise.}
		\end{cases}
	\]
	We choose $\varphi = \phi(|X|)X$ in \eqref{eq:1stVarG} and using $\left| \boldsymbol{H}_g\right|_g\le \mu$, $\left|D \theta\right|\le \mu$, we have
	\begin{equation}
		\left|\int_{ } \left[ \left< X, \nabla^g_S \phi \right> _g +\phi \mathrm{div}^g_S X  \right] \sqrt{\mathrm{det}g_S}dV^\theta(X,S)\right|\le C\mu \int_{ } |\phi||X| d\|V+W\|.
		\label{eq:pf1stVarL2}
	\end{equation}

	Note that we have
	\[
		\left< X, \nabla^g_S \phi \right> _g \sqrt{\mathrm{det}g_S}= X \cdot \phi' \frac{X^{\top _S}}{|X|} \sqrt{\mathrm{det}g_S}=-\frac{|X^{\top _S}|^2}{|X|\delta}\sqrt{\mathrm{det}g_S}\ge -(1+C\mu) \frac{|X^{\top _S}|^2}{|X|\delta},
	\]
	for any $|X| \in (\rho,\rho+\delta)$.
	Together with $\left|\mathrm{div}_S^gX \sqrt{\mathrm{det}g_S}-\mathrm{div}_SX\right|\le C\mu$ and \eqref{eq:pf1stVarL2}, we can obtain
	\[
		\int_{ } n\phi d\|V^\theta\|\le \frac{1+C\mu}{\delta} \int_{B_{\rho+\delta}\backslash B_\rho} \frac{|X^{\top _S}|^2}{|X|}dV^\theta(X,S)+C\mu(\|V\|(\mathcal{B}_{\rho+\delta})+\omega_n (\rho+\delta)^n).
	\]
	Thus, together $\|V\|(\mathcal{B}_\rho)\le C \rho^n$ by the monotonicity formula, we can take $\delta \rightarrow 0^+$ to get
	\begin{align*}
		n\|V^\theta\|(\mathcal{B}_\rho)\le{}& (1+C\mu)\rho\frac{d}{d\rho}\int_{B_{\rho}} \frac{|X^{\top _S}|^2}{|X|^2}dV^\theta(X,S)+C\mu(\|V\|(\mathcal{B}_{\rho})+\omega_n \rho^n)\\
		\le{}& (1+C\mu)\rho \frac{d}{d\rho}\|V^\theta\|(B_\rho)+C\mu \rho^n.
	\end{align*}
	for a.e. $\rho \in (0,1)$.
	Therefore, the above claim holds.

	Now, we use the monotonicity formula \eqref{eq:MonotonicityFormulaBoundary} with $\sigma \rightarrow 0^+$ and together with \eqref{eq:pfEstVolumeBall}, we have
	\begin{align}
		\frac{1}{\omega_n}\int_{ B_\rho} \frac{|X^{\bot_S} |^2}{|X|^{n+2}}dV(X,S)
		\le{}&
		(1+C\mu) \frac{1}{n \omega_n\rho^{n-1}}\frac{d}{d\rho}\|V^\theta\|(B_\rho)-\frac{1}{2}\Theta(\mathcal{V},0)+C\mu\nonumber \\
		\le{}&(1+C\mu) \frac{1}{n \omega_n\rho^{n-1}}\frac{d}{d\rho}\|V^\theta\|(B_\rho)-\Theta(\|\boldsymbol{C}\|,0)+C\mu.
		\label{eq:pfMonOtherForm}
	\end{align}
	Here, we use the fact $\frac{1}{2}\Theta(\mathcal{V},0)\ge  \Theta(\|\boldsymbol{C}\|,0)-C\mu$.

	We choose a smooth test function $\psi$ by
	\begin{equation}
		\psi(t)=
		\begin{cases}
		1, & t<\frac{27}{32} \\
		0, & t>\frac{7}{8}.
		\end{cases}
		\label{eq:pfTestPsi}
	\end{equation}
	Multiplying both sides of \eqref{eq:pfMonOtherForm} by $\psi^2(|X|)\rho^{n-1}$ and integrating, we obtain
	\begin{align*}
		&n\int_{ 0} ^1 \psi^2(\tau) \tau^{n-1}\int_{ B_\tau(0)} \frac{|X^{\bot_S} |^2}{|X|^{n+2}}dV(X,S)d\tau\\
		\le{}&
		\int_{ 0}^1 \psi^2(\tau)(1+C\mu) \frac{d}{d\tau}\|V^\theta\|(B_\tau(0))d\tau-\int_{0}^1n\omega_n\rho^{n-1}\psi^2(\tau)\Theta(\|\boldsymbol{C}\|,0)d\tau+C\mu\\
		={}& \int_{ } \psi^2(|X|)(1+C\mu)d\|V^\theta\|(X)-\int_{ } \psi^2(\left|X\right|)d\|\boldsymbol{C}\|(X)+C\mu \\
		\le{}& \int_{ } \psi^2(|X|)d\|V^\theta\|(X)-\int_{ } \psi^2(\left|X\right|)d\|\boldsymbol{C}\|(X)+C\mu.
	\end{align*}

	Consequently, we deduce	
	\begin{equation}
		C \int_{ B_{\frac{13}{16}}} \frac{|X^{\bot_S} |^2}{|X|^{n+2}}dV(X,S)\le
		\int_{ } \psi^2(|X|)d\|V^\theta\|(X)-\int_{ } \psi^2(|X|)d\|\boldsymbol{C}\|(X)+C\mu.
		\label{eq:pfXNormal}
	\end{equation}

	On the other hand, we set $\varphi=\psi^2(|X|)x$ in \eqref{eq:1stVarFormula}, yielding
	\[
		\left|\int_{ } \mathrm{div}_S \varphi dV^\theta(X,S)\right|\le C\mu.
	\]
	Here, the vector field $x=x_1e_1+x_2e_2$ at point $X=(x,y)=(x_1,x_2,y)$.
	
	Upon a simple calculation, we have
	\begin{align*}
		\mathrm{div}_S(x)={}&1+\sum_{j =1}^{n-1}|e_{2+j}^{\bot_S} |^2,\\
		\left|x^{\bot_S}  \cdot D_y \psi^2\right|\le{}& 2\psi|D_y \psi||x^{\bot_S} | \left( \sum_{j =3}^{n+1}|e_j^{{\bot_S}}|^2 \right)^{\frac{1}{2}}.
	\end{align*}

	Therefore, it follows that
	\begin{align*}
{} & \int_{ } \psi^2 \left( 1+ \sum_{j=3}^{n+1} |e_j^{\bot_S} |^2\right)dV^\theta(X,S)\le - \int_{ } D \psi^2 \cdot x^{\top_S} d V^\theta(X,S)+C\mu\\
		={}&-\int_{ } D_x \psi \cdot x^{\top_S}dV^\theta(X,S)+ \int_{ } D_y \psi^2 \cdot x^{\bot_S} dV^\theta(X,S)+C\mu\\
		\le{}& - 2\int_{ } \psi \psi' \frac{|x^{\top_S}|^2}{|X|}dV^\theta(X,S)+2 \int_{ } \psi \psi' |x^{\bot_S} | \left( \sum_{j=3}^{n+1} |e_j^{\bot_S} |^2\right)dV^\theta(X,S)+C\mu.
	\end{align*}
	Applying the Cauchy-Schwarz inequality results in
	\begin{align*}
		{} & \int_{ } \psi^2 \left( 1+\frac{1}{2}\sum_{j=3}^{n+1} |e_j^{\bot_S} |^2 \right)dV^\theta(X,S) \\
		\le{} & -2\int_{ } \psi \psi' \frac{|x^{\top_S}|^2}{|X|}dV^\theta(X,S)+2 \int_{ }  (\psi')^2 |x^{\bot_S} |^2dV^\theta(X,S)+C\mu \\
		\le{}& C \int_{ B_1\backslash B_{\frac{13}{16}}} (\psi^2+(\psi')^2)|x^{\bot_S} |^2 dV(X,S)-2\int_{ } \psi \psi' \frac{|x|^2}{|X|}dV^\theta(X,S)+C\mu.
	\end{align*}
	Here, we have used $\psi'=0$ when $|X|<\frac{13}{16}$.

	Considering \ref{it:L2Est1} from Theorem \ref{thm__l_2_estimate}, it's evident that
	\[
		|x^{\bot_S} |^2\le |u(x',y)|^2+|x'|^2\left|D u(x',y)\right|
	\]
	for a point $(x,y)$ on the graph of $u$, with $S$ being the tangent plane at $(x,y)$, and $(x',y)$ being the projection onto the half-plane $H$. Therefore,
	\begin{align*}
		{} & \int_{ } \psi^2\left( 1+\frac{1}{2}\sum_{j=3 }^{n+1}|e_j^{\bot_S} |^2 \right)dV^\theta(X,S) \\
		\le{} & C \int_{ U \cap B_{\frac{7}{8}}}( |u|^2+r^2\left|D u\right|^2)d \mathcal{H}^n+C \int_{ B_{\frac{7}{8}}\backslash \mathrm{graph}u} r^2d\|V\|(X)-2\int_{ \mathrm{graph}u } \psi \psi'\frac{r^2}{|X|}d\|V\|(X)\\
		{}& -C\int_{ \mathrm{spt}\|W\|\backslash U_W} r^2\cos \theta d\|W\|-2\int_{ U_W}\psi\psi' \frac{r^2}{|X|}\cos \theta d\|W\| +C\mu.
	\end{align*}

	On $\mathrm{graph}u$, we can rewrite the integral as
	\begin{align*}
		&\int_{ \mathrm{graph}u} \psi \psi' \frac{r^2}{|X|}d\|V\|(X)\\
		={}&\int_{ U_V} \psi(|X|+u^2(X)) \psi'(|X|^2+u^2(X))\frac{r^2(X+u(X))}{|X+u(X)|}\sqrt{1+|Du|^2}d\mathcal{H}^n(X)
	\end{align*}
	leading to
	\[
		\left|\int_{ \mathrm{graph}u} \psi \psi' \frac{r^2}{|X|}d\|V\|(X)-
		\int_{U_V} \psi \psi' \frac{r^2}{|X|}d\mathcal{H}^n(X)\right|\le C\int_{ U_V} (u^2+r^2|D u|^2)d \mathcal{H}^n.
	\]
	for some $C=C(n)$.
	Note that,
	\[
		\int_{ B_1} \psi^2 d\|\boldsymbol{C}\|(X)=-2\int_{ B_1} \frac{r^2}{|X|}\psi \psi' d\|\boldsymbol{C}\|(X),
	\]
	and considering $-\psi \psi' \frac{r^2}{|X|}\ge 0$ for any $X\neq 0$, we derive
	\begin{align*}
		{} & \frac{1}{2}\int_{B_{\frac{27}{32}} } \sum_{j=3 }^{n+1}|e_j^{\bot_S} |^2 dV(X,S)+
		\int_{ B_1} \psi^2 d\|V^\theta\|(X)-
		\int_{ B_1} \psi^2 d\|\boldsymbol{C}\|(X)\\
		\le{} &  C \int_{ U \cap B_{\frac{7}{8}}}( |u|^2+r^2\left|D u\right|^2)d \mathcal{H}^n+C \int_{ B_{\frac{7}{8}}\backslash \mathrm{graph}u} r^2d\|V\|(X)\\
			  &
		+\int_{ B_{\frac{7}{8}}\backslash U_W}r^2 d\|W\|(X) +C\mu .
	\end{align*}

	Along with \eqref{eq:pfXNormal} and Lemma \ref{lem_graph}, we establish
	\begin{equation}
		\int_{ B_{\frac{13}{16}}} \sum_{j=3 }^{n+1}|e_j^{\bot_S} |^2 dV(X,S)+
		\int_{ B_{\frac{13}{16}}(0)} \frac{|X^{\bot_S} |^2}{|X|^{n+2}}dV(X,S)\le
		C \mathcal{E}^2(\mathcal{V},H).
		\label{eq:pfENorm}
	\end{equation}

	At last, we choose $\varphi(X)=\zeta^2 |X|^{-n+\omega} \mathrm{dist}^2(X,\mathrm{spt}\|\boldsymbol{C}\|) \frac{X}{|X|^2}$ in \ref{eq:1stVarFormula} where
	$\zeta$ is the cut off function with $\zeta=1$ on $B_{\frac{3}{4}}(0)$ and $\zeta=0$ outside of $B_{\frac{13}{16}}(0)$. (Despite $\phi$ not being a $C^1$ vector field, such $\phi$ can still be utilized through an approximate argument.)

	Consequently, we arrive at
	\begin{align}
		&\int_{B_{\frac{3}{4}}}
		\frac{\mathrm{dist}^2(X,\mathrm{spt}\|\boldsymbol{C}\|)}{|X|^{n+2-\omega}}d\|V\|(X)\nonumber \\
		\le{}&
		C \left( \int_{ B_1} \zeta^2 \frac{|X^{\bot_S} |^2}{|X|^{n+2-\omega}}dV(X,S)+\int_{ } \frac{|D_S \zeta|^2 \mathrm{dist}^2(X,\mathrm{spt}\|\boldsymbol{C}\|)}{|X|^{n-\omega}}d\|V\|(X)+\mu(\mathcal{V}) \right).
	\end{align}

	Given that $D_S\zeta$ is supported in $\overline{B_{\frac{13}{16}}\backslash B_{\frac{3}{4}}}$, along with \eqref{eq:pfENorm}, it follows that
	\[
		\int_{ B_{\frac{3}{4}}(0)} \frac{\mathrm{dist}^2(X,\mathrm{spt}\|\boldsymbol{C}\|)}{|X|^{n+2-\omega}}d\|V^\theta\|(X)\le
		C  \mathcal{E}^2(\mathcal{V},H).
	\]
\end{proof}

\begin{corollary}
	\label{cor_noHole}
	For any $\rho \in (0,\frac{1}{4}]$ and $\omega \in (0,1)$, there exists a constant $\varepsilon=\varepsilon(n,\rho,\Lambda_0)$ such that if $(\mathcal{V},H)$ satisfies $(\Lambda_0,\varepsilon)$-Hypothesis, then for each $Z = (0,\zeta,\eta) \in \mathrm{spt}\|V\|\cap B_{\frac{5}{8}}\cap \partial \mathbb{H}^{n+1}$ with $\Theta(\mathcal{V},Z)\ge 1-\cos \theta(Z)$, we have
	\begin{equation}
		|\zeta|^2\le C\mathcal{E}^2(\mathcal{V},H), 
		\label{eq:corDistSing}
	\end{equation}
	for some $C=C(n,\Lambda_0)$, and
	\begin{equation}
	\int_{ B_{\frac{\rho}{2}}(Z)} \frac{\mathrm{dist}^2(X,\mathrm{spt}\|\boldsymbol{C}_Z\|)}{|X-Z|^{n+2-\omega}}d\|V^\theta\|(X)\le
	\frac{C}{\rho^{n+2-\omega}} \int_{ B_\rho(Z)} \mathrm{dist}^2(X,\mathrm{spt}\|\boldsymbol{C}_Z\|)d\|V^\theta\|(X)+C\mu \rho^{\omega},
	\label{eq:corL2ImproveExcess}
\end{equation}
for some $C=C(n,\Lambda_0,\omega)$.
Here, the cone $\boldsymbol{C}_Z$ defined as
\[
	\boldsymbol{C}_Z:=(\tau_{-Z})_{\#}(\boldsymbol{C}).
\]
\end{corollary}



\begin{proof}
	We begin with a straightforward geometric observation.
	For any point $Z=(0,\zeta,\eta)$ with $|\zeta|\le \frac{\rho}{8}$, if a point $X \in \mathbb{H}^{n+1}$ is such that $r(X)\ge \frac{\rho}{2}$ and $\mathrm{dist}(X,H)\le \frac{\rho}{8}$, then
	\[
		\mathrm{dist}(X,\mathrm{spt}\|\boldsymbol{C}_H\|)=\mathrm{dist}(X,H),\quad \mathrm{dist}(X,\|(\tau_{-Z})_\#\boldsymbol{C}_H\|)=\mathrm{dist}(X,(\tau_{-Z})_\#H).
	\]
	Therefore, we have
	\begin{equation}
		\mathrm{dist}(X,\mathrm{spt}\|\boldsymbol{C}_H\|)+\mathrm{dist}(X,\|(\tau_{-Z})_\#\boldsymbol{C}_H\|)\ge |H^\bot (Z)|.
		\label{eq:pfL2Dist}
	\end{equation}

	If we choose $\tau=\frac{\rho}{8}$ in Theorem \ref{thm__l_2_estimate} and choose $\varepsilon=\varepsilon(n,\rho,\Lambda_0)$ small enough, alongside with Theorem \ref{thm_allard}, for any $Z=(0,\zeta,\eta) \in \mathrm{spt}\|V\|\cap B_{\frac{5}{8}}\cap \left\{ \Theta(\mathcal{V},X)\ge 1-\cos \theta(Z) \right\}$, we have $r(Z)\le \frac{\rho}{8}$.
	For any $X \in \mathrm{spt}\|V\|\cap B_\rho(Z)\backslash \left\{ r<\frac{\rho}{2} \right\}$, we get $\mathrm{dist}(X,H)\le \frac{\rho}{8}$.

	From \eqref{eq:pfL2Dist}, noting that $|H^\bot (Z)|\ge C|\zeta|$ for some $C=C(\Lambda_0)$, by choosing $\varepsilon_0=\varepsilon_0(n,\rho,\Lambda_0)$ to be sufficiently small, we deduce
	\[
		\mathrm{dist}(X,\mathrm{spt}\|\boldsymbol{C}_H\|)+\mathrm{dist}(X,\|(\tau_{-Z})_\#\boldsymbol{C}_H\|)\ge C|\zeta|,
	\]
	for any $X \in \mathrm{spt}\|V\|\cap B_\rho(Z) \backslash \left\{ r<\frac{\rho}{2} \right\}$ for some constant $C=C(\Lambda_0)$.

	Note that
	$\|V\|(B_\rho(Z)\backslash \left\{ r<\frac{\rho}{2} \right\})\ge C \rho^n$ as $|\zeta|\le \frac{\rho}{8}$, 
	we find that
	\begin{equation}
		|\zeta|^2\le
		\frac{C}{\rho^n}\left(  \int_{ B_\rho(Z)} \mathrm{dist}^2(X,\mathrm{spt}\|\boldsymbol{C}_H\|)d\|V\|(X)+ \int_{ B_\rho(Z)} \mathrm{dist}^2(X,\mathrm{spt}\|(\tau_{-Z})_\#\boldsymbol{C}_H\|)d\|V\|(X)\right),
		\label{eq:pfEstShiftSing}
	\end{equation}
	for some $C=C(\Lambda_0)$.

	Using Lemma \ref{lem_change_center} and choosing $\varepsilon=\varepsilon(n,\Lambda_0)$ small enough allows us to apply Theorem \ref{thm__l_2_estimate} with $(\mathcal{V}_{Z,\frac{1}{3}},L_X^g(H))$ in place of $(\mathcal{V},H)$, which leads to
	\begin{align}
		&\frac{1}{(4\rho)^{n+2-\omega}}
		\int_{ B_{4\rho}(0)} \mathrm{dist}^2(X,\mathrm{spt}\|\boldsymbol{C}_{L_Z^g(H)}\|)d\|(\Pi^g_{Z,\frac{1}{3}})_{\#}V^\theta\|(X)\nonumber \\
		\le{}&C\left( \int_{ B_{1}(0)} \mathrm{dist}^2(X,\mathrm{spt}\|\boldsymbol{C}_{L_Z^g(H)}\|)d\|(\Pi^g_{Z,\frac{1}{3}})_{\#}V^\theta\|(X) +\mu(\mathcal{V})\right)\nonumber \\
		\le{}& C \left( \int_{ B_{\frac{1}{3}(1+C\mu)}(Z)} \mathrm{dist}^2(\Pi^g_{Z,\frac{1}{3}}(X),\mathrm{spt}\|(L_Z^g)_{\#}\boldsymbol{C}_H\|)|J_S\Pi^g_{Z,\frac{1}{3}}|d\|V^\theta\|(X)+\mu(\mathcal{V}) \right)\nonumber  \\
		\le{}& C \left( \int_{ B_1(0)} (1+C\mu)\mathrm{dist}^2( \eta_{Z,\frac{1}{3}}(X),\mathrm{spt}\|\boldsymbol{C}_H\| )d\|V^\theta\|(X)+\mu(\mathcal{V}) \right) \nonumber \\
		\le{}& C \left( \int_{ B_1(0)} \mathrm{dist}^2(X-Z,\mathrm{spt}\|\boldsymbol{C}_H\|)d\|V^\theta\|(X)+\mu(\mathcal{V}) \right) \nonumber \\
		\le{}& C \int_{ B_1(0)} \mathrm{dist}^2(X,\mathrm{spt}\|\boldsymbol{C}_H\|)d\|V^\theta\|(X)+C\mu(\mathcal{V})+C |\zeta|^2,
		\label{eq:pfLongEstOfL2}
	\end{align}
	for some $C=C(n,\Lambda_0,\omega)$, where we have
	used \eqref{eq:propLcloseG} and \eqref{eq:ballCompLXg} here.

	On the other hand, using a similar computation, if we choose $\varepsilon=\varepsilon(n,\Lambda_0)$ small enough, we can obtain
	\begin{align}
		{} & \int_{ B_{4\rho}(0)} \mathrm{dist}^2(X,\mathrm{spt}\|\boldsymbol{C}_{L_Z^g(H)}\|)d\|(\Pi_{Z,\frac{1}{3}}^g)_{\#}V^\theta\|(X) \nonumber \\
		\ge{} & (1-C\mu)\int_{ B_{\frac{4}{3}\rho(1-C\mu)}(Z)} \mathrm{dist}^2(X-Z,\mathrm{spt}\|\boldsymbol{C}_H\|)d\|V^\theta\|(X)\nonumber \\
		\ge{} & (1-C\mu)\int_{ B_{\frac{4}{3}\rho(1-C\mu)}(Z)} \mathrm{dist}^2(X,\mathrm{spt}\|(\tau_{-Z})_{\#}\boldsymbol{C}_H\|)d\|V^\theta\|(X)\nonumber \\
		\ge{} & \frac{1}{2}\int_{ B_{\rho}(Z)} \mathrm{dist}^2(X,\mathrm{spt}\|(\tau_{-Z})_{\#}\boldsymbol{C}_H\|)d\|V^\theta\|(X).
		\label{eq:pfLongEstLeft}
	\end{align}
	
	Together with \eqref{eq:pfEstShiftSing}, there exists $\varepsilon=\varepsilon(n,\rho,\Lambda_0)$ small enough such that
\begin{align}
	|\zeta|^2\le{}& \frac{C}{\rho^n} \left[ 
	\int_{ B_1(0)} \mathrm{dist}^2(X,\mathrm{spt}\|\boldsymbol{C}_H\|)d\|V^\theta\|(X)\right.\nonumber \\
&\left.+\rho^{n+2-\omega}\int_{ B_1(0)} \mathrm{dist}^2(X,\mathrm{spt}\|\boldsymbol{C}_H\|)d\|V^\theta\|(X)+\rho^{n+2-\omega}\mu(\mathcal{V}) + \rho^{n+2-\omega}|\zeta|^2\right] \nonumber \\
		(1-C\rho^{2-\omega})|\zeta|^2\le{}& 
		\frac{C}{\rho^n} \int_{ B_1(0)} \mathrm{dist}^2(X,\mathrm{spt}\|\boldsymbol{C}_H\|)d\|V^\theta\|(X)+C\rho^{n+2-\omega}\mu(\mathcal{V})
		\label{eq:pfEstZeta}
\end{align}
for some $C=C(n,\Lambda_0,\omega)$.
Notably, the constant $C$ does not depend on $\rho$.
Therefore, by choosing $\omega=\frac{1}{2}$ and a sufficiently small $\rho_0=\rho_0(n,\Lambda_0)$, the coefficient $(1-C\rho_0^{\frac{3}{2}})$ of $|\zeta|^2$ becomes greater than $\frac{1}{2}$.
For such fixed $\rho_0$, we choose $\varepsilon$ small enough such that \eqref{eq:pfEstZeta} holds with $\rho=\rho_0$.
This leads to $|\zeta|^2\le C\mathcal{E}^2(\mathcal{V},H)$ for some $C=C(n,\Lambda_0)$.

Based on Lemma \ref{lem_change_center} again, given $\rho \in (0,\frac{5}{24}]$, by choosing $\varepsilon=\varepsilon(n,\rho,\Lambda_0)$ small enough, we can apply Theorem \ref{thm__l_2_estimate} with $(\mathcal{V}_{Z,\rho},L_Z^g(H))$ in place of $(\mathcal{V},H)$ to get,
\begin{align}
	&\int_{ B_{\frac{3}{4}}} \frac{\mathrm{dist}^2(X,\mathrm{spt}\|\boldsymbol{C}_{L_Z^g(H)}\|)}{|X|^{n+2-\omega}}d\|(\Pi_{Z,\rho}^g)_{\#}V^\theta\|(X)\nonumber \\
	\le{}&
	C\left( \int_{ B_1} \mathrm{dist}^2(X,\mathrm{spt}\|\boldsymbol{C}_{L_Z^g(H)}\|)d\|(\Pi_{Z,\rho}^g)_{\#}V^\theta\|+\mu(\mathcal{V})\right)
	\label{eq:pfCorChangeCenter}
\end{align}
Following a similar argument to that used in deriving \eqref{eq:pfLongEstOfL2} and \eqref{eq:pfLongEstLeft}, it can be shown that
\begin{align*}
	{} &  \int_{ B_1} \mathrm{dist}^2(X,\mathrm{spt}\|\boldsymbol{C}_{L_Z^g(H)}\|)d\|(\Pi_{Z,\rho}^g)_{\#}V^\theta\|(X)\\
	\le{} & \int_{ B_{(1+C\mu)\rho}(Z)} \frac{1+C\mu}{\rho^{n+2}}\mathrm{dist}^2(X,\mathrm{spt}\|\boldsymbol{C}_Z\|)d\|V^\theta\|(X),\\
	{} & \int_{ B_{\frac{3}{4}}} \frac{\mathrm{dist}^2(X,\mathrm{spt}\|\boldsymbol{C}_{L_Z^g(H)}\|)}{|X|^{n+2-\omega}}d\|(\Pi_{Z,\rho}^g)_{\#}V^\theta\|(X) \\
	\ge{} & \int_{ B_{\frac{3}{4}(1-C\mu)\rho}(Z)} \frac{1-C\mu}{\rho^{\omega}} \frac{\mathrm{dist}^2(X,\mathrm{spt}\|\boldsymbol{C}_Z\|)}{|X-Z|^{n+2-\omega}}d\|V^\theta\|(X).
\end{align*}

Therefore, by choosing $\varepsilon=\varepsilon(n,\Lambda_0,\rho)$ small enough, we establish
\[
	\int_{ B_{\frac{3}{5}\rho}(Z)} \frac{\mathrm{dist}^2(X,\mathrm{spt}\|\boldsymbol{C}_Z\|)}{|X-Z|^{n+2-\omega}}d\|V^\theta\|(X)\le
	\frac{C}{\rho^{n+2-\omega}} \int_{ B_{\frac{6}{5}\rho}(Z)} \mathrm{dist}^2(X,\mathrm{spt}\|\boldsymbol{C}_Z\|)d\|V^\theta\|(X)+C\mu \rho^{\omega},
\]
for any $\rho \in (0,\frac{5}{24}]$ and some $C=C(n,\Lambda_0,\omega)$.
Hence, \eqref{eq:corL2ImproveExcess} holds for any $\rho \in (0,\frac{1}{4}]$.
\end{proof}
\begin{lemma}
	\label{lem_noHole}
	Given $\delta \in (0,\frac{1}{16})$. Suppose $(\mathcal{V},H)$ satisfies $(\Lambda_0,\varepsilon)$-Hypothesis for some $\varepsilon=\varepsilon(n,\delta,\Lambda_0)$ small enough.
	Then,
	\begin{enumerate}[\normalfont(a)]
		\item $B_\delta(0,y)\cap \left\{ \Theta(\mathcal{V},Z)\ge 1-\cos \theta(Z) \right\}\cap \partial \mathbb{H}^{n+1}\neq \emptyset $ for each $(0,y) \in \left\{ 0 \right\}\times B_{\frac{1}{2}}^{n-1}(0)$.
		\item For $\omega \in (0,1)$ and $\sigma \in [\delta,\frac{1}{16})$, we have
			\begin{equation}
				\int_{ B_{\frac{1}{2}}(0)\cap \left\{ r<\sigma \right\}} \mathrm{dist}^2(X,\mathrm{spt}\|\boldsymbol{C}_H\|)d\|V^\theta\|(X)\le
				C \sigma^{1-\omega}
				\mathcal{E}^2(\mathcal{V},H),
				\label{eq:lemNoHole}
			\end{equation}
			where $C=C(n,\Lambda_0,\omega)$.
	\end{enumerate}
\end{lemma}
\begin{proof}
	Assuming the first statement is incorrect, then there exists $\delta \in (0,\frac{1}{16})$ and a sequence of $(\mathcal{V}_k,H_k)=((V_k,W_k,\theta_k,g_k),H_k)$ satisfying $(\Lambda_0,\varepsilon_k)$-hypothesis for some $\varepsilon_k \rightarrow 0^+$, $\mathcal{V}_k\rightarrow \mathcal{V}=(V,W,\theta,\delta)$, $H_k\rightarrow H$ and $(0,y_i)\rightarrow (0,y) \in \left\{ 0 \right\}\times B^{n-1}_{\frac{2}{3}}(0)$ such that
	\[
		B_\delta(0,y_i)\cap \left\{ \Theta(\mathcal{V}_i,Z)\ge 1-\cos \theta_i(Z)\right\}\cap \partial \mathbb{H}^{n+1}=\emptyset .
	\]
	Based on Theorem \ref{thm:smallDensity},
	$V$ is a stationary integral varifold in $B_{\frac{\delta}{2}}(0,y)$ in free boundary sense.
	The constancy theorem implies $\mathrm{spt}\|V\|\subset H$, which leads $H$ to be orthogonal to $\partial \mathbb{H}^{n+1}$.
	This contradicts with $H=H^{\theta}$ for $\theta \in [\frac{\pi}{2}+\Lambda_0,\pi-\Lambda_0]$.

	For the second statement,
	we choose $Z=(0,\zeta,\eta) \in B_{\frac{5}{8}}\cap \mathrm{spt}\|V\|$ with $\Theta(\mathcal{V},Z)\ge 1-\cos \theta(Z)$.
	By Corollary \ref{cor_noHole}, if we choose $\varepsilon=\varepsilon(n,\Lambda_0)$ small enough, we deduce
	\[
		\int_{ B_{\frac{1}{8}}(Z)} \frac{\mathrm{dist}^2(X,\mathrm{spt}\|\boldsymbol{C}_Z\|)}{|X-Z|^{n+2-\omega}}d\|V^\theta\|(X)\le C \int_{ B_{\frac{1}{4}}(Z)} \mathrm{dist}^2(X,\mathrm{spt}\|\boldsymbol{C}_Z\|)d\|V^\theta\|(X)+C\mu(\mathcal{V}).
	\]

	Using the inequality $\frac{1}{|X-Z|^2} \ge 64$ for $X \in B_{\frac{1}{8}}(Z)$, triangle inequality $\frac{1}{2}\mathrm{dist}^2(X,\mathrm{spt}\|\boldsymbol{C}\|)-|\zeta|^2\le \mathrm{dist}^2(X,\mathrm{spt}\|\boldsymbol{C}_Z\|)\le 2\mathrm{dist}^2(X,\mathrm{spt}\|\boldsymbol{C}\|)+2|\zeta|^2$, we establish
	\begin{align*}
		{}&\frac{1}{2}\int_{ B_{\frac{1}{8}}(Z)} \frac{\mathrm{dist}^2(X,\mathrm{spt}\|\boldsymbol{C}\|)}{|X-Z|^{n-\omega}}d\|V^\theta\|(X)-|\zeta|^2 \int_{ B_{\frac{1}{8}}(Z)} d\|V^\theta\|(X)\\
		\le{}& \int_{ B_{\frac{1}{8}}(Z)} \frac{\mathrm{dist}^2(X,\mathrm{spt}\|\boldsymbol{C}_Z\|)}{|X-Z|^{n-\omega}}d\|V^\theta\|(X)\le \frac{1}{64} \int_{ B_{\frac{1}{8}}(Z)} \frac{\mathrm{dist}^2(X,\mathrm{spt}\|\boldsymbol{C}_Z\|)}{|X-Z|^{n+2-\omega}} d\|V^\theta\| \\
		\le{}& C\int_{ B_{\frac{1}{4}}(Z)} \mathrm{dist}^2(X,\mathrm{spt}\|\boldsymbol{C}_Z\|)d\|V^\theta\|(X)+C\mu(\mathcal{V}) \\
		\le{}& C\int_{ B_{\frac{1}{4}}(Z)} \mathrm{dist}^2(X,\mathrm{spt}\|\boldsymbol{C}\|)d\|V^\theta\|(X)+C|\zeta|^2 \int_{ B_{\frac{1}{4}}(Z)} d\|V^\theta\|(X) +C\mu(\mathcal{V}).
	\end{align*}
	This implies
	\begin{align}
		&\int_{ B_{\frac{1}{8}}(Z)} \frac{\mathrm{dist}^2(X,\mathrm{spt}\|\boldsymbol{C}\|)}{|X-Z|^{n-\omega}}d\|V^\theta\|(X)\nonumber\\
		\le{}&
		C\int_{ B_{\frac{1}{4}}(Z)} \mathrm{dist}^2(X,\mathrm{spt}\|\boldsymbol{C}\|)d\|V^\theta\|(X)+C|\zeta|^2 \|V^\theta\|(B_{\frac{1}{4}}(Z))+C\mu(\mathcal{V})\nonumber \\
		\le{}& C \int_{ B_1} \mathrm{dist}^2(X,\mathrm{spt}\|\boldsymbol{C}\|)d\|V^\theta\|(X)+C\mu(\mathcal{V}).
		\label{eq:pfL2DistDecay}
	\end{align}
	in view of \eqref{eq:corDistSing}.

	For any $\sigma \in [\delta,\frac{1}{16})$, and $(0,y) \in B_{\frac{1}{2}}$, by choosing $\varepsilon=\varepsilon(n,\delta,\Lambda_0)$ small enough, based on the first part of this lemma, we find that $Z=(0,\zeta,\eta) \in B_\delta(0,y)\subset B_\sigma(0,y)$ with $\Theta(\mathcal{V},Z)\ge 1-\cos \theta(Z)$. Employing \eqref{eq:pfL2DistDecay}, we obtain,
	\begin{align*}
		&\frac{1}{\sigma^{n-\omega}}\int_{ B_\sigma(0,y)}\mathrm{dist}^2(X,\mathrm{spt}\|\boldsymbol{C}\|)d\|V^\theta\|(X) \le \frac{C}{(2\sigma)^{n-\omega}} \int_{ B_{2\sigma}(Z)} \mathrm{dist}^2(X,\mathrm{spt}\|\boldsymbol{C}\|)d\|V^\theta\|(X) \\
		\le{} & C \int_{ B_{\frac{1}{8}}(Z)} \frac{\mathrm{dist}^2(X,\mathrm{spt}\|\boldsymbol{C}\|)}{|X-Z|^{n-\omega}} d\|V^\theta\|(X)
		\le{} C \int_{ B_1} \mathrm{dist}^2(X,\mathrm{spt}\|\boldsymbol{C}\|)d\|V^\theta\|(X)+C\mu(\mathcal{V}).
	\end{align*}
	
	Now,using at most $C=C(n)\sigma^{n-1}$ balls $B_{2\sigma}(0,y)$ with $(0,y) \in B_{\frac{1}{2}}$ to cover $B_{\frac{1}{2}}\cap \left\{ r<\sigma \right\}$, we get
	\[
		\frac{\sigma^{n-1}}{\sigma^{n-\omega}}
		\int_{ B_{\frac{1}{2}}\cap  \left\{ r<\sigma \right\}} \mathrm{dist}^2(X,\mathrm{spt}\|\boldsymbol{C}\|)d\|V^\theta\|(X)\le
		C \int_{ B_1} \mathrm{dist}^2(X,\mathrm{spt}\|\boldsymbol{C}\|)d\|V^\theta\|(X)+C\mu(\mathcal{V}),
	\]
	which gives us \eqref{eq:lemNoHole}.
\end{proof}

\section{Properties of Blow-ups}%
\label{sec:blow_up_argument}

We fix a constant $\Lambda_0 \in (0,1)$, and consider a sequence of VPCA-quadruples $\left\{ \mathcal{V}_k=(V_k,W_k,\theta_k,g_k) \right\}$, a sequence of half hyperplanes  $\left\{ H_k \right\}$, together with a positive sequence $\left\{ \varepsilon_k \right\}$ with $\lim_{k\rightarrow +\infty} \varepsilon_k=0$.
We assume $(\mathcal{V}_k,H_k)$ satisfies $(\Lambda_0,\varepsilon_k)$-hypothesis for each $k\ge 1$.
We write $\boldsymbol{C}_k=\boldsymbol{C}_{H_k}$.
Based on Theorem \ref{thm_compactnessRIV}, we assume $\mathcal{V}_k\rightarrow \mathcal{V}=(V,W,\theta,\delta)$ and $H_k\rightarrow H$ as $k\rightarrow +\infty$.

We denote
\[
	\mathcal{E}_k:=\mathcal{E}(\mathcal{V}_k,H_k).
\]

We choose $\left\{ \delta_k \right\}, \left\{ \tau_k \right\}$ be two sequences of positive decreasing numbers converging to 0.
Based on the $L^2$ estimate theorem (Theorem \ref{thm__l_2_estimate}) and Lemma \ref{lem_noHole}, we obtain that,
\begin{enumerate}[\normalfont(a)]
	\item (Theorem \ref{thm__l_2_estimate}) $V_k\lfloor(B_{\frac{13}{16}}\backslash \{ r<\tau_k \})=\left|\mathrm{graph}u_k \cap B_{\frac{13}{16}}\backslash \{ r<\tau_k \}\right|$ where $u_k \in C^{1,\beta}(B_{\frac{13}{16}}\cap H_k \backslash \{ r<\frac{\tau_k}{2} \},H_k^\bot )$ for some $\beta=\beta(n,\Lambda_0) \in (0,1)$ and satisfies $\mathrm{dist}(X+u_k(X),\mathrm{spt}\|\boldsymbol{C}\|)=|u_k(X)|$ for $X \in B_{\frac{13}{16}}\cap H_k \backslash \{ r<\frac{\tau_k}{2} \}$.
		$W_k\lfloor(B_{\frac{13}{16}}\backslash \{ r<\tau_k \})=|B_{\frac{13}{16}}\cap \{ x_2< -\tau \}\cap \partial \mathbb{H}^{n+1}|$.
		\label{it:graphK}
	\item (Corollary \ref{cor_noHole}) For each $Z=(0,\zeta,\eta) \in \mathrm{spt}\|V^{\theta_k}_k\|\cap B_{\frac{5}{8}}$ with $\Theta(\mathcal{V}_k,Z)\ge 1-\cos \theta_k(Z)$, we have
		\begin{equation}
			|\zeta|\le C\mathcal{E}_k.
			\label{eq:BlowUpSing}
		\end{equation}
		\label{it:BlowUpSingK}
	\item (Lemma \ref{lem_noHole}) For each $(0,y) \in \{ 0 \}\times \mathbb{R}^{n-1}\cap B_{\frac{1}{2}}$, we have
		\begin{equation}
			B_{\delta_k}(0,y)\cap \{ \Theta(\mathcal{V}_k,Z)\ge 1-\cos \theta_k(Z)\}\cap \partial \mathbb{H}^{n+1}\neq \emptyset.
			\label{eq:BlowUpNoHole}
		\end{equation}
		and
		\begin{equation}
			\int_{ B_{\frac{1}{2}}\cap \{ r<\sigma \}} \mathrm{dist}^2(X,\mathrm{spt}\|\boldsymbol{C}_k\|)d\|V^{\theta_k}_k\|(X)\le C\sqrt{\sigma}\mathcal{E}_k^2,
			\label{eq:BlowUpNonConcen}
		\end{equation}
		for any $\sigma \in (\delta_k,\frac{1}{16}]$ where the constant $C$ is independent of $\sigma$.
		\label{it:BlowUpNoHoleK}
	\item (Corollary \ref{cor_noHole}) For each $\omega\in (0,1)$, $\rho \in (0,\frac{1}{4})$, and $Z=(0,\zeta,\eta) \in \mathrm{spt}\|V^{\theta_k}_k\|\cap B_{\frac{5}{8}}$ such that $\Theta(\mathcal{V}_k,Z)\ge 1-\cos \theta_k(Z)$, we have
		\begin{align}
			{} & \int_{ B_{\frac{\rho}{4}}(Z)\cap H_k \backslash \{ r<\tau_k \}} \frac{\left|u_k-\zeta \sin \theta_k(0)\vec{n}_k \right|^2}{|X+u_k-Z|^{n+2-\omega}}d\mathcal{H}^n(X) \nonumber \\
			\le{} & \frac{C}{\rho^{n+2-\omega}} \int_{ B_\rho(Z)} \mathrm{dist}^2(X,\mathrm{spt}\|(\tau_{-Z})_{\#}\boldsymbol{C}_k\|)d\|V^{\theta_k}_k\|(X)+C\mu(\mathcal{V}_k) \rho^{\omega},
			\label{eq:BlowUpImproveL2}
		\end{align}
		for $k$ sufficient large, where $C=C(n,\Lambda_0,\omega)$, and $\vec{n}_k$ represents the upward-pointing unit normal vector of $H_k$.
		\label{it:BlowUpImproveL2K}
\end{enumerate}

Let $l_k$ be the real number such that $H_k$ can be written as
\[
	H_k=\{ X+l_kr(X)\vec{n}:X \in H \},
\]
where $\vec{n}$ is the upward-pointing unit normal vector of $H$.
In other words, we regard $H_k$ as a graph of function $X\rightarrow l_kr(X)$ over $H$.

We define $\tilde{u}_k$ by
\begin{equation}
	\tilde{u}_k(X)=u_k(X+l_kr(X)).
	\label{eq:defUtil}
\end{equation}

Note that $\tilde{u}_k$ is well-defined on $B_{\frac{3}{4}}\cap H\backslash \{ r<2\tau_k \}$ for $k$ sufficient large.
We set $\tilde{u}_k=0$ on $B_{\frac{3}{4}}\cap H\cap \{ r<2 \tau_k \}$.
It follows that
\begin{equation}
	\|\tilde{u}_k\|^2_{C^{1,\beta}(B_{\frac{3}{4}}\cap H\backslash \left\{ r<\tau\right\})}\le C
	\mathcal{E}_k^2
	\label{eq:estBlowUpC1a}
\end{equation}
for any $\tau \in (0,\frac{1}{2})$ and any $k$ with $2\tau_k<\tau$ where $C=C(n,\Lambda_0,\tau)$ in view of Theorem \ref{thm_allard}.

Therefore, up to a subsequence, we assume
\begin{equation}
	\mathcal{E}_k^{-1}\tilde{u}_k\rightarrow v
	\label{eq:defBlowUp}
\end{equation}
for some $v \in C^{1,\beta}_{\mathrm{loc}}(B_{\frac{3}{4}}\cap H,H^\bot )$, where the convergence is in $C^{1,\beta}_{\mathrm{loc}}(B_{\frac{3}{4}}\cap H,H^\bot )$.
Since $u_k$ solves the minimal surface equation weakly as a graph over plane $H_k$ under metric $g_k$, and using $g_k \rightarrow \delta$ in $C^1$ sense, $H_k\rightarrow H$, and $l_k\rightarrow 0$, we find $v$ is a harmonic function.

Further, from \eqref{eq:BlowUpNonConcen}, and estimate \eqref{eq:estBlowUpC1a}, for any $\sigma \in (0,\frac{1}{16})$, we see that
\[
	\int_{ B_{\frac{3}{8}}\cap H\cap \left\{ r(X)<\sigma \right\}} |\tilde{u}_k|^2\le C \sqrt{\sigma}\mathcal{E}_k^2
\]
for $k$ sufficient large and $C=C(n,\Lambda_0)$.
his establishes that $\tilde{u}_k \rightarrow v$ in $L^2(B_{\frac{3}{8}}\cap H,H^\bot )$ as well.

\begin{definition}
	
	We say a sequence $\left\{ (\mathcal{V}_k,H_k) \right\}$ is a \textit{blow-up sequence} if $(\mathcal{V}_k,H_k)$ satisfies $(\Lambda_0,\varepsilon_k)$-hypothesis for some $\varepsilon_k \rightarrow 0^+$.

	Given a blow-up sequence, $\left\{( \mathcal{V}_k,H_k )\right\}$, we say $v \in C^{1,\beta}_{\mathrm{loc}}(B_{\frac{3}{8}}\cap H)\cap L^2(B_{\frac{3}{8}}\cap H)$ is a \textit{blow-up} if $v$ can be obtained by the limit \eqref{eq:defBlowUp} described above.
\end{definition}

The following part of this section is devoted to proving that any blow-up $v$ has sufficient regularity.
This part is similar to the argument in \cite[Section 12]{Wickramasekera2014Regularity}.
But for the sake of completeness, we give a detailed proof here.

\begin{proposition}
	\label{prop_HolderBlowUp}
	If $v$ is a blow-up, then $v \in C^{0,\alpha}(\overline{B_{\frac{3}{16}}\cap H},H^\bot )$ with the following estimate
	\[
		\|v\|_{C^{0,\alpha}(\overline{B_{\frac{3}{16}}\cap H},H^\bot )}^2\le
		C \int_{ B_{\frac{3}{8}}\cap H} |v(X)|^2d\mathcal{H}^n(X),
	\]
	where $\alpha=\alpha(n,\Lambda_0) \in (0,1)$ and $C=C(n,\Lambda_0)$.
\end{proposition}

\begin{proof}
	For each $Y=(0,y) \in B_{\frac{3}{16}}$, by \eqref{eq:BlowUpNoHole}, we can find $Z_k=(0,\zeta_k, \eta_k) \in \mathrm{spt}\|V_k\|\cap B_{\frac{3}{4}}$ with $\Theta(\mathcal{V}_k,Z_k)\ge 1-\cos\theta_k(Z_k)$ and $Z_k\rightarrow (0,y)$ as $k\rightarrow +\infty$.
	According to \eqref{eq:BlowUpSing}, the ratio $\frac{\zeta_k}{\mathcal{E}_k}$ is uniformly bounded, allowing us to define
	\begin{equation}
		\kappa(y):=\lim_{k\rightarrow +\infty} \frac{\zeta_k}{\mathcal{E}_k}.
		\label{eq:defKappa}
	\end{equation}

	Note that the definition of $\kappa$ does not rely on the choice of the sequence $\left\{ Z_k \right\}$.
	Indeed, for any fixed $\rho \in (0,\frac{1}{8}]$ and any $\tau \in (0,\frac{1}{8})$, employ estimate \eqref{eq:BlowUpImproveL2}, we obtain
	\begin{equation}
		\int_{ B_{\frac{\rho}{4}}(Y)\cap H} \frac{|v-\kappa(y) \sin\theta\vec{n} |^2}{|X-Y|^{n+2-\omega}}d\mathcal{H}^n(X)\le \frac{C}{\rho^{n+2-\omega}}
		\int_{ B_\rho(Y)\cap H} |v-\kappa(y)\sin\theta\vec{n}|^2d\mathcal{H}^n(X),
		\label{eq:pfL2BlowUpDecay}
	\end{equation}
	where the constant $C$ does not depend on $\rho$. 
	Should there be another sequence $\left\{ Z_l \right\}$ resulting in $\kappa'(y)=\lim_{l\rightarrow +\infty} \frac{\zeta_l}{\mathcal{E}_l}$, then $\eqref{eq:pfL2BlowUpDecay}$ holds with $\kappa'(y)$ in place of $\kappa(y)$.
	Since the right-hand side of \eqref{eq:pfL2BlowUpDecay} is bounded by a constant which does not rely on the choice of the sequence $\left\{ Z_k \right\}$ in view of \eqref{eq:BlowUpSing}, using Cauchy-Schwarz inequality, we get
	\begin{equation}
		\int_{ B_{\frac{\rho}{8}}(Y)\cap H} \frac{(\kappa(y)-\kappa'(y))^2\sin ^2 \theta}{|X-Y|^{n+2-\omega}}d\mathcal{H}^n(X)\le C.
		\label{eq:pfNotRelyOnZk}
	\end{equation}
	But the left-hand side of \eqref{eq:pfNotRelyOnZk} is finite if and only if $\kappa(y)=\kappa'(y)$.
	Consequently, $\kappa(y)$ is well-defined.
	We define $v(Y)=\kappa(y) \sin\theta\vec{n}$ for any $Y=(0,y) \in B_{\frac{3}{16}}$.

	Now, we claim
	\begin{equation}
		\frac{1}{\rho^n}\int_{ B_\rho(X_0)\cap H} |v(X)-v(X_0)|^2d\mathcal{H}^n(X)\le C\rho^{2\alpha}\int_{ B_{\frac{3}{4}}\cap H} |v(X)|^2d\mathcal{H}^n(X),\quad \forall X_0 \in B_{\frac{1}{4}}\cap H,
		\label{eq:pfClaim}
	\end{equation}
	for some $\alpha \in (0,1)$.
	We set $\omega=\frac{1}{2}$ in \eqref{eq:pfL2BlowUpDecay} and then
	\begin{equation}
		\frac{1}{\sigma^n}\int_{ B_\sigma(Y)\cap H} |v(X)-v(Y)|^2d\mathcal{H}^n(X)\le C \left( \frac{\sigma}{\rho} \right)^{\frac{3}{2}} \frac{1}{\rho^n}\int_{ B_\rho(Y)\cap H} |v(X)-v(Y)|^2d\mathcal{H}^n(X).
		\label{eq:pfIter}
	\end{equation}
	for any $0<\sigma\le \frac{\rho}{4}\le \frac{1}{32}$.


	The key idea is to employ the iteration \eqref{eq:pfIter} along with the properties of harmonic functions to prove \eqref{eq:pfClaim}.

	Let $\lambda \in (0,\frac{1}{16})$ be determined later on.
	Suppose $X_0=(\xi,\eta)$ and write $Y_0=(0,\eta)$.
	If $|\xi|\le \lambda \rho$, then we have
	\begin{align*}
		&\frac{1}{(\lambda\rho)^n}\int_{ B_{\lambda\rho}(X_0)\cap H} |v(X)-v(Y_0)|^2d\mathcal{H}^n(X)\\
		\le{}&\frac{2^n}{(\lambda \rho+|\xi|)^n}\int_{ B_{\lambda\rho+|\xi|}(Y_0)\cap H} |v(X)-v(Y_0)|^2d\mathcal{H}^n(X) \\
		\le{} & \frac{C}{(\rho-|\xi|)^n} \left( \frac{\lambda\rho + |\xi|}{\rho- |\xi|} \right)^{\frac{3}{2}} \int_{ B_{\rho-|\xi|}(Y_0)\cap H} |v(X)-v(Y_0)|^2d\mathcal{H}^n(X)\\
		\le{}& \frac{C}{\rho^n} \left( \frac{2\lambda}{1-\lambda} \right)^{\frac{3}{2}} \int_{ B_\rho(X_0)\cap H} |v(X)-v(Y_0)|^2d\mathcal{H}^n(X).
	\end{align*}
	Then, we can choose $\lambda$ small enough such that
	\begin{equation}
		\frac{1}{(\lambda\rho)^n}
		\int_{ B_{\lambda\rho}(X_0)\cap H} |v(X)-v(Y_0)|^2d\mathcal{H}^n(X)\le \frac{1}{4\rho^n} \int_{ B_\rho(X_0)\cap H} |v(X)-v(Y_0)|^2d\mathcal{H}^n(X).
		\label{eq:pfChangeCenterIter}
	\end{equation}

	In the case $|\xi|\ge \lambda \rho$, as $v$ is a harmonic function, we deduce
	\begin{equation}
		\frac{1}{(\sigma \lambda \rho)^n}
		\int_{ B_{\sigma\lambda\rho}(X_0)\cap H} 
		|v(X)-v(X_0)|^2d\mathcal{H}^n(X)\le
		\frac{C \sigma^2}{(\lambda \rho)^n}
		\int_{ B_{\lambda\rho}(X_0)\cap H} |v-b\vec{n}|^2d\mathcal{H}^n(X).
		\label{eq:pfHarmonicIter}
	\end{equation}
	for any real number $b \in \mathbb{R}$ and $\sigma \in \left(0,\frac{1}{2}\right)$.
	We choose $j$ with $\lambda^{j+1}<|\xi|\le \lambda^j$ and together with \eqref{eq:pfHarmonicIter} and \eqref{eq:pfChangeCenterIter}, we obtain
	\begin{align}
		&\frac{1}{(\sigma \lambda^{j+1})^n}
		\int_{ B_{\sigma \lambda^{j+1}}\cap H} 
|v(X)-v(X_0)|^2d\mathcal{H}^n(X)\nonumber \\
		\le{}&
\frac{C \sigma^2}{\lambda^{n(j+1)}}
\int_{ B_{\lambda^{j+1}}(X_0)\cap H} |v(X)-v(Y_0)|^2d\mathcal{H}^n(X)\nonumber \\
\le{}& \frac{C \sigma^2}{4^j \lambda^n} \int_{ B_{\lambda}(X_0)\cap H}  |v(X)-v(Y_0)|^2d\mathcal{H}^n(X),
\label{eq:pfIterHarmonic}
	\end{align}
	for any $\sigma \in \left(0,\frac{1}{2}\right)$, where we have used \eqref{eq:pfHarmonicIter} with $v(Y_0)$ in place of $b$ and used \eqref{eq:pfChangeCenterIter} $j$ times.
Applying the triangle inequality and equation \eqref{eq:pfIterHarmonic} with $\sigma=\frac{1}{4}$, we deduce
	\begin{align*}
		&|v(X_0)-v(Y_0)|^2
		\\
		=&
		\frac{C}{\lambda^{n(j+1)}}
		\int_{ B_{\frac{1}{4}\lambda^{j+1}(X_0)}\cap H}
		|v(X_0)-v(Y_0)|^2d\mathcal{H}^n(X)\\
		\le{}& \frac{C}{\lambda^{n(j+1)}}
		\int_{ B_{\frac{1}{4}\lambda^{j+1}(X_0)}\cap H}
		|v(X)-v(X_0)|^2d\mathcal{H}^n(X) +\frac{C}{\lambda^{n(j+1)}}
		\int_{ B_{\lambda^{j+1}(X_0)}\cap H}
		|v(X)-v(Y_0)|^2d\mathcal{H}^n(X)\\
		\le{}& \frac{C}{4^j\lambda^n}\int_{ B_\lambda(X_0)\cap H} |v(X)-v(Y_0)|^2d\mathcal{H}^n(X).
	\end{align*}
	
Thus, for any $1\le i\le j$, in conjunction with equation \eqref{eq:pfChangeCenterIter}, we conclude
	\begin{align*}
		&\frac{1}{\lambda^{n(i+1)}}\int_{ B_{\lambda^{i+1}(X_0)}\cap H} |v(X)-v(X_0)|^2d\mathcal{H}^n(X)\\
		\le{}& 
	\frac{C}{\lambda^{n(i+1)}}\int_{B_{\lambda^{i+1}}(X_0)\cap H} |v(X)-v(Y_0)|^2d\mathcal{H}^n(X)+
	\frac{C}{\lambda^{n(i+1)}}\int_{B_{\lambda^{i+1}}(X_0)\cap H} |v(X_0)-v(Y_0)|^2d\mathcal{H}^n(X)\\
		\le{}& \frac{C}{4 ^{i}\lambda^n}
		\int_{ B_\lambda(X_0)\cap H}|v(X)-v(Y_0)|^2d\mathcal{H}^n(X)+\frac{C}{4^j\lambda^n}\int_{ B_\lambda(X_0)\cap H} |v(X)-v(Y_0)|^2d\mathcal{H}^n(X) \\
		\le{}& \frac{C}{4 ^{i}\lambda^n}
		\int_{ B_\lambda(X_0)\cap H}|v(X)-v(Y_0)|^2d\mathcal{H}^n(X) 
	\end{align*}
Notably, this inequality remains valid even for $i \ge j+1$, as we can select $\sigma=\lambda^{j-i}$ in \eqref{eq:pfIterHarmonic}.
Consequently, through a standard interpolation technique, we find that
	\begin{align*}
		&\frac{1}{\rho^n} \int_{ B_\rho(X_0)\cap H} |v(X)-v(X_0)|^2d\mathcal{H}^n(X)\le C\rho^{2\alpha} \int_{ B_\lambda(X_0)\cap H} |v(X)-v(Y_0)|^2d\mathcal{H}^n(X)\\
		\le{}& C\rho^{2\alpha}\int_{ B_{\frac{3}{8}}\cap H} |v(X)-v(Y_0)|^2d\mathcal{H}^n(X),
	\end{align*}
	for any $\rho \in (0,\lambda)$.

	Now, we focus on estimating $v(Y_0)$.
	To do that, we apply \eqref{eq:BlowUpSing} with $(\mathcal{V}_k)_{0,\frac{1}{3}}$ in place of $\mathcal{V}_k$ in view of Lemma \ref{lem_change_center}, leading to
	\begin{align}
		|\zeta_k|^2\le{}& C
		\left( \int_{ B_{\frac{1}{3}}} \mathrm{dist}^2(X,\mathrm{spt}\|\boldsymbol{C}_k\|) d\|V^{\theta_k}_k\|(X)+\mu_k\right),
		\label{eq:pfEstSingular}
	\end{align}
	for any $Z_k=(0,\zeta_k,\eta_k) \in \mathrm{spt}\|V^{\theta_k}_k\|\cap B_{\frac{5}{24}}$ with $\Theta(\mathcal{V}_k,Z_k)\ge 1-\cos\theta_k(Z_k)$ for $k$ large enough.
	For any $Y \in B_{\frac{3}{16}}\cap \left\{ r=0 \right\}$, by choosing $\left\{ Z_k \right\}\subset \mathrm{spt}\|V^{\theta_k}_k\|\cap B_{\frac{5}{24}}$ with $\Theta(\mathcal{V}_k,Z_k)\ge 1-\cos\theta_k(Z_k)$ such that $Z_k \rightarrow Y$ and dividing both sides of inequality \eqref{eq:pfEstSingular} by $\mathcal{E}_k^2$, then taking $k$ goes to $+\infty$, we obtain
	\[
	|v(Y)|^2\le C \int_{ B_{\frac{3}{8}}\cap H}|v(X)|^2 d\mathcal{H}^n(X),
	\]
	for any $Y =(0,y) \in B_{\frac{3}{16}}$.

In combination with \eqref{eq:pfClaim}, it follows that
	\[
		\|v\|_{C^{0,\alpha}(\overline{B_{\frac{3}{16}}\cap H},H^\bot)}\le
		C \int_{ B_{\frac{3}{8}}\cap H} |v(X)|^2d\mathcal{H}^n(X).
	\]
	by the standard PDE theory.
\end{proof}
\begin{remark}
It is noteworthy that, according to the definition of blow-ups, it holds that
	\[
		\int_{ B_{\frac{3}{8}}\cap H} |v(X)|^2d\mathcal{H}^n(X)\le 1.
	\]
	\label{rmk:boundedL2}
\end{remark}

\begin{proposition}
	\label{prop_C2estimate}
	If $v$ is a blow-up, then $v \in C^2(\overline{B_{\frac{1}{8}}\cap H}, H^\bot)$ with the following estimate
	\[
	\|v\|^2_{C^2(\overline{B_{\frac{1}{8}}\cap H},H^\bot)}\le C \int_{ B_{\frac{3}{8}}\cap H} |v(X)|^2d\mathcal{H}^n(X)
	\]
	where $C=C(n,\Lambda_0)$.
\end{proposition}
\begin{proof}
For the sake of simplicity, we parametrize $H$ using $(r,y) \in \mathbb{R}^+\times \mathbb{R}^{n-1}$ and consider $v$ as a function mapping to real values.
	Let us consider a test function $\psi(r,y)\in C_c^2(B_{\frac{3}{8}}^n(0) \cap \left\{ r\ge 0 \right\})$ satisfying the following conditions,
	\begin{equation}
		\frac{\partial \psi}{\partial r}\equiv 0\text{ near }r=0\quad \text{ and }\quad \int_{B_{\frac{3}{8}}^{n-1}(0)} \psi(0,y)dy=0.
		\label{eq:pfDefPsi}
	\end{equation}
	
	We set $\varphi(x,y)=\psi(|x|,y) e_2$ in \eqref{eq:1stVarFormula} with $\mathcal{V}_k$ in place of $\mathcal{V}$ and $0$ in place of $X_0$.
	This leads to
	\begin{equation}
		\left|\int_{ } \mathrm{div}_S \varphi(X)dV^{\theta_k}_k(X,S)\right|\le
		C\mu_k \int_{ } \left( \left| \varphi(X) \right| + |X| \left|D _S \varphi\right| \right)dV^{\theta_k}_k(X,S)+C\mu_k.
		\label{eq:pf1stVar}
	\end{equation}

	We fix any arbitrary $\tau>0$ such that
	\begin{equation}
		\frac{\partial \psi}{\partial r}\equiv 0 \text{ when }r<\tau.
		\label{eq:pfpsiNear0}
	\end{equation}
	
	First, we deal with the left-hand side of \eqref{eq:pf1stVar}.
	We divide the integral $\int_{} \mathrm{div}_S \varphi(X) dV^{\theta_k}_k(X,S)$ into three parts as
	\begin{align*}
		\int_{} \mathrm{div}_S \varphi(X) dV^{\theta_k}_k(X,S)={}& 
		-\int_{\left\{ r\ge \tau \right\}} \cos \theta_k\mathrm{div}_S \varphi(X) dV_k(X,S)+\int_{ \left\{ r\ge \tau \right\}} \mathrm{div}_S \varphi(X) dW_k(X,S)\\
			&+\int_{ \left\{ r< \tau \right\}} \mathrm{div}_S \varphi(X) dV^{\theta_k}_k(X,S)\\
		=:{}& -\mathrm{I}_k+\mathrm{II}_k+\mathrm{III}_k.
	\end{align*}

	For the first term, according to Theorem \ref{thm__l_2_estimate} \ref{it:Wsupport}, for $k$ large enough, we get
	\begin{align}
		\mathrm{I}_k={}&\int_{\left\{ r\ge \tau \right\}} \cos \theta_k\mathrm{div}_S \varphi(X) dW_k(X,S)=-\int_{ \left\{ r\ge\tau \right\}} \cos \theta_k\frac{\partial \psi}{\partial r}(r,y)drdy\nonumber \\
		={}&-\int_{ } \psi(0,y)dy-\int_{ } \psi(r,y) \sin \theta_k \frac{\partial \theta_k}{\partial r}drdy=-\int_{\left\{ r\ge \tau \right\}} \psi(r,y) \sin \theta_k \frac{\partial \theta_k}{\partial r}drdy,\nonumber 
	\end{align}
	by \eqref{eq:pfDefPsi}.
	By $|D\theta_k|\le \mu_k$, we obtain
	\begin{equation}
		|\mathrm{I}_k|\le C\mu_k.
		\label{eq:pfI}
	\end{equation}

	Regarding the second term, we write $H_k=H^{\theta'_k}$, $\vec{n}_k=\cos\theta'_k e_1+\sin\theta'_k e_2, \vec{n}'_k=\sin \theta'_k e_1-\cos\theta'_k e_2$.
	Indeed, $\vec{n}_k$ is the unit normal vector of the half-plane $H_k$ defined before.

	Then, by Theorem \ref{thm__l_2_estimate} \ref{it:L2Est1} and the definition of $\tilde{u}_k$ in \eqref{eq:defUtil}, for $k$ larger enough, we can parametrize the support of $\|V_k\|$ by the following map
	\begin{align}
		F(r,y)={}& r\sqrt{1+l_k^2} \vec{n}'_k +\sum_{i =1}^{n-1}y_i e_{2+i}+\tilde{u}_k(r,y)\vec{n}_k,
		\label{eq:pfGraphOfUtil}
	\end{align}
	for $(r,y) \in B_{\frac{3}{8}}^{n}(0)\cap \left\{ r\ge \tau \right\}$.

	We define
	\[
		\tilde{\nu}_k:=\vec{n}_k-\frac{1}{\sqrt{1+l_k^2}}\frac{\partial \tilde{u}_k}{\partial r}\vec{n}'_k-D_y\tilde{u}_k.
	\]
	Here, recall that $D_y \tilde{u}_k:=\sum_{i =1}^{n-1}\frac{\partial \tilde{u}_k}{\partial y_i}e_{2+i}$.
	Then $\nu_k:=\frac{\tilde{\nu}_k}{|\tilde{\nu}_k|}$ is the unit normal vector field of the image of $F(r,y)$ defined in \eqref{eq:pfGraphOfUtil}.
	We can see that value of $D\psi$ at point $F(r,y)$ for $(r,y)\in B_{\frac{3}{8}}^n(0)\cap \left\{ r\ge \tau \right\}$ is given by
	\[
		D\psi=\frac{\partial \psi}{\partial r}
		\frac{r\sqrt{1+l_k^2}\vec{n}'_k+\tilde{u}_k\vec{n}_k}{\sqrt{r^2+l_k^2r^2+\tilde{u}_k^2}}+D_y \psi.
	\]

	Thus, we can obtain
	\begin{align*}
		\mathrm{II}_k={}&\int_{ \left\{ r\ge \tau \right\}} \mathrm{div}_S \varphi(X)dV_k(X,S)\\
		={} & \int_{ \left\{ r\ge \tau \right\}} \left( D\psi\cdot e_2-D\psi \cdot \nu_k e_2\cdot\nu_k  \right) \sqrt{1+l_k^2+\left( \frac{\partial \tilde{u}_k}{\partial r} \right)^2+(1+l_k^2)|D_y\tilde{u}_k|^2}drdy\\
		={} & \int_{ \left\{ r\ge \tau \right\}}\left[ \frac{\partial \psi}{\partial r} \frac{-r \sqrt{1+l_k^2}\cos \theta'_k+\tilde{u}_k\sin \theta'_k}{\sqrt{r^2(1+l_k^2)+\tilde{u}_k^2}}
		-\left( \frac{\partial \psi}{\partial r}\frac{\tilde{u}_k-r\frac{\partial \tilde{u}_k}{\partial r}}
	{|\tilde{\nu}_k|\sqrt{r^2(1+l_k^2)+\tilde{u}_k^2}} \right. \right.\\
			{}& \left. \left. - \frac{D_y \psi \cdot D_y \tilde{u}_k}{|\tilde{\nu}_k|} \right) 
			 \frac{1}{|\tilde{\nu}_k|}\left( \frac{1}{\sqrt{1+l_k^2}}\frac{\partial \tilde{u}_k }{\partial r}\cos \theta'_k+ \sin \theta'_k  \right) \right]\sqrt{1+l_k^2}|\tilde{\nu}_k|drdy\\
			={}& \mathrm{II}_k^1+\mathrm{II}_k^2+\mathrm{II}_k^3 +\mathrm{II}_k^4,
	\end{align*}
	where
	\begin{align*}
		\mathrm{II}_k^1={} & \int_{ \left\{ r\ge \tau \right\}} - \frac{r \left( 1+l_k^2 \right)\cos \theta'_k}{\sqrt{r^2(1+l_k^2)+\tilde{u}_k^2}}\frac{\partial \psi}{\partial r}|\tilde{\nu}_k|drdy,\\
		\mathrm{II}_k^2={} & \int_{ \left\{ r\ge \tau \right\}} \frac{\tilde{u}_k \sin \theta'_k}{\sqrt{r^2(1+l_k^2)+\tilde{u}_k^2}} \left( 1-\frac{1}{|\tilde{\nu}_k|^2} \right)\frac{\partial \psi}{\partial r} \sqrt{1+l_k^2}|\tilde{\nu}_k|drdy,\\
		\mathrm{II}_k^3={}& \int_{ \left\{ r\ge \tau \right\}} \left( \frac{\partial \psi}{\partial r} \frac{r \frac{\partial \tilde{u}_k}{\partial r}-\tilde{u}_k}{\sqrt{r^2(1+l_k^2)+\tilde{u}_k^2}}+ D_y \psi \cdot D_y \tilde{u}_k\right)
		\frac{\partial \tilde{u}_k}{\partial r}\frac{\cos\theta'_k}{|\tilde{\nu}_k|}drdy,\\
		\mathrm{II}_k^4={}& \int_{ \left\{ r\ge \tau \right\}} \left( \frac{\partial \psi}{\partial r} \frac{\partial \tilde{u}_k}{\partial r} \frac{r}{\sqrt{r^2(1+l_k^2)+\tilde{u}_k^2}}+ D_y \psi \cdot D_y \tilde{u}_k\right)
		\frac{\sin \theta'_k\sqrt{1+l_k^2}}{|\tilde{\nu}_k|}drdy.
	\end{align*}

	For the term $\mathrm{II}_k^1$, using $\int_{ \left\{ r\ge \tau \right\}} \frac{\partial \psi}{\partial r}drdy=0$, we find
	\begin{align*}
		& - \frac{1}{\sqrt{1+l_k^2}\cos \theta'_k} \mathrm{II}_k^1
		=\int_{ \left\{ r\ge \tau \right\}} \frac{r |\tilde{\nu}_k|\sqrt{1+l_k^2}}{\sqrt{r^2(1+l_k^2)+\tilde{u}_k^2}}\frac{\partial \psi}{\partial r}drdy\\
		={}& \int_{ \left\{ r\ge \tau \right\}}  \left( \frac{r |\tilde{\nu}_k|\sqrt{1+l_k^2 }}{\sqrt{r^2(1+l_k^2)+\tilde{u}_k^2}}-1 \right)\frac{\partial \psi}{\partial r}drdy\\
		={}& \int_{ \left\{ r\ge \tau \right\}} \frac{r^2\left( \frac{\partial \tilde{u}_k}{\partial r} \right)^2+r^2(1+l_k^2)|D_y \tilde{u}_k|^2-\tilde{u}_k^2}{\sqrt{r^2(1+l_k^2)+\tilde{u}_k^2}\left( r|\tilde{\nu}_k|\sqrt{1+l_k^2}+\sqrt{r^2(1+l_k^2)+\tilde{u}_k^2} \right)}\frac{\partial \psi}{\partial r}drdy
	\end{align*}

	Dividing both side by $\mathcal{E}_k$ and taking limit as $k\rightarrow +\infty$ yields
	\begin{equation}
		\lim_{k\rightarrow +\infty} \mathcal{E}_k^{-1}\mathrm{II}_k^1=0.
		\label{eq:pfII1}
	\end{equation}

	Here, we have used the fact that $l_k \rightarrow 0$, $\mathcal{E}_k^{-1}\tilde{u}_k\rightarrow v$ in the sense of $C^{1,\beta}(\left\{ r\ge \tau \right\}\cap B_{\frac{3}{8}}^n(0))$.

	For the term $\mathrm{II}_k^2$ and $\mathrm{II}_k^3$, we directly have
	\begin{equation}
		\lim_{k\rightarrow +\infty} \mathcal{E}_k^{-1}\mathrm{II}_k^2=\lim_{k\rightarrow +\infty} \mathcal{E}_k^{-1}\mathrm{II}_k^3=0.
		\label{eq:pfII23}
	\end{equation}
	
	For the term $\mathrm{II}_k^4$, we note that
	\begin{equation}
		\lim_{k\rightarrow +\infty} \mathcal{E}_k^{-1}\mathrm{II}_k^4=\int_{ \left\{ r\ge \tau \right\}} \frac{\partial \psi}{\partial r}\frac{\partial v}{\partial r}+D_y \psi \cdot D_y v drdy=\int_{ \left\{ r\ge \tau \right\}} D\psi \cdot Dv drdry=-\int_{ \left\{ r\ge \tau \right\}} v \Delta \psi drdy.
		\label{eq:pfII4}
	\end{equation}
	
	Summarizing \eqref{eq:pfII1}, \eqref{eq:pfII23} and \eqref{eq:pfII4}, we conclude
	\begin{equation}
		\lim_{k\rightarrow +\infty} \mathcal{E}_k^{-1}\mathrm{II}_k=\int_{ \left\{ r\ge \tau \right\}} D\psi \cdot Dv drdy.
		\label{eq:pfII}
	\end{equation}

	Concerning the third term $\mathrm{III}_k$, using \eqref{eq:pfpsiNear0}, we have $e_1\cdot D\psi=e_2\cdot D\psi=0$ and hence $D\psi=\sum_{i =1}^{n-1}\frac{\partial \psi}{\partial y_i}e_{2+i}$.
	Therefore, for $(x,y) \in \mathrm{spt}\|V_k^{\theta_k}\|\cap \left\{ r<\tau \right\}\cap B_{\frac{3}{8}}(0)$ and $S \in G(n+1,n)$, we can write
	\[
		\mathrm{div}_S \varphi=
		e_2\cdot (D\psi)^{\top_S}=e_2\cdot D\psi-e_2\cdot (D\psi)^{\bot _S}=-e_2\cdot \sum_{i =1}^{n-1}\frac{\partial \psi}{\partial y_i} e_{2+i}^{\bot _S}.
	\]
	By Theorem \ref{thm__l_2_estimate} \ref{it:L2EsteiNormal}, $\mathrm{III}k$  can be estimated as
	\begin{align}
		\mathrm{III}_k\le{}&
		\int_{ \left\{ r<\tau \right\}} 
		|D\psi| \sum_{i =1}^{n-1}|e_{2+i}^{\bot _S}|dV^{\theta_k}_k(X,S)\\
		\le{}& \|D\psi\|_{\infty}
		\sum_{i =1}^{n-1}
		\left( \int_{ \left\{ r<\tau \right\}\cap B_{\frac{3}{8}}(0)} |e_{2+i}^{\bot _S}|^2dV_k(X,S) \right)^{\frac{1}{2}}\left( \int_{ \left\{ r<\tau \right\}\cap B_{\frac{3}{8}}(0)} d\|V_k\|(X) \right)^{\frac{1}{2}}\nonumber \\
		\le{}& C\|D\psi\|_\infty \mathcal{E}_k \sqrt{\tau}
		\label{eq:pfIII}
	\end{align}
	where we have used the fact that $\|V^{\theta_k}_k\|(B_{\frac{3}{8}}(0)\cap \left\{ r<\tau \right\})\le C\tau$.
	Divide both side of \eqref{eq:pfIII} by $\mathcal{E}_k$, and take the limit as $k\rightarrow +\infty$, we can get
	\[
		\lim_{k\rightarrow 0} \mathcal{E}_k^{-1}\left|\mathrm{III}_k\right|\le C\|D\psi\|_\infty\sqrt{\tau}.
	\]
	Together with \eqref{eq:pfI} and \eqref{eq:pfII}, we conclude
	\[
		\left|
		\lim_{k\rightarrow +\infty} \mathcal{E}_k^{-1} \int_{ } \mathrm{div}_S \varphi dV^{\theta_k}_k+ \int_{ \left\{ r\ge \tau \right\}} v D\psi drdy
		\right|
		\le C \|D\psi\|_{\infty}\sqrt{\tau},
	\]
	where the constant $C$ is independent of $\tau$.
	Since $v \in L^1(B_{\frac{3}{8}}^n(0)\cap \left\{ r\ge 0 \right\})$, we can take the limit $\tau \rightarrow 0^+$ to get
	\begin{equation}
		\lim_{k\rightarrow +\infty} \mathcal{E}_k^{-1} \int_{ } \mathrm{div}_S \varphi dV^{\theta_k}_k(X,S)=-\int_{ \left\{ r\ge 0 \right\}} v \Delta\psi drdy.
		\label{eq:pfLeftSide}
	\end{equation}

	On the other hand, for the right side of \eqref{eq:pf1stVar}, we have
	\begin{align*}
		\mu_k\int_{ }\left( |\varphi(X)|+|X||D_S \varphi|\right)dV^{\theta_k}_k(X,S)\le{}& C \mu_k\left( \|\psi\|_{\infty}+\|D\psi\|_{\infty} \right)\|V^{\theta_k}_k\|(B_{\frac{3}{8}}(0)) \\
		\le{}& C \mathcal{E}_k^2\left( \|\psi\|_{\infty}+\|D\psi\|_{\infty} \right)\|V^{\theta_k}_k\|(B_{\frac{3}{8}}(0))
	\end{align*}
	which, together with \eqref{eq:pf1stVar} and \eqref{eq:pfLeftSide}, allows us to assert that
	\begin{equation}
		\int_{ \left\{ r\ge 0 \right\}} v \Delta \psi drdy=0,
		\label{eq:pfCondHarmonic}
	\end{equation}
	for $\psi$ such that \eqref{eq:pfDefPsi} holds.
	
	For any function $f$ defined on $H$ and $l=1,2,\cdots ,n-1$, we write
	\[
		\delta_{l,h}f(r,y)=f(r,y_1,\cdots ,y_{l-1},y_{l}+h,y_{l+1},\cdots ,y_{n-1})-f(r,y).
	\]

	We extend $v$ to $\mathbb{R}^n$ by $v(r,y)=v(-r,y)$ for $r<0$.
	Giving that $\psi \in C_c^2(B_{\frac{5}{16}})$ such that $\frac{\partial \psi}{\partial r}\equiv 0$ near $r=0$, it is easy to see that
	$\delta_{l,-h}\psi$ satisfies \eqref{eq:pfDefPsi} for any $l=1,\cdots ,n-1$ and $h\in (-\frac{1}{16},\frac{1}{16})$. Following \eqref{eq:pfCondHarmonic}, we obtain
	\begin{equation}
		0=\int v\Delta\delta_{l,-h}\psi drdy=\int_{ } \delta_{l,h} v \Delta \psi drdy.
		\label{eq:pfQuotientHarmonic}
	\end{equation}

	By an approximate argument, we know \eqref{eq:pfQuotientHarmonic} holds for any $\psi \in C_c^2(B_{\frac{5}{16}}^n)$ which is even in the $r$ variable.

	But since $\delta_{l,h}v$ is also an even function in $r$, we can get \eqref{eq:pfQuotientHarmonic} holds for any $\psi \in C_c^2(B_{\frac{5}{16}})$.
	Consequently, $\delta_{l,h}v$ is a smooth harmonic function defined on $B_{\frac{5}{16}}^n(0)$.
	Proposition \ref{prop_HolderBlowUp} and the definition of $\delta_{l,h}$ imply that
	\[
		\left|\int_{ B_{\frac{5}{16}}(0)} \frac{1}{h}\delta_{l,h}vdrdy\right|^2\le C \int_{ B_{\frac{3}{8}}\cap H} |v(X)|^2d\mathcal{H}^n(X).
	\]

	Hence, the standard elliptic estimates for harmonic functions imply that there exists a harmonic function $v_l$ defined on $B_{\frac{9}{32}}^n(0)$ such that $\delta_{l,h}v\rightarrow v_l$ in $C^2(B_{\frac{9}{32}}^n(0))$ as $h\rightarrow 0$, and
	\begin{equation}
		\sup_{B_{\frac{9}{32}}^n(0)}\left( 
		|v_l|^2+|Dv_l|^2+|D^2v_l|^2\right)\le C \int_{ B_{\frac{3}{8}}\cap H} |v(X)|^2d\mathcal{H}^n(X).
		\label{eq:pfvLDeri}
	\end{equation}

	Recall that for any $r\neq 0$ and $(r,y) \in B_{\frac{9}{32}}^n(0)$, we know $\frac{\partial v}{\partial y_l}=v_l$. So
	\begin{equation}
		v(r,y)=v(r,y_1,\cdots ,y_{l-1},0,y_{l+1},\cdots ,y_{n-1})+\int_{0}^{y_l}v_l(r,y_1,\cdots ,y_{l-1},t,y_{l+1},\cdots ,y_{n-1})dt. 
		\label{eq:pfvIntegral}
	\end{equation}

	By the continuity of $v$ and $v_l$, we can let $r\rightarrow 0^+$ and obtain that \eqref{eq:pfvIntegral} holds for $r=0$.
	Hence, $v(0,y)$ is a smooth function on $B_{\frac{9}{32}}^{n-1}(0)$ and from \eqref{eq:pfvLDeri} and Proposition \ref{prop_HolderBlowUp}, we can assert that
	\[
		\sup_{B_{\frac{9}{32}}^{n-1}(0)}\left( 
		|v(0,y)|^2+|D_yv(0,y)|^2+|D^2_yv(0,y)|^2+|D^3_y v(0,y)|^2\right)\le C \int_{ B_{\frac{3}{8}}\cap H} |v(X)|^2d\mathcal{H}^n(X).
	\]

	Then, along with the standard $C^{2,\alpha}$ estimates for harmonic function (e.g., \cite{Morrey1966Book}), we conclude that $v(r,y)$ is a $C^{2,\alpha}$ function on $B_{\frac{1}{8}}^n(0)\cap \left\{ r\ge 0 \right\}$ for some $\alpha\in (0,1)$ with estimate
	\[
		\|v\|_{C^{2,\beta}(B_{\frac{1}{8}}^n(0))\cap \left\{ r\ge 0 \right\}}^2\le
		C \int_{ B_{\frac{3}{8}}^n(0)\cap H} |v(X)|^2d\mathcal{H}^n(X).
	\]
\end{proof}

\section{Improvement of excess}
\label{sec:iterations}

With the application of the $C^2$ estimate for blow-ups, it is possible to improve the excess.

\begin{proposition}
	\label{prop_improveExcess}
	There exists $\lambda_0=\lambda_0(n,\Lambda_0) \in (0,\frac{1}{4})$ and $C=C(n,\Lambda_0) \in (0,+\infty)$ such that following holds.
	For any $\lambda \in (0,\lambda_0]$, there exists $\varepsilon_0$ such that 
	if a VPCA-quadruple $\mathcal{V}=(V,W,\theta,g)$ and a half-hyperplane $H$ satisfies $(\Lambda_0,\varepsilon_0)$-hypothesis, 
	then there exists $\Gamma \in \mathrm{SO}(n)\subset \mathrm{SO}(n+1)$, a half plane $H'$ in $\mathbb{H}^{n+1}$, and we write $\mathcal{V}'=(V',W',\theta',g'):=((\Gamma \circ \eta_{0,\lambda})_{\#}V,(\Gamma \circ \eta_{0,\lambda})_{\#}W,\theta \circ\Gamma \circ \eta_{0,\lambda}, \frac{1}{\lambda^2}(\Gamma \circ \eta_{0,\lambda})_{*}g)$, such that
	\begin{enumerate}[\normalfont(a)]
		\item $\left|\Gamma-\mathrm{Id}\right|\le C\mathcal{E}(\mathcal{V},H)$.
		\item $|H'-H|\le C\mathcal{E}(\mathcal{V},H)$.
		\item $\mathcal{E}(\mathcal{V}',H')\le \frac{1}{2}\mathcal{E}(\mathcal{V},H)$.
	\end{enumerate}
\end{proposition}

\begin{proof}
	[Proof of Proposition \ref{prop_improveExcess}]
	For simplicity, we write
	\[
		\mathcal{V}^\Gamma_{0,\lambda}=((\Gamma \circ \eta_{0,\lambda})_{\#}V,(\Gamma \circ \eta_{0,\lambda})_{\#}W,\theta \circ\Gamma \circ \eta_{0,\lambda}, \frac{1}{\lambda^2}(\Gamma \circ \eta_{0,\lambda})_{*}g).
	\]
	
	Assuming the contrary of this proposition, for any given constant $C$, there exists a blow-up sequence $\left\{ (\mathcal{V}_k,H_k) \right\}$, a positive sequence $\left\{ \lambda_k \right\}$ with $\lim_{k\rightarrow +\infty} \lambda_k=0$ such that $(\mathcal{V}_k,H_k)$ satisfies $(\Lambda_0,\varepsilon_k)$-hypothesis for some $\varepsilon_k\rightarrow 0^+$.
	But for each $k$, and $\left|\Gamma-\mathrm{Id}\right|\le C \mathcal{E}(\mathcal{V}_k,H_k)$ and $|H'-H_k|\le C\mathcal{E}(\mathcal{V}_k,H_k)$, we can always find $\lambda \in (0,\frac{1}{4})$ such that
	\begin{equation}
		\mathcal{E}( (\mathcal{V}_k)^\Gamma_{0,\lambda},H')> \frac{1}{2}\mathcal{E}(\mathcal{V}_k,H_k).
		\label{eq:pfContraAssump}
	\end{equation}

	By Theorem \ref{thm_compactnessRIV}, we assume $\mathcal{V}_k \rightarrow \mathcal{V}=(V_0,W_0,\theta_0,\delta)$ and $H_k \rightarrow H_0$.

Up to a subsequence, we can choose two sequences of positive decreasing numbers $\left\{ \delta_k \right\}$ and $\left\{ \tau_k \right\}$ such that \crefrange{it:graphK}{it:BlowUpImproveL2K} in Section \ref{sec:blow_up_argument} hold.

We write $Y_i = \frac{1}{2}\lambda e _{2+i}$
for each $i=1,2,\cdots ,n-1$. 
From \ref{it:BlowUpSingK} in Section \ref{sec:blow_up_argument},
we can find $Z_{i,k}=(0,\zeta_{i,k}, \eta_{i,k}) \in \mathrm{spt}\|V^{\theta_k}_k\|$ such that $\Theta(\mathcal{V}_k,Z_{i,k})\ge 1-\cos \theta_k(Z_{i,k})$ and $\lim_{k\rightarrow +\infty} Z_{i,k}=Y_i$.
We denote $P_k$ the hyperplane in $\mathbb{R}^n=\partial \mathbb{H}^{n+1}$ spanned by $\left\{ Z_{i,k} \right\}_{i=1,2,\cdots ,n-1}$.
Similarly, $P_0$ denotes the hyperplane spanned by $\left\{ Y_i \right\}_{i=1,2,\cdots ,n-1}$.
Noting that $(0,0,\eta_{i,k})\rightarrow Y_i$ as $k\rightarrow +\infty$ and $\left\{ (0,0,\eta_{i,k}) \right\}_{i=1,\cdots ,n-1}$ spans $P_0$ for $k$ large enough, we can estimate the distance between $P_k$ and $P_0$ as
\[
	|P_k-P_0|\le C \max_{1\le i\le n-1}\frac{|\zeta_{i,k}|}{\lambda}
\]
where $C$ is a constant which does not rely on $\lambda$.
Consequently, we can find an orthogonal rotation $\Gamma_k \in \mathrm{SO}(n)$ such that it maps $P_k$ to $P_0$ and
\begin{equation}
	\left|\Gamma_k-\mathrm{Id}\right|\le C \max_{1\le i\le n-1}\frac{|\zeta_{i,k}|}{\lambda}.
	\label{eq:pfOrth}
\end{equation}

On the other hand, we note that
$\lim_{k\rightarrow +\infty} \mathcal{E}_k^{-1} \zeta_{i,k}=\kappa(y)$ by the definition of $\kappa(Y_i)$ in \eqref{eq:defKappa}.
Hence
\[
	v(Y_i)=\kappa(y)\sin \theta_0 \vec{n}_0.
\]
Here, $\vec{n}_0$ is the unit normal vector of $H_0$ pointing upward. By $\kappa(0)=0$, we know $v(0)=0$.
Thus, by Proposition \ref{prop_C2estimate} and Remark \ref{rmk:boundedL2}, we infer
\[
	|v(Y_i)|\le C \lambda.
\]

Hence,
\[
	\mathcal{E}_k^{-1}\left|\zeta_{i,k}\right|\le C \lambda
\]
for $k$ large enough, resulting in
\begin{equation}
	\left|\Gamma_k-\mathrm{Id}\right|\le C \mathcal{E}_k.
	\label{eq:pfDistGamma}
\end{equation}

Now, we consider another sequence
$\{ (\hat{\mathcal{V}}_k,H_k) \}$, where
$\mathcal{\hat{V}}_k=(\hat{V}_k,\hat{W}_k,\hat{\theta}_k,\hat{g}_k):=\mathcal{V}^{\Gamma_k}_{0,\frac{7}{8}}$, which is also a blow-up sequence in view of Lemma \ref{lem_change_center}, and estimate \eqref{eq:pfDistGamma}.
This leads to
\[
	\mathcal{\hat{E}}_k\le C\mathcal{E}_k
\]
for some constant $C$.
We set $\hat{v} \in C_{\mathrm{loc}}^{1,\beta}(B_{\frac{3}{8}}\cap H_0)\cap L^2(B_{\frac{3}{8}}\cap H_0)$ to be the blow-up associated to $\{ (\mathcal{\hat{V}}_k,H_k) \}$ and we actually know $\hat{v} \in C^2(B_{\frac{1}{8}}\cap H_0)$ by Proposition \ref{prop_C2estimate}.

By the choice of $Z_{i,k}$ and $\Gamma_k$, the point $\hat{Z}_{i,k}:=\eta_{0,\frac{7}{8}}\circ \Gamma_k(Z_{i,k})$ can be expressed as
\begin{equation}
	\hat{Z}_{i,k}=(0,0,\hat{\eta}_{i,k})\in \mathrm{spt}\|\hat{V}^{\theta_k}_k\|
	\label{eq:pfZprim}
\end{equation}
with $\Theta(\mathcal{\hat{V}},\hat{Z}_{i,k})\ge 1-\cos \theta_k(Z_{i,k})$ and $\lim_{k\rightarrow +\infty} \hat{Z}_{i,k}=\frac{4}{7}\lambda e_{2+i}$.

Moreover, according to \eqref{eq:pfZprim}, we get $\hat{v}(\frac{4}{7}\lambda e_{2+i})=0$ for each $i=1,2,\cdots ,n-1$.
Along with $\hat{v}(0)=0$, we can find a point $Y_i'$ on segment $[0,\frac{4}{7}\lambda e_{2+i}]$ such that $\frac{\partial \hat{v}}{\partial y_i}(Y_i')=0$.
Again, using $|D^2\hat{v}|^2\le C \int_{ B_{\frac{3}{8}}\cap H_0} \left|\hat{v}(X)\right|^2d\mathcal{H}^n(X)$ from Proposition \ref{prop_C2estimate}, we derive
\begin{equation}
	\sup_{B_{\lambda}^n}\left|D_y \hat{v}\right|^2\le C \lambda^2\int_{ B_{\frac{3}{8}}\cap H_0} |\hat{v}(X)|^2d\mathcal{H}^n(X).
	\label{eq:pfDyVprim}
\end{equation}

We introduce a linear function $\hat{L}(r,y)=\frac{\partial \hat{v}}{\partial r}(0,0)r$, by Proposition \ref{prop_C2estimate} and \eqref{eq:pfDyVprim}, we deduce
\[
	\sup_{B_{\frac{8}{7}\lambda}^n} \left|D\left( \hat{v}-\hat{L}(r,y) \right)\right|^2\le C \lambda^2 \int_{ B_{\frac{3}{8}}\cap H_0} |\hat{v}(X)|^2d\mathcal{H}^n(X).
\]
Therefore,
\begin{equation}
	\int_{ B_{\frac{8}{7}\lambda}^n} \left|\hat{v}-\hat{L}(r,y)\right|^2d\mathcal{H}^n(X)\le C \lambda^{n+4} \int_{ B_{\frac{3}{8}}\cap H_0} |\hat{v}(X)|^2d\mathcal{H}^n(X).
	\label{eq:pfFineL2}
\end{equation}

We then choose the half plane $\hat{H}_k:= \{ X+(l_kr(X)+\mathcal{\hat{E}}_k \hat{L}(X))\vec{n}_0:X \in H_0 \}$ and denote $\hat{\boldsymbol{C}}_k=\boldsymbol{C}_{\hat{H}_k}$.
Here, we understand $\hat{L}(X)=\hat{L}(r(X),y)$ for $X=(x,y)$.
Clearly, we know
\begin{equation}
	|\hat{H}_k-H_k|\le C \mathcal{\hat{E}}_k\le C\mathcal{E}_k.
	\label{eq:pfDistHk}
\end{equation}

From the construction of $\hat{v}$, we see that
\[
	\lim_{k\rightarrow +\infty}
	\frac{1}{(\mathcal{\hat{E}}_k)^2}\int_{ B_{\frac{8}{7}\lambda}} \mathrm{dist}^2(X,\mathrm{spt}\|\hat{\boldsymbol{C}}_k\|)d\|\hat{V}^{\hat{\theta}_k}_k\|(X)=\int_{ B_{\frac{8}{7}\lambda}}\left|\hat{v}-\hat{L}(r,y)\right|^2d\mathcal{H}^n(X)
\]
Together with \eqref{eq:pfFineL2} and Remark \ref{rmk:boundedL2}, we conclude
\begin{align}
	\int_{ B_1} \mathrm{dist}^2(X,\mathrm{spt}\|\hat{\boldsymbol{C}}_k\|)d\|(\eta_{0,\frac{8}{7}\lambda})_{\#}\hat{V}^{\hat{\theta}_k}_k\|(X)\le{} & C \lambda ^{n+2}(\mathcal{\hat{E}}_k)^2\le C\lambda^{2}\mathcal{E}_k^2,
	\label{eq:pfImproveL2}
\end{align}
for $k$ large enough.
Using the fact $\mu((\eta_{0,\frac{8}{7}\lambda})_{\#}\mathcal{V})\le C \lambda \mu(\mathcal{V})$ and choosing $\mathcal{V}'=\eta_{0,\frac{8}{7}\lambda}\mathcal{\hat{V}}_k$, $H'=\hat{H}_k$ for $k$ large enough, we find that
\[
	\mathcal{E}^2(\mathcal{V}',H')\le C \lambda^2 \mathcal{E}_k^2+\lambda \mathcal{E}_k^2\le \frac{1}{4}\mathcal{E}_k^2,
\]
if we choose $\lambda$ small enough.

This contradicts with the formula \eqref{eq:pfContraAssump}.
\end{proof}

\begin{proposition}
	\label{prop_ImproveExcessRho}
	There exists $\varepsilon=\varepsilon(n,\Lambda_0) \in (0,1)$, $\alpha=\alpha(n,\Lambda_0)\in (0,1)$, and a constant $C=C(n,\Lambda_0)$ such that the following holds.
	For any $(\mathcal{V},H)$ which satisfies $(\Lambda_0,\varepsilon)$-Hypothesis, for any $Z \in \mathrm{spt}\|V^\theta\|\cap B_{\frac{1}{4}}\cap \partial \mathbb{H}^{n+1}$ with $\Theta(\mathcal{V},Z)\ge 1-\cos \theta(Z)$, there exists an orthogonal rotation $\Gamma^Z \in \mathrm{SO}(n)\subset \mathrm{SO}(n+1)$ and a half plane $H^Z$ such that following holds,
	\begin{align}
		\left|\Gamma^Z-\mathrm{Id}\right| ={} & C \mathcal{E}(\mathcal{V},H)\label{eq:propGammaZ} \\
		\mathrm{dist}(H^Z,H) \le{} & C \mathcal{E}(\mathcal{V},H)\label{eq:propHZ}\\
		\mathcal{E}((\mathcal{V}^Z)_{0,\rho}^{\Gamma^Z},H^Z)\le{}& C \rho^\alpha\mathcal{E}(\mathcal{V},H)\label{eq:propExcessZ}.
	\end{align}
\end{proposition}

\begin{proof}
	We choose $\varepsilon_0,\lambda_0$ to be the constants such that Proposition \ref{prop_improveExcess} holds.

We proceed to establish the following lemma
	\begin{lemma}
		Suppose $(\mathcal{V},H)$ satisfies $(\Lambda_0,\varepsilon)$-Hypothesis for $\varepsilon$ small enough.
		Then, for each $Z \in \mathrm{spt}\|V^\theta\|\cap B_{\frac{1}{4}}\cap \partial \mathbb{H}^{n+1}$ with $\Theta(\mathcal{V},Z)\ge 1-\cos \theta(Z)$, we can find a sequence of orthogonal rotations $\left\{ \Gamma_k^Z \right\}_{i=0}^\infty\subset \mathrm{SO}(n)\subset \mathrm{SO}(n+1)$ with $\Gamma_0^Z=\mathrm{Id}$, a sequence of half planes $\left\{ H_k^Z \right\}_{i=0}^\infty$ with $H_0^Z=H$ such that,
	\begin{align}
		\left|\Gamma_k^Z-\Gamma_{k-1}^Z\right|\le{}& \frac{C}{2^k}\mathcal{E}(\mathcal{V}^Z,H) \label{eq:pfGammaIter} \\
		\mathrm{dist}(H^Z_k,H^Z_{k-1})\le{}& \frac{C}{2^k}\mathcal{E}(\mathcal{V}^Z,H)\label{eq:pfHalfPlaneIter}\\
		\mathcal{E}((\mathcal{V}^Z)^{\Gamma_k^Z}_{0,\lambda_0^k},H_k^Z)\le{}& \frac{1}{2^k}\mathcal{E}(\mathcal{V}^Z,H) \label{eq:pfExcessIter}
	\end{align}
	where the constant $C=C(n,\Lambda_0)$.
	Here, $\mathcal{V}^Z:=\mathcal{V}_{Z,\frac{1}{2}}$.
	\end{lemma}
	\begin{proof}
The base case of $k=1$ directly follows from Proposition \ref{prop_improveExcess} and Lemma \ref{lem_change_center}.
	Assume that we have found $\Gamma_1^Z,\cdots, \Gamma_{k-1}^Z$ and $H_1^Z,\cdots, H_{k-1}^Z$ such that the above inequalities hold for $1,2,\cdots ,k-1$ in place of $k$.
	We write $\mathcal{V}'=(\mathcal{V}^Z)^{\Gamma_k^Z}_{0,\lambda_0}$.
	Note that by \eqref{eq:pfHalfPlaneIter} and \eqref{eq:pfExcessIter}, we find
	\begin{align*}
		\left|H_{k-1}^Z-H_0\right|\le{}&
		\sum_{i =1}^{k-1}\left|H_{i}^Z-H_{i-1}^Z\right|+\left|H-H_0\right|\le
		2C \mathcal{E}(V,H)+\varepsilon\\
		\mathcal{E}(\mathcal{V}',H_{k-1})\le{}& \mathcal{E}(\mathcal{V}^Z,H).
	\end{align*}
	Therefore, $(\mathcal{V}',H_{k-1})$ satisfies $(\Lambda_0,\varepsilon_0)$-hypothesis if we choose $\varepsilon$ small enough.
	By applying Proposition \ref{prop_improveExcess}, we can find $\Gamma_k'$ and $H_k'$ such that
	\begin{align}
		\left|\Gamma_k'-\mathrm{Id}\right|\le{} & C \mathcal{E}(\mathcal{V}',H_{k-1}^Z) \le \frac{2C}{2^k}\mathcal{E}(\mathcal{V}^Z,H)\label{eq:pfChoiceGamma}\\
		\mathrm{dist}(H_k',H_{k-1}^Z)\le{} & C \mathcal{E}(\mathcal{V}',H_{k-1}^Z)  \le \frac{2C}{2^k}\mathcal{E}(\mathcal{V}^Z,H)\label{eq:pfChoiceHk}\\
		\mathcal{E}((\mathcal{V}')^{\Gamma_k'}_{0,\lambda_0},H_k')\le{}& \frac{1}{2}\mathcal{E}(\mathcal{V}',H_{k-1}^Z)\le \frac{1}{2^{k}}\mathcal{E}(\mathcal{V}^Z,H).\label{eq:pfChoiceExcess}
	\end{align}
	Setting $\Gamma_k^Z=\Gamma_k'\circ \Gamma_{k-1}^Z$, it follows that
	\[
		(\mathcal{V}')^{\Gamma_k'}_{0,\lambda_0}=
		(\mathcal{V}^Z)^{\Gamma_k^Z}_{0,\lambda_0^k}.
	\]
Noting that $|\Gamma_k'-\mathrm{Id}|=|\Gamma_k^Z-\Gamma_{k-1}^Z|$ and considering \eqref{eq:pfChoiceGamma}, \eqref{eq:pfChoiceHk}, and \eqref{eq:pfChoiceExcess}, the conditions
\eqref{eq:pfGammaIter}, \eqref{eq:pfHalfPlaneIter}, and \eqref{eq:pfExcessIter} are satisfied.
This completes the proof of the lemma.
	\end{proof}

	Now, let us retire to the proof of Proposition \ref{prop_ImproveExcessRho}.
	The limits $\Gamma^Z:=\lim_{k\rightarrow +\infty}\Gamma_k^Z$ and $H^Z:=\lim_{k\rightarrow +\infty}H_k^Z$ exist based on \eqref{eq:pfGammaIter} and \eqref{eq:pfHalfPlaneIter}.
Additionally, by using \eqref{eq:pfGammaIter} and \eqref{eq:pfHalfPlaneIter} again, and noting that $\mathcal{E}(\mathcal{V}^Z,H)\le C\mathcal{E}(\mathcal{V},H)$, we obtain
\begin{align}
	\left|\Gamma^Z-\Gamma_k^Z\right|
	\le{}& \frac{C}{2^k}\mathcal{E}(\mathcal{V},H),
	\label{eq:pfLimitGamma}\\
	\mathrm{dist}(H^Z,H_k^Z)\le{}& \frac{C}{2^k}\mathcal{E}(\mathcal{V},H)\label{eq:pfLimitH}.
\end{align}
Using triangle inequality, along with \eqref{eq:pfLimitGamma}, \eqref{eq:pfLimitH}, and \eqref{eq:pfExcessIter}, we deduce
\begin{equation}
	\mathcal{E}((\mathcal{V}^Z)^{\Gamma^Z}_{0,\lambda_0^k},H^Z)\le \frac{C}{2^k}\mathcal{E}(\mathcal{V},H).
	\label{eq:pfLimitExcess}
\end{equation}
Furthermore, a standard interpolation implies
\begin{equation}
	\mathcal{E}((\mathcal{V}^Z)^{\Gamma^Z}_{0,\rho},H^Z)\le C \rho^{\alpha} \mathcal{E}(\mathcal{V},H),
\end{equation}
for any $\rho \in (0,\frac{1}{4}]$ and some $\alpha =\alpha(n,\Lambda_0)\in (0,1)$, $C=C(n,\Lambda_0) \in (0,+\infty)$.
Therefore, \eqref{eq:propExcessZ} holds in view of the definition of $\mathcal{V}^Z$.
Note that \eqref{eq:propGammaZ} and \eqref{eq:propHZ} hold if we choose $k=0$ in \eqref{eq:pfLimitGamma} and \eqref{eq:pfLimitH}.
\end{proof}
\begin{proof}[Proof of Theorem \ref{thm_l2positiveAngle}]

From Proposition \ref{prop_ImproveExcessRho}, for any $\varepsilon_1>0$, we can find $\Gamma^Z \in \mathrm{SO}(n)$ and a half plane $H^Z$ such that $(\mathcal{V}^{Z,\rho},H^Z)$ satisfies $(\Lambda_0,\varepsilon_1)$-hypothesis for any $\rho \in (0,1)$ by choosing $\varepsilon$ small enough.
Here, we write $\mathcal{V}^{Z,\rho}:=\mathcal{V}^{\Gamma^Z}_{Z,\rho}$.

Let $S=S_{\mathcal{V}}:=\left\{ Z=(0,\zeta,\eta) \in \mathrm{spt}\|V^\theta\| \cap B_1:\Theta(\mathcal{V},Z)\ge 1-\cos \theta(Z)\right\}$.
We aim to demonstrate that $S \cap B_{\frac{1}{16}}$ can be written as a graph of a $C^{1,\alpha}$ function defined on $\left\{ 0 \right\}\times B_{\frac{1}{16}}^{n-1}$.

For any $(0,y) \in \left\{ 0 \right\}\times B_{\frac{1}{16}}^{n-1}$, we show that there exists a unique point $Z=(0,\zeta,y) \in S$.

At first, if $\left\{ 0 \right\}\times \mathbb{R}\times \left\{ y \right\}\cap S=\emptyset $, then, since $S$ is closed and $0 \in S$, we can find $\delta \in (0,\frac{1}{16})$ such that
\begin{equation}
	\left\{ 0 \right\}\times B_\delta^{n-1}(y)\cap S=\emptyset \text{ and }\left\{ 0 \right\}\times  \partial B_\delta^{n-1}(y)\cap S\neq \emptyset.
	\label{eq:pfSmallBall}
\end{equation}
We choose $Z_0 \in \left\{ 0 \right\}\times \partial B_\delta^{n-1}(y)\cap S$ and consider $\mathcal{V}':=(\mathcal{V}^{Z_0})^{\Gamma^{Z_0}}_{0,6\delta}$, $H'=H^{Z_0}$.
From \eqref{eq:pfSmallBall}, there exists a point $y' \in \partial B_{\frac{1}{3}}^{n-1}(0)$ such that
\[
	\Gamma^{Z_0}\circ L^g_{Z_0}(\left\{ 0 \right\}\times B_{\frac{1}{3}}^{n-1}(y'))\cap B_1\cap \left\{ \Theta(\mathcal{V}',Z)\ge 1-\cos\theta'(Z) \right\}=\emptyset.
\]
When $\varepsilon$ is small enough, from \eqref{eq:pfLimitGamma} and Proposition \ref{prop_Pi2Delta}, we can find $y'' \in B_{\frac{1}{2}}^{n-1}$ such that $B_{\frac{1}{4}}(y'')\subset \Gamma^{Z_0}\circ L^g_{Z_0}(\left\{ 0 \right\}\times B_{\frac{1}{3}}^{n-1}(y'))$.
Therefore,
\begin{equation}
	B_{\frac{1}{4}}(0,y'')\cap \left\{ \Theta(\mathcal{V}',Z)\ge 1-\cos\theta'(Z) \right\}=\emptyset.
	\label{eq:pfEmptyIntersection}
\end{equation}

On the other hand, for sufficiently small $\varepsilon$, we can apply Lemma \ref{lem_noHole} to conclude
\[
	B_{\frac{1}{32}}(0,y'')\cap \left\{ \Theta(\mathcal{V}',Z)\ge 1-\cos\theta'(Z) \right\}\cap \partial H\neq \emptyset,
\]
which contradicts with \eqref{eq:pfEmptyIntersection}. 
Therefore, $\left\{ 0 \right\}\times \mathbb{R}\times \left\{ y \right\}\cap S$ is non-empty.

We choosing two distinct points $Z_1=(0,\zeta_1,\eta_1),Z_2=(0,\zeta_2,\eta_2) \in S\cap B_{\frac{1}{16}}$, and write $r=\left|Z_2-Z_1\right|$, $\mathcal{V}'=(\mathcal{V}^{Z_1})^{\Gamma^{Z_1}}_{0,6r}$, $H'=H^{Z_1}$.
By choosing $\varepsilon$ small enough, we know $Z_2'=(0,\zeta_2',\eta_2'):=(\Gamma^{Z_1}\circ L^g_{Z_1}\circ \eta_{Z_1,3r})(Z_2) \in B_{\frac{1}{2}}\backslash B_{\frac{1}{4}}$. By applying Corollary \ref{cor_noHole} with $(\mathcal{V}',H')$ in place of $(\mathcal{V},H)$, we obtain
\[
	\left|\zeta_2'\right|\le \frac{1}{16},
\]
which leads to  $\frac{\left|\zeta_2'\right|}{\left|\eta_2'\right|}\le \frac{1}{3}$.
In view of \eqref{eq:corDistSing} in Corollary \ref{cor_noHole}, by choosing $\varepsilon$ small enough again, we see that
\begin{equation}
	\left|\zeta_1-\zeta_2\right|\le \frac{1}{2}\left|\eta_1-\eta_2\right|.
\label{eq:pfLip}
\end{equation}
In particular, $\eta_1=\eta_2$ will imply $\zeta_1=\zeta_2$. 
Consequently, we know $\left\{ 0 \right\}\times \mathbb{R}\times \left\{ y \right\}\cap S$ contains a unique point, allowing us to define a function $\varphi:B_{\frac{1}{16}}^{n-1} \rightarrow \mathbb{R}$ such that
\[
	\tilde{\varphi}(y):=(0,\varphi(y),y) \in S.
\]
where \eqref{eq:pfLip} implies $\varphi$ is indeed a Lipschitz function.

Thus, $\varphi$ is differentiable almost everywhere.
Suppose $\varphi$ is differentiable at $y_0$, then from Corollary \ref{cor_noHole} and \eqref{eq:pfSmallBall}, we deduce
\begin{equation}
	\Gamma^{\tilde{\varphi}(y_0)}\circ L^g_{\tilde{\varphi}(y_0)}(\left\{ 0 \right\}\times \mathbb{R}^{n-1})=D\tilde{\varphi}(y_0)(\left\{ 0 \right\}\times \mathbb{R}^{n-1}).
	\label{eq:pfGammaLine}
\end{equation}
In other word, $\Gamma^{(0,\varphi(y_0),y_0)}(\left\{ 0 \right\}\times \mathbb{R}^{n-1})$ can be determined by the tangent space of $S$ at $(0,\varphi(y_0),y_0)$.

We fix two points $y_1$ and $y_2$ where $\varphi$ is differentiable.
We denote $Z_1=(0,\varphi(y_1),y_1)$, $Z_2=(0,\varphi(y_2),y_2)$, $r=\left|Z_2-Z_1\right|$.
By applying Proposition \ref{prop_ImproveExcessRho} with $(\mathcal{V}',H'):=(\mathcal{V}^{Z_1,3r},H^{Z_1})$ in place of $(\mathcal{V},H)$, we obtain
\[
	\left|\Gamma^{Z_2'}-\mathrm{Id}\right|\le
	C \mathcal{E}(\mathcal{V}',H')\le Cr^\alpha \mathcal{E}(\mathcal{V},H),
\]
where $Z'_2=\Gamma^{Z_1}\circ \Pi^g_{Z_1,3r}(Z_2)$.

We have two ways to express the tangent space of $S_{\mathcal{V}'}$.
One expression is
\[
	T_{Z_2'}S_{\mathcal{V}'}=\Gamma^{Z_1}\circ L^g_{Z_1}\circ(\Gamma^{Z_2}\circ L^g_{Z_2})^{-1}(\left\{ 0 \right\}\times \mathbb{R}^{n-1}).
\]

The alternative expression is
\[
	T_{Z_2'}S_{\mathcal{V}'}=\Gamma^{Z_2'}\circ L^{g'}_{Z_2'}(\left\{ 0 \right\}\times \mathbb{R}^{n-1}).
\]

Therefore,
\[
	\left|\Gamma^{Z_1}\circ L^g_{Z_1}(\left\{ 0 \right\}\times \mathbb{R}^{n-1})-\Gamma^{Z_2}\circ L^g_{Z_2}(\left\{ 0 \right\}\times \mathbb{R}^{n-1})\right|\le C r^\alpha \mathcal{E}(\mathcal{V},H).
\]

From \eqref{eq:pfGammaLine}, we find
\[
	\left|D\varphi(y_1)-D\varphi(y_2)\right|\le 
	C |Z_1-Z_2|^{\alpha}\mathcal{E}(\mathcal{V},H)\le C \left|y_1-y_2\right|^{\alpha}\mathcal{E}(\mathcal{V},H).
\]

In conclusion, we know $\varphi$ is $\alpha$-H\"older continuous.

Now, we denote $P$ to be the plane in which $H$ lies, and $S'$ the orthogonal projection of $S$ to the hyperplane $P$.
From Theorem \ref{thm__l_2_estimate}, we know $S'\subset \left\{ r<\frac{1}{64} \right\}$.
Let $\Omega$ be the component of $P\cap B_{\frac{1}{16}}\backslash T_{\mathcal{V}}'$ which contain $H\cap B_1\backslash \left\{ r\ge \frac{1}{16} \right\}$.

For any $Z \in S \cap B_{\frac{1}{16}}$ and $r \in (0,1)$, 
applying Theorem \ref{thm__l_2_estimate} with $((V',W',\theta',g'),H)=(\mathcal{V}^{Z,r},H^Z)$ allow us to write $\mathrm{spt}\|V'\|\cap B_{\frac{13}{16}}\cap \left\{ r\ge \frac{1}{32} \right\}$ as a graph of some $C^{1,\beta}$ function over plane $H^Z$ with
\[
	|Du ^{Z,r}|\le C r^\beta\mathcal{E}(\mathcal{V},H).
\]

Hence, by choosing $\mathcal{E}(\mathcal{V},H)$ small enough, we can ensure that $\mathrm{spt}\|V\|\cap B_\frac{r}{2}(Z)\backslash \left\{ r(X-Z)\le \frac{r}{16} \right\}$ can be written as a graph of some function $u$ with
\begin{equation}
	\left|Du\right|\le \tan\left(\frac{1}{2}\Lambda_0\right).
	\label{eq:pfBoundedGrad}
\end{equation}
By the arbitrariness of $Z$ and $r$ and the unique continuation, we know $\mathrm{spt}\|V\|\cap B_{\frac{1}{16}}$ can be written as a graph of a Lipschitz function $u$ defined on $\Omega$ with boundary value
\[
	u|_{\partial \Omega\cap T'}=\varphi.
\]

Utilizing the standard regularity theory (e.g., \cite{Morrey1966Book}), it is infer that $u$ is a $C^{1,\gamma}$ function defined on $B_{\frac{1}{32}}\cap \Omega$ for some $\gamma \in (0,1)$.
We write $\Sigma$ to be the graph of $u$ over $B_{\frac{1}{32}}\cap \Omega$ and $U$ to be the connected region in $B_{\frac{1}{32}}\cap \partial \mathbb{H}^{n+1}\backslash S$ which contain $B_{\frac{1}{32}}\cap \{ x_1=0, x_2<-\frac{1}{64} \}$.
By the gradient estimate \eqref{eq:pfBoundedGrad}, we know $\measuredangle(\Sigma,U) \in (\frac{\pi}{2}+\frac{\Lambda_0}{2},\pi-\frac{\Lambda_0}{2})$ and hence $\measuredangle _g(\Sigma,U) \in (\frac{\pi}{2}+\frac{\Lambda_0}{3},\pi-\frac{\Lambda_0}{3})$ by choosing $\mathcal{E}(\mathcal{V},H)$ small enough.

Finally, for any $X \in \partial \Sigma\cap B_{\frac{1}{32}}$, we only need to show $\measuredangle _g(\Sigma,U)=\theta$ at $X$.
To see this, we consider the tangent cone $\mathcal{V}'=(V',W',\theta(X),\delta)$ of $\mathcal{V}$ at $X$.
Since the regularity for $\mathrm{spt}\mathcal{V}$ has been established, we can write $V'=|P'\cap \mathcal{B}_1|$ for some $n$-dimensional subspace of $\mathbb{R}^{n+1}$.
Notably, \eqref{eq:pfBoundedGrad} implies $P'$ is not orthogonal to $P$.
As $V'-\cos \theta(X)W'$ is stationary in free boundary sense, we know the only possibility is, $W'=|H^0\cap \mathcal{B}_1|$ and 
$\measuredangle (P'\cap \mathcal{B}_1, \mathrm{spt}\|W'\|)=\theta(X)$.
Consequently, we obtain $\measuredangle _g(\Sigma,U)=\theta$ along $\partial \Sigma\cap B_{\frac{1}{32}}$.
\end{proof}

\section{Free boundary case}%
\label{sec:free_boundary_case}

In this section, we aim to establish the regularity for almost stationary varifolds in $\mathcal{B}_1$ in free boundary case.
Some proofs in this part are omitted unless they are different from the previous part.

For any $\mathcal{V} \in \mathcal{RIV}(\mu_0)$ with $\mu_0 \in (0,1)$, we define $\mu_F(\mathcal{V})$ by
\[
	\mu_F(\mathcal{V})=\inf \left\{ \mu: \mathcal{V} \text{ is }\mu \text{-stationary in free boundary sense}\right\}.
\]
Recall that $\mu$-stationary in free boundary sense has been defined in Definition \ref{def_muStat}.
We define the $L^2$-excess in free boundary sense by
\[
	\mathcal{E}^2_F(\mathcal{V},H):=\int_{ \mathcal{B}_1} \mathrm{dist}^2(X,H)d\|V\|(X)+\mu_F(\mathcal{V}).
\]

\begin{hypothesis}
	\label{hyp_free_boundary}
	Given $\varepsilon>0$, a VPCA-quadruple $\mathcal{V} \in \mathcal{RIV}(\mu)$ for $\mu$ small enough, 
	and a half hyperplane $H$ in $\mathbb{H}^{n+1}$,
	we say $(\mathcal{V},H)$ satisfies $\varepsilon$-Hypothesis if the following conditions hold.
	\begin{enumerate}[\normalfont(a)]
		\item $\mathcal{V}$ is $\mu$-stationary in free boundary sense with $\mu<\varepsilon$, $g(0)=\delta$.
			\label{it:hypFBmuStat}
		\item $\Theta(\mathcal{V},0)\ge 1-2\varepsilon$ and $\|V\|(\mathcal{B}_1)<\frac{3}{4}\omega_n$,
		\item $H$ can be written as $H=H^{\theta'}$ with $|\theta'-\frac{\pi}{2}|<\varepsilon$.
		\item The half hyperplane $H$ satisfies $\mathcal{E}^2_F(V,H)<\varepsilon$.
	\end{enumerate}

\end{hypothesis}
Note that the condition \ref{it:hypFBmuStat} implies $\mu_F(\mathcal{V})<\varepsilon$.

Analogue to Theorem \ref{thm_l2positiveAngle}, we have the following regularity result.

\begin{theorem}
	\label{thm_l2freeFB}
	There exists $\varepsilon=\varepsilon(n)\in (0,1)$ such that if $(\mathcal{V},H)$ satisfies $\varepsilon$-Hypothesis, then $\mathrm{spt}\|V\|\cap B_{\frac{1}{32}}=\Sigma$ where $\Sigma$ is a $C^{1,\gamma}$-hypersurface with boundary in $B_{\frac{1}{32}}$ such that $\partial \Sigma \cap B_{\frac{1}{32}}\subset \partial\mathbb{H}^{n+1}$ and $\measuredangle _g(\Sigma,U)=\theta$ or $\frac{\pi}{2}$ along $\partial\Sigma$ where $U$ is the connected component of $\mathcal{B}_{\frac{1}{32}}\cap \partial \mathbb{H}^{n+1}\backslash \Sigma$ positioned below the $\partial\Sigma$.
\end{theorem}

Analogous to Theorem \ref{thm__l_2_estimate}, we give the following $L^2$-estimate.
\begin{theorem}
	
	For any $\tau \in (0,\frac{1}{8}),\omega \in (0,1)$, there exists $\varepsilon_0=\varepsilon_0(n,\tau) \in (0,\frac{1}{2})$, $\beta=\beta(n) \in (0,1)$ such that if $(\mathcal{V},H)$ satisfies $\varepsilon_0$-Hypothesis.
	Then, the following conclusions hold,
	\begin{enumerate}[\normalfont(a)]
		\item $V\lfloor(B_{\frac{{13}}{16}}\backslash \left\{ r<\tau \right\})=|\mathrm{graph}u\cap B_{\frac{{13}}{16}}\backslash \left\{ r<\tau \right\}$
			where $u \in C^{1,\beta}(B_{\frac{{13}}{16}}\cap H \backslash \left\{ r<\frac{\tau}{2} \right\},H^\bot )$ 
			and $\mathrm{dist}(X+u(X),H)=|u(X)|$ for $X \in B_{\frac{{13}}{16}}\cap H\backslash \left\{ r<\tau \right\}$.
		\item $\int_{ B_{\frac{3}{4}}} \frac{|X^{\bot_S} |^2}{|X|^{n+2}}dV(X,S)\le C \mathcal{E}^2_F(\mathcal{V},H)$.
		\item $\sum_{j =3}^{n+1}\int_{ B_{\frac{3}{4}}} |e_j^{\bot_S} |^2dV(X,S)\le C \mathcal{E}^2_F(\mathcal{V},H)$.
		\item $\int_{ B_{\frac{3}{4}}} \frac{\mathrm{dist}^2(X,H)}{|X|^{n+2-\omega}}d\|V\|(X)\le C_1 \mathcal{E}^2_F(\mathcal{V},H)$.
	\end{enumerate}
	Here, $C=C(n)$ and $C_1=C_1(n,\omega)$.
	\label{thm:free_boundary_l_2_estimate}
\end{theorem}
\begin{proof}
	At first, we adapt Lemma \ref{lem_graph_over_cone} through Lemma \ref{lem_change_center} to the case of free boundary.
	Basically, this adaptation requires revisiting the same argument with $\mathrm{dist}(X,H)$ in place of $\mathrm{dist}(X,\mathrm{spt}\|\boldsymbol{C}\|)$, $\mu_F(\mathcal{V})$ in place of $\mu(\mathcal{V})$, and use the fact that the first variation of $\mu$-stationary varifold in free boundary sense can be written as
	\begin{equation}
		\int_{ } \mathrm{div}_S \varphi dV(X,S)\le
		C\mu_F(\mathcal{V}) \int_{ }(|\varphi|+|D_S\varphi|) d\|V+W\|
		\label{eq:pf1stVarFB}
	\end{equation}
	for any $\varphi \in \mathfrak{X}^1_{c,\tan}(\mathcal{B}_1)$.

	To establish Theorem \ref{thm:free_boundary_l_2_estimate}, we claim
	\[
		\|V\|(\mathcal{B}_\rho)\le \frac{(1+C\mu_F(\mathcal{V}))\rho}{n} \frac{d}{d\rho} \|V\|(\mathcal{B}_\rho)+C\mu_F(\mathcal{V}) \rho^n,\quad \text{ for a.e. }\rho \in (0,1)
	\]
	for some $C=C(n)$ for $\varepsilon_0$ small enough.
	Then, the monotonicity formula yields
	\[
		\int_{ B_{\frac{13}{16}}} 
		C \frac{\left|X^{\bot _S}\right|^2}{|X|^{n+2}}d\|V\|(X)\le \int_{ }  \psi^2(|X|)d\|V\|(X)-\int_{ } \psi^2(|X|)dH(X)  +C\mu_F(\mathcal{V}).
	\]
	Here, $\psi$ is a smooth function defined by \eqref{eq:pfTestPsi} and the fact that $\Theta(\mathcal{V},0)\ge 1-2\mu_F(\mathcal{V})$ has been used.

	Next, we can choose $\varphi = \psi^2(|X|)x$ in \eqref{eq:pf1stVarFB} and after a detailed computation, we arrive at
	\begin{equation}
		\int_{ B_{\frac{13}{16}}} \sum_{i=3 }^{n+1}|e_i^{\bot_S} |^2 dV(X,S)+
		\int_{ B_{\frac{13}{16}}(0)} \frac{|X^{\bot_S} |^2}{|X|^{n+2}}dV(X,S)\le
		C \mathcal{E}^2_F(\mathcal{V},H).
	\end{equation}
	Finally, by setting $\varphi=\zeta^2|X|^{-n+\omega}\mathrm{dist}^2(X,H) \frac{X}{|X|^2}$ in \eqref{eq:pf1stVarFB} where $\zeta=1$ on $B_{\frac{3}{4}}$ and $\zeta=0$ outside of $B_{\frac{{13}}{16}}$, it follows that
	\[
		\int_{ B_{\frac{3}{4}}} \frac{\mathrm{dist}^2(X,H)}{|X|^{n+2-\omega}}d\|V\|(X)\le C\mathcal{E}^2_F(\mathcal{V},H).
	\]
\end{proof}

Similar to Lemma \ref{lem_change_center}, we show the following lemma specific to the free boundary case.
\begin{lemma}
	\label{lem_changeCenterFB}
	Let $r\in (0,\frac{1}{3})$. 
	For any $\varepsilon \in (0,1)$, there exists $\varepsilon_0=\varepsilon_0(\varepsilon,n, r) \in (0,\varepsilon)$ such that if $(\mathcal{V},H)$ satisfies $\varepsilon_0$-Hypothesis, then for any $X \in B_{\frac{5}{8}}(0)\cap \partial \mathbb{H}^{n+1}$ with $\Theta(\|\mathcal{V}\|,X)\ge 1-2\varepsilon$, we have
	$(\mathcal{V}_{X,r}, L_X^g(H))$ satisfies $\varepsilon$-Hypothesis.
\end{lemma}
\begin{proof}
	The proof is similar to the proof of Lemma \ref{lem_change_center}.
\end{proof}

Following from Theorem \ref{thm:free_boundary_l_2_estimate} and Lemma \ref{lem_changeCenterFB}, we derive the following corollary.
\begin{corollary}
	\label{cor_noHoleFB}
	For any $\rho \in (0,\frac{1}{4}]$, $\omega \in (0,1)$, there exists $\varepsilon=\varepsilon(n,\rho)$ such that if $(\mathcal{V},H)$ satisfies $\varepsilon$-Hypothesis, then for each $Z = (0,\zeta,\eta) \in \mathrm{spt}\|V\|\cap B_{\frac{5}{8}}\cap \partial \mathbb{H}^{n+1}$ with $\Theta(\mathcal{V},Z)\ge 1-2\varepsilon$, we have
	\begin{equation}
		|\zeta|^2\le C\mathcal{E}^2_F(\mathcal{V},H), 
		\label{eq:corDistSingFB}
	\end{equation}
	for some $C=C(n)$, and
	\begin{equation}
	\int_{ B_{\frac{\rho}{2}}(Z)} \frac{\mathrm{dist}^2(X,H_Z)}{|X-Z|^{n+2-\omega}}d\|V\|(X)\le
	\frac{C}{\rho^{n+2-\omega}} \int_{ B_\rho(Z)} \mathrm{dist}^2(X,H_Z)d\|V\|(X)+C\mu \rho^{\omega},
	\label{eq:corL2ImproveExcessFB}
\end{equation}
for some $C=C(n,\omega)$,
where the half-hyperplane $H_Z$ defined as $H_Z:=\tau_{-Z}(H)$.
\end{corollary}

\begin{lemma}
	\label{lem_noHoleFB}
	For any $\delta \in (0,\frac{1}{16})$ and suppose $(\mathcal{V},H)$ satisfies $\varepsilon$-Hypothesis for some $\varepsilon=\varepsilon(n,\delta)$ small enough.
	Then,
	\begin{enumerate}[\normalfont(a)]
		\item $B_\delta(0,y)\cap \left\{ \Theta(\mathcal{V},Z)\ge 1-2\varepsilon \right\}\cap \partial \mathbb{H}^{n+1}\neq \emptyset $ for each $(0,y) \in \left\{ 0 \right\}\times B_{\frac{1}{2}}^{n-1}(0)$.
		\item For $\omega \in (0,1)$ and $\sigma \in [\delta,\frac{1}{16})$, we have
			\[
				\int_{ B_{\frac{1}{2}}(0)\cap \left\{ r<\sigma \right\}} \mathrm{dist}^2(X,\mathrm{spt}\|\boldsymbol{C}_H\|)d\|V^\theta\|(X)\le
				C \sigma^{1-\omega}
				\mathcal{E}^2(\mathcal{V},H),
			\]
			where $C=C(n,\omega)$.
	\end{enumerate}
\end{lemma}
\begin{proof}
	We only provide the proof for the first statement as the approach to the second statement differs slightly from that in Lemma \ref{lem_noHole}.
	On the contrary, let us assume there exists $\delta \in (0,\frac{1}{16})$ and a consequence of $(\mathcal{V}_k,H_k)=((V_k,W_k,\theta_k,g_k),H_k)$ satisfying $\varepsilon_k$-Hypothesis for some $\varepsilon_k\rightarrow 0^+$, $\mathcal{V}_k\rightarrow \mathcal{V}=(V,W,\theta,\delta)$, $H_k\rightarrow H$ and $(0,y_i)\rightarrow (0,y) \in \left\{ 0 \right\}\times B_{\frac{2}{3}}^{n-1}$ such that
	\begin{equation}
		B_\delta(0,y_i)\cap \left\{ \Theta(\mathcal{V}_i,Z)\ge 1-2\varepsilon_i\right\}\cap \partial \mathbb{H}^{n+1}=\emptyset.
		\label{eq:pfnoLargeDensity}
	\end{equation}

	It is easy to see that $V=|H^{\frac{\pi}{2}}\cap \mathcal{B}_1|$ by the constancy theorem.

	According to Theorem \ref{thm:smallDensity}, we find that each $V_i$ is an integral $\varepsilon_i$-stationary rectifiable $n$-varifold in free boundary sense.
	Thus, for any $Z \in \mathrm{spt}\|V_i\|\cap \partial \mathbb{H}^{n+1}\cap B_{\frac{\delta}{2}}(0,y)$, we have $\Theta(\|V_i\|,Z)\ge \frac{1}{2}$ for $i$ large enough.
	Therefore,
	\[
		\Theta(\mathcal{V},Z)=
		2\Theta(\|V\|,Z)-\lim_{\rho\rightarrow 0^+} \frac{2}{\omega_n \rho^n}\int_{ \mathcal{B}_\rho}\cos \theta d\|W\|\ge 1-2 \sin(\varepsilon_i)\ge 1-2\varepsilon_i.
	\]
	This contradicts with \eqref{eq:pfnoLargeDensity}, which implies $\mathrm{spt}\|V_i\|\cap B_{\frac{1}{2}}(0,y)\cap \partial \mathbb{H}^{n+1}=\emptyset$.
	Therefore, $V_i$ is an integral $\varepsilon_i$-stationary rectifiable $n$-varifold in $B_{\frac{\delta}{2}}(0,y)$ (Extend the metric $g_i$ to $B_1$ if necessary).
	The Compactness Theorem (Theorem \ref{thm_compactness_theorem_for_n_varifolds_under_lipschitz_metric}) implies $V$ is an integral stationary rectifiable $n$-varifold in $B_{\frac{\delta}{2}}$.
	This finding contradicts with the fact that $V=|H^{\frac{\pi}{2}}\cap \mathcal{B}_1|$. (Indeed, If $V$ is stationary in $B_{\frac{\delta}{2}}$ and $\mathrm{spt}V \in \mathbb{H}^{n+1}$, the according to the Maximum Principle \cite{Solomon1989Maximum}, we can establish $\mathrm{spt}\|V\|\cap B_{\frac{\delta}{2}}\subset \partial\mathbb{H}^{n+1}$.)
\end{proof}

Now, we can describe the blow-ups in free boundary case.
We consider a sequence of VPCA-quadruples $\left\{ \mathcal{V}_k=(V_k,W_k,\theta_k,g_k) \right\}$, a sequence of half hyperplanes  $\left\{ H_k \right\}$, together with a positive sequence $\left\{ \varepsilon_k \right\}$ with $\lim_{k\rightarrow +\infty} \varepsilon_k=0$.
We assume $(\mathcal{V}_k,H_k)$ satisfies $\varepsilon_k$-hypothesis for each $k\ge 1$.
According to Theorem \ref{thm_compactnessRIV}, we assume $\mathcal{V}_k\rightarrow \mathcal{V}=(V,W,\theta,\delta)$ and $H_k\rightarrow H$ as $k\rightarrow +\infty$.
Notably, $H=H^{\frac{\pi}{2}}$.

We define
\[
	\mathcal{E}_{F,k}:=\mathcal{E}_F(\mathcal{V}_k,H_k).
\]

We choose two sequences, $\left\{ \delta_k \right\}, \left\{ \tau_k \right\}$, consisting of positive, decreasing numbers that converge to 0.
By Theorem \ref{thm:free_boundary_l_2_estimate}, Corollary \ref{cor_noHoleFB}, and Lemma \ref{lem_noHoleFB}, we have
\begin{enumerate}[\normalfont(a)]
	\item $V_k\lfloor(B_{\frac{13}{16}}\backslash \{ r<\tau_k \})=\left|\mathrm{graph}u_k \cap B_{\frac{13}{16}}\backslash \{ r<\tau_k \}\right|$ where $u_k \in C^{1,\beta}(B_{\frac{13}{16}}\cap H_k \backslash \{ r<\frac{\tau_k}{2} \},H_k^\bot)$ for some $\beta=\beta(n) \in (0,1)$ and satisfies $\mathrm{dist}(X+u_k(X),H_k)=|u_k(X)|$ for $X \in B_{\frac{13}{16}}\cap H_k \backslash \{ r<\frac{\tau_k}{2} \}$.
		\label{it:graphKFB}
	\item For each $Z=(0,\zeta,\eta) \in \mathrm{spt}\|V^{\theta_k}_k\|\cap B_{\frac{5}{8}}$ with $\Theta(\mathcal{V}_k,Z)\ge 1-2\varepsilon_k$, we have
		\begin{equation}
			|\zeta|\le C\mathcal{E}_{F,k}.
			\label{eq:BlowUpSingFB}
		\end{equation}
		\label{it:BlowUpSingFB}
	\item For each $(0,y) \in \{ 0 \}\times \mathbb{R}^{n-1}\cap B_{\frac{1}{2}}$, we have
		\begin{equation}
			B_{\delta_k}(0,y)\cap \{ \Theta(\mathcal{V}_k,Z)\ge 1-2\varepsilon_k\}\cap \partial \mathbb{H}^{n+1}\neq \emptyset.
			\label{eq:BlowUpNoHoleFB}
		\end{equation}
		and
		\begin{equation}
			\int_{ B_{\frac{1}{2}}\cap \{ r<\sigma \}} \mathrm{dist}^2(X,H_k)d\|V_k\|(X)\le C\sqrt{\sigma}\mathcal{E}_{F,k}^2,
			\label{eq:BlowUpNonConcenFB}
		\end{equation}
		for any $\sigma \in (\delta_k,\frac{1}{16}]$ where the constant $C$ does not depend on $\sigma$.
		\label{it:BlowUpNoHoleKFB}
	\item For each $\omega\in (0,1)$, $\rho \in (0,\frac{1}{4})$, and $Z=(0,\zeta,\eta) \in \mathrm{spt}\|V^{\theta_k}_k\|\cap B_{\frac{5}{8}}$ such that $\Theta(\mathcal{V}_k,Z)\ge 1-2\varepsilon_k$, we have
		\begin{align}
			{} & \int_{ B_{\frac{\rho}{4}}(Z)\cap H_k \backslash \{ r<\tau_k \}} \frac{\left|u_k-\zeta \sin \theta_k(0)\vec{n}_k \right|^2}{|X+u_k-Z|^{n+2-\omega}}d\mathcal{H}^n(X) \nonumber \\
			\le{} & \frac{C}{\rho^{n+2-\omega}} \int_{ B_\rho(Z)} \mathrm{dist}^2(X,\tau_{-Z}(H_k))d\|V_k\|(X)+C\mu_F(\mathcal{V}_k) \rho^{\omega},
			\label{eq:BlowUpImproveL2FB}
		\end{align}
		for $k$ sufficient large.
		Here, $C=C(n,\omega)$, and $\vec{n}_k$ is the unit normal vector of $H_k$ and pointing upward.
		\label{it:BlowUpImproveL2KFB}
\end{enumerate}
We define
\[
	\tilde{u}_k(X)=u_k(X+l_kr(X)),
\]
where $l_k$ is the constant such that $H_k= \{ X+l_kr(X)\vec{n}:X \in H \}$.
\begin{definition}
	\label{def_BlowUpFB}
	We say a sequence $\{ (\mathcal{V}_k,H_k) \}$ is a \textit{blow-up sequence in free boundary sense} if $(\mathcal{V}_k,H_k)$ satisfies $\varepsilon_k$-hypothesis for some $\varepsilon_k \rightarrow 0^+$.

	Given a blow-up sequence in free boundary sense $\left\{( \mathcal{V}_k,H_k )\right\}$, we say $v \in C^{1,\beta}_{\mathrm{loc}}(B_{\frac{3}{8}}\cap H)\cap L^2(B_{\frac{3}{8}}\cap H)$ is a \textit{blow-up in free boundary sense} if $v$ can be obtained by
	\[
		v=\lim_{k\rightarrow +\infty} \mathcal{E}_{F,k}^{-1}\tilde{u}_k.
	\]
\end{definition}

Analogous to \ref{prop_C2estimate}, we present the following continuity estimate.
\begin{proposition}
	\label{prop_C2estimateFB}
	If $v$ is a blow-up in free boundary sense, then $v \in C^2(\overline{B_{\frac{1}{8}}\cap H}, H^\bot)$ with the estimate
	\[
	\|v\|^2_{C^2(\overline{B_{\frac{1}{8}}\cap H},H^\bot)}\le C \int_{ B_{\frac{3}{8}}\cap H} |v(X)|^2d\mathcal{H}^n(X)
	\]
	where $C=C(n)$.
\end{proposition}

With the help of Proposition \ref{prop_C2estimateFB}, we can improve the $L^2$-excess for $\mathcal{V}_k$ as illustrated in the following proposition.
\begin{proposition}
	\label{prop_ImproveExcessRhoFB}
	There exists $\varepsilon=\varepsilon(n) \in (0,1)$, $\alpha=\alpha(n)\in (0,1)$, and a constant $C=C(n)$ such that the following holds.

	For any $(\mathcal{V},H)$ which satisfies $\varepsilon$-Hypothesis, and for any $Z \in \mathrm{spt}\|V^\theta\|\cap B_{\frac{1}{4}}\cap \partial \mathbb{H}^{n+1}$ with $\Theta(\mathcal{V},Z)\ge 1-2\varepsilon$, there exists an orthogonal rotation $\Gamma^Z \in \mathrm{SO}(n)\subset \mathrm{SO}(n+1)$, and a half plane $H^Z$ for which the following are true,
	\begin{align}
		\left|\Gamma^Z-\mathrm{Id}\right| ={} & C \mathcal{E}_F(\mathcal{V},H)\label{eq:propGammaZFB} \\
		\mathrm{dist}(H^Z,H) \le{} & C \mathcal{E}_F(\mathcal{V},H)\label{eq:propHZFB}\\
		\mathcal{E}_F((\mathcal{V}^Z)_{0,\rho}^{\Gamma^Z},H^Z)\le{}& C \rho^\alpha\mathcal{E}_F(\mathcal{V},H)\label{eq:propExcessZFB}.
	\end{align}
\end{proposition}
\begin{proof}
	[Proof of Theorem \ref{thm_l2freeFB}]
	We define
	\[
		S=\left\{ Z=(0,\zeta,\eta) \in \mathrm{spt}\|V\|\cap B_1\cap \partial \mathbb{H}^{n+1}:\Theta(\mathcal{V},Z)\ge 1-2\varepsilon\right\}.
	\]
	Using Proposition \ref{prop_ImproveExcessRhoFB}, we can show that $S\cap B_{\frac{1}{16}}$ is a graph of some H\"older continuity function $\varphi$ defined on $B_{\frac{1}{16}}\cap \left\{ x_1=x_2=0 \right\}$.
	Applying Proposition \ref{prop_ImproveExcessRhoFB} and Theorem \ref{thm:free_boundary_l_2_estimate} once more, we can establish that $\mathrm{spt}\|V\|$ can be written as a graph of a function $u$ defined on $B_{\frac{1}{16}}\cap H^{\frac{\pi}{2}}$ with the boundary value $u=\varphi$ on $B_{\frac{1}{16}}\cap \left\{ x_1=x_2=0 \right\}$.
	The standard regularity theory tells us $u$ is a $C^{1,\gamma}$ function defined on $B_{\frac{1}{32}}\cap H^{\frac{\pi}{2}}$for some $\gamma \in (0,1)$.

	Finally, for any $X \in \partial\Sigma \cap \partial\mathbb{H}^{n+1}$, we need to show $\measuredangle_g(\Sigma,U)=\theta$ or $\frac{\pi}{2}$ at $X$. 
	Analogue to the proof of Theorem \ref{thm_l2positiveAngle}, we can consider the tangent cone $\mathcal{V}'=(V',W',\theta(X),\delta)$ of $\mathcal{V}$ at $X$.
	The difference here is that we cannot assert $\mathrm{spt}\|V'\|$ is orthogonal to $\partial\mathbb{H}^{n+1}$ even if $\theta(X)\neq \frac{\pi}{2}$.
	Therefore, we conclude that $\measuredangle_g(\Sigma,U)=\theta$ or $\frac{\pi}{2}$ at $X$.
\end{proof}

\section{Proof of Main Results}%
\label{sec:proof_of_main_theorem}

Now, we are ready to prove our main results.
\begin{proof}
	[Proof of Theorem \ref{thm_densityG}]
	By the argument at the beginning of Subsection \ref{sub:VPA}, we only need to consider the VPCA-quadruple $\mathcal{V}=(V,W,\theta,g)$ satisfies items \ref{it:thetaG} to \ref{it:TangentHalfPlane} in Subsection \ref{sub:VPA}.

	Moreover, by Remark \ref{rmk_changeAngle}, after a possible rotation, we only need to consider the case $\theta(0)\in [\frac{\pi}{2},\pi)$ and $g(0)=\delta$.

	We choose $\varepsilon_1 \in (0,1)$ such that Theorem \ref{thm_l2freeFB} holds with $\varepsilon_1$ in place of $\varepsilon$.
	Then, we choose $\Lambda_0 =\frac{1}{4}\min \{(\pi-\theta(0)), \frac{\varepsilon_1}{2^n}\}$ and $\varepsilon_2\in (0,1)$ so that Theorem \ref{thm_l2positiveAngle} holds with $(\varepsilon_2,\Lambda_0)$ in place of $(\varepsilon,\Lambda_0)$.

	Next, let us analyze the tangent cone of $\mathcal{V}$ at $0$.
	We denote $\mathcal{V}'=(V',W',\theta(0),\delta):=\lim_{i\rightarrow +\infty} \mathcal{V}_{0,\rho_i}$ a tangent cone of $\mathcal{V}$ at $0$ where $\rho_i\rightarrow 0^+$ as $i\rightarrow +\infty$.

	For short, we write $\mathcal{V}_i=(V_i,W_i,\theta_i,g_i):=\mathcal{V}_{0,\rho_i}$.
	To estimate the $L^2$-excess of $\mathcal{V}_i$, we assume $\mathcal{V}_i$ is defined on $\mathcal{B}_2$ for $i$ large enough by Remark \ref{rmk_extendQuad}.

	Upon an appropriate rotation, we may assume $V'=|H_\theta\cap \mathcal{B}_2|$ for some $\theta_0 \in [0,\pi]$ given that $V$ has a multiplicity-one tangent half-hyperplane at $0$.

	The first case is $H^{\theta_0}=H^{\frac{\pi}{2}}$. 
	This implies $\theta(0)=\frac{\pi}{2}$ or $W'=0$, which leads to $V_i \rightarrow V'$ in the sense of Radon measure.
	Thus,
	\begin{align*}
		\Theta(\mathcal{V}_i,0)={} & 2\Theta(\|V_i\|,0)=1 \\
		\lim_{i\rightarrow +\infty} \|V_i\|(\mathcal{B}_1) ={} & \lim_{i\rightarrow +\infty} \frac{1}{\rho_i^n}\|V_i\|(\mathcal{B}_{\rho_i})=\omega_n \Theta(\|V_i\|,0)=\frac{1}{2}\omega_n\\
		\limsup_{i\rightarrow +\infty}\int_{ \mathcal{B}_1} \mathrm{dist}^2(X,H^{\frac{\pi}{2}})d\|V_i\|(X)  \le{}& \lim_{i\rightarrow +\infty} \int_{ } \eta(X)\mathrm{dist}^2(X,H^{\frac{\pi}{2}})d\|V_i\|(X)  \\
		={}& \int_{H^{\frac{\pi}{2}} } \eta(X) \mathrm{dist}^2(X,H^{\frac{\pi}{2}})d\mathcal{H}^n(X)=0,
	\end{align*}
	for any non-negative smooth function $\eta$ with compact support in $B_2$ and $\eta=1$ on $B_1$.
	We then choose $i$ large enough such that $\mathcal{V}_i \in \mathcal{RIV}(\frac{\varepsilon_1}{4})$ and $(\mathcal{V}_i,H^{\frac{\pi}{2}})$ satisfies $\varepsilon_1$-Hypothesis.
	By applying Theorem \ref{thm_l2freeFB}, we can establish the regularity result.

	The second case is $H^{\theta_0}\neq \frac{\pi}{2}$.
	Since $V'-\cos \theta(0)W'$ is stationary in free boundary sense, after a rigid rotation if needed, we find the only possibility is $W'=|\left\{ x_1=0,x_2\le 0 \right\}\cap \mathcal{B}_2|$ and $\theta_0=\theta(0)$.
	Similar to the previous case, we deduce
	\begin{align*}
		\Theta(\mathcal{V}_i,0)={} &  1-\cos \theta(0)\\
		\lim_{i\rightarrow +\infty} \|V_i-\cos \theta_iW_i\|(\mathcal{B}_1)={} & \frac{1}{2}\omega_n(1-\cos \theta(0)),\\
		\limsup_{i\rightarrow +\infty} \int_{ \mathcal{B}_i} \mathrm{dist}^2(X,\boldsymbol{C}_{H^{\theta_0}})d\|V_i-\cos \theta_iW_i\|(X) ={}& 0.
	\end{align*}

	Here are two subcases to consider.
	If $\theta(0) \ge \frac{\varepsilon_1}{2^{n+1}}+\frac{\pi}{2}$, we choose $i$ large enough ensuring $\mathcal{V}_i \in \mathcal{RIV}(\frac{\varepsilon_2}{4})$ and $(\mathcal{V}_i,H^{\theta_0})$ satisfies $(\varepsilon_1,\Lambda_0)$-Hypothesis.
	Thus, Theorem \ref{thm_l2positiveAngle} is applied to obtain the regularity result.

	In the case where $\theta(0)<\frac{\varepsilon_1}{2^{n+1}}+\frac{\pi}{2}$, the only difference is showing $\mathcal{E}_F(\mathcal{V}_i,H^{\theta_0})<\varepsilon_1$ to utilize Theorem \ref{thm_l2freeFB}.
	For this purpose, we pick $i$ large such that $\theta_i\le \frac{3\varepsilon_1}{2^{n+2}}$ on $\mathcal{B}_1\cap \partial\mathbb{H}^{n+1}$ and
	\[
		\int_{ \mathcal{B}_1} \mathrm{dist}^2(X,H^{\theta_0})d\|V_i-|\cos \theta_i| W_i\|(X)< \frac{\varepsilon_1}{2}
	\]
	as $V_i-|\cos \theta_i|W_i\rightarrow V'-\cos \theta(0) W'$ in the sense of Radon measure on $\mathcal{B}_2$.
	Note that
	\[
		\int_{ \mathcal{B}_1} |\cos \theta_i| d\|W_i\|(X)\le \frac{3\varepsilon_1}{2^{n+2}}\omega_n\le \frac{3\varepsilon_1}{4}.
	\]
	Thus, $\mathcal{E}_F^2(\mathcal{V}_i,H^{\theta_0})<\varepsilon_1$.
	It is straightforward to see that $\Theta(\mathcal{V},0)=1-\cos \theta(0)\ge 1-2\varepsilon_1$.
	Therefore, $(\mathcal{V}_i,H^{\theta_0})$ satisfies $\varepsilon_1$-Hypothesis for $i$ large enough.
	Hence, Theorem \ref{thm_l2freeFB} is applied to get the regularity result.
\end{proof}

\begin{proof}
	[Proof of Corollary \ref{cor_contactAngleG}]
	Indeed, this follows directly from Theorem \ref{thm_densityG}, as the contact angle is readily determined from the tangent cone of $V$ at point $Y$.
\end{proof}

\begin{proof}
	[Proof of Theorem \ref{thm_densityG}]

Let us proceed by contradiction.
	Suppose there exists a sequence of positive numbers $\left\{ \varepsilon_k \right\}$ with $\lim_{k\rightarrow +\infty} \varepsilon_k=0$ such that for every $V$ satisfying the condition of Theorem \ref{thm_densityG}, the regularity result fails.
	Upon the initial discussion in Subsection \ref{sub:VPA} and after an appropriate scaling, for each $i$, we can find a VPCA-quadruple $\mathcal{V}_i=(V_i,W_i,\theta_i,g_i)$ satisfying the following conditions,
	\begin{enumerate}[\normalfont(a)]
		\item $g_i(0)=\delta$, and $0 \in \mathrm{spt}\|V_i\|$,
		\item $\mathcal{V}_i \in \mathcal{RIV}(\varepsilon_i)$,
			\label{it:pfVisRIV}
		\item $n=2$ and $\theta_i \in [\Lambda_0,\pi-\Lambda_0]$, or $\theta_i \in (\frac{\pi}{2}-\varepsilon_i,\frac{\pi}{2}+\varepsilon_i)$,
			\label{it:pfTheta}
		\item $\|V_i\|(B_1)\le \omega_n\left(\frac{1}{2}+2\varepsilon_i\right)$.
		\label{it:pfLimit}
	\end{enumerate}
	By Compactness Theorem (Theorem \ref{thm_compactnessRIV}), we may assume $\mathcal{V}_i\rightarrow \mathcal{V}=(V,W,\theta,\delta)$ where $\theta$ is a constant in $[\Lambda_0,\pi-\Lambda_0]$.
	From the condition \ref{it:pfLimit}, we know $\|V\|(B_1)\le \frac{1}{2}$.

	Now, we claim that either $V=|H^{\theta}\cap \mathcal{B}_1|$ or $V=|H^{\frac{\pi}{2}}\cap \mathcal{B}_1|$.

	In the case $n\neq 2$, condition \ref{it:pfTheta} implies $\theta=\frac{\pi}{2}$.
	By Theorem \ref{thm_compactnessRIV} and Theorem \ref{thm_compactness_theorem_for_n_varifolds_under_lipschitz_metric}, we find $V$ is an integral stationary rectifiable $n$-varifold in free boundary sense in $\mathcal{B}_1$.
	Consequently, the only possibility is $\mathrm{spt}\|V\|=H^{\frac{\pi}{2}}\cap \mathcal{B}_1$ after a possible rotation.

	For the case $n=2$, we need to analyze the structure of $V$.
	Without loss of generality, we assume $\theta\ge \frac{\pi}{2}$.
	For any $X \in \mathrm{spt}\|V\|\backslash \partial\mathbb{H}^{3}$, since $V$ is a stationary cone, each tangent cone of $V$ at $X$ should correspond to the union of integer-density multiple half-planes with a common boundary by a dimension reduction argument.
	Hence, $\Theta(\|V\|,X)=\frac{m}{2}$ for some integer $m\ge 2$.

	By the property of stationary cone and Remark \ref{rmk_reflectionFB}, we have $\Theta(\|V\|,X)\le \Theta(\mathcal{V},0)\le 2-\Lambda_0$, which implies $m\le 3$.
	By Allard regularity and regularity of multiplicity-one stationary varifold near cylindrical cone due to Simon \cite{Simon1993cylindrical} (see also \cite{AllardAlmgren1976OneDimStructure}), we know $V \lfloor(\mathcal{B}_1\backslash \partial\mathbb{H}^2)=\sum_{i \in \mathcal{I}}|C(\gamma_i)\cap B_1|$ where $\mathcal{I}$ is an index set and $C(\gamma_i)$ is a cone associated with $\gamma_i$ defined by
	\[
		C(\gamma_i):=\left\{ rX:X \in \gamma_i, r \in (0,+\infty) \right\}
	\]
	where $\gamma_i$ represents closed geodesic arc in $\partial B_1\cap \mathbb{H}^{3}$.
	Now, we write $\mathcal{P}$ to be the set of all endpoints of $\gamma_i$ for each $i \in \mathcal{I}$. 
	It follows that $\{\gamma_i\}_{i \in \mathcal{I}}$ satisfies the following conditions,
	\begin{enumerate}[\normalfont(a)]
		\item $\gamma_i$ intersects $\gamma_j$ can only happen at the endpoint of $\gamma_i$ and $\gamma_j$ for different $i,j \in \mathcal{I}$.
		\item For each $X \in \mathcal{P}$, we have either $X \in \partial \mathbb{H}^3$ or there exists three geodesic arcs $\gamma_{i_1},\gamma_{i_2},\gamma_{i_3}$ such that $X$ is the common endpoint of $\gamma_{i_1},\gamma_{i_2},\gamma_{i_3}$ and the contact angle between $\gamma_{i_p},\gamma_{i_q}$ is $\frac{\pi}{3}$ for $1\le p\neq q\le 3$ at $X$.
	\end{enumerate}
	
	Utilizing the result from spherical geometric, we deduce that $\mathcal{I}$ contains at most 1 element and that $\gamma_i$ is a half-circle for $i \in \mathcal{I}$.
	This result is elementary and interesting readers can refer to Appendix \ref{sec:append_stationaryNetwork} for more details.

	Hence, the claim is proved.

	Similar to the first case in this proof, we can apply either Theorem \ref{thm_l2freeFB} or Theorem \ref{thm_l2positiveAngle} depending on the value of $\theta$ for some $i$ large enough to obtain the regularity result.
	This is a contradiction to our initial assumption.
\end{proof}


\appendix

\section{Almost Stationary Varifolds under $C^1$ Metrics}%
\label{sec:varifolds}

In this section, we introduce the notion of $\mu$-stationary varifold under some $C^1$ metrics.
The majority of the results here align closely with the standard theory of varifold.
Therefore, we will not delve into detailed proofs for some statements, as they follow analogously from arguments found in standard texts, such as Chapter 8 of Simon's book \cite{simon1983lectures}.

We assume $U$ is a domain in $\mathbb{R}^{n+k}$.

\subsection{Theory of Varifolds under $C^1$ Metrics}%
\label{sub:theory_of_varifolds_under_lipschitz_metrics}

This subsection gathers some fundamental aspects of varifolds under $C^1$ metrics.
Let $g$ denote a $C^1$ metric defined on $U$.

\begin{definition}
	\label{def_WeightMeasureG}
	For any $n$-varifold $V \in G(n,U)$, we define the \textit{weight measure under metric $g$} as
	\begin{equation}
		\|V\|_g(K)=\int_{ K } \sqrt{\mathrm{det}g_S}dV(X,S).
		\label{eq:defWeightMeasure}
	\end{equation}
\end{definition}
Here, we define $\mathrm{det}g_S$ as
\[
	\mathrm{det}g_S:=\mathrm{det} \left( \left< \tau_i,\tau_j \right>_g \right)_{i,j=1}^n
\]
where $\left\{ \tau_i \right\}_{i=1}^n$ represents an orthonormal basis of $S$ with respect to the Euclidean metric $\delta$.

It is straightforward to observe that the measures $|V|_g$ and $|V|$ are mutually absolutely continuous.

\begin{remark}
	Should $V$ correspond to an $n$-dimensional submanifold $M$ in $U$ with a density of 1, then $\|V\|_g$ is the $n$-dimensional volume measure of $M$ under the metric $g$.
\end{remark}

By computing the first variation of the weight measure under the metric $g$, we define the first variation $\delta^g V$ as
\[
	\delta^g V(\varphi):= \left.\frac{d}{dt}\right|_{t=0}\|(\phi_t)_{\#}V\|_g(W)=\int_{W \times G(n+k,n)} \mathrm{div}_S^g \varphi(x)\sqrt{\mathrm{det}g_S}dV(x,S).
\]
where $\varphi \in \mathfrak{X}_c^1(W)$ and $\phi_t$ is the variation associated with $\varphi$.

\begin{definition}
	\label{def_localBounded1st}
	We say a varifold $V$ has \textit{locally bounded first variation in} $U$ \textit{under metric} $g$ if for each $W\subset \subset U$, there exists a constant $C<\infty$ such that
	\[
		\left|\delta^gV(\varphi)\right|\le C \sup_{U}\left|\varphi\right|_g,\quad \forall \varphi \in \mathfrak{X}_c(W).
	\]
\end{definition}
Given that $V$ has locally bounded first variation in $U$ under metric $g$, by general Riesz Representation Theorem, we can find a Radon measure $\|\delta^gV\|$ on $U$ given by
\[
	\|\delta ^gV\|(W)=\sup_{\varphi \in \mathfrak{X}^1_c(W),|\varphi|_g\le 1} |\delta^g V(\varphi)|,
\]
along with a $\|\delta^g V\|$-measurable $\nu_g$ vector field with $\left|\nu_g(X)\right|_g=1$, for $\|\delta^g V\|$-a.e. $X \in U$ such that
\[
	\delta^g V(\varphi)=-\int_{ U} \left< \nu_g, \varphi \right> _g d\|\delta^g V\|.
\]

Further, it can be expressed as
\[
	\int_{ U} \left< \nu_g, \varphi \right> _g d\|\delta^g V\|=\int_{ U} \left< \boldsymbol{H}_g, \varphi \right> _g d\|V\|_g + \int_{ U} \left< \nu_g, \varphi \right> _g d\sigma
\]
where
\[
	\boldsymbol{H}_g(X):= \lim_{\rho\rightarrow 0^+} \frac{\|\delta^g V\|(B_\rho(X))}{\|V\|_g(B_\rho(X))} \nu_g(X),
\]
and $d \sigma$ denotes the singular part of $\|\delta^g V\|$ with respect to $\|V\|_g$.

Suppose $k=0$ and $\Omega$ is a $\mathcal{H}^n$-measurable set in $U$.
We call $\Omega$ a \textit{Caccioppoli set in $U$ under metric $g$} if $V:=|\Omega|$ has locally bounded first variation in $U$ under metric $g$.

The proposition presented below establishes the equivalence between the given definition of a Caccioppoli set under a $C^1$ metric and the standard definition under the Euclidean metric.

\begin{proposition}
	\label{prop_equivCacci}
	$\Omega$ is a Caccioppoli set in $U$ under metric $g$ if and only if $\Omega$ is a Caccioppoli set in $U$ under Euclidean metric $\delta$.
\end{proposition}
\begin{proof}
	We prove the "if" part first.

	Suppose $\Omega$ is a Caccioppoli set.
	Then, by De Giorgi's structure theorem, we get $\partial^* \Omega$ is countably $n$-rectifiable and $\|\delta|\Omega|\|=\mathcal{H}^{n-1} \lfloor(\partial ^* \Omega)$ and there exists a $\mathcal{H}^{n-1}\lfloor(\partial ^*\Omega)$-measurable vector field $\nu_\Omega$ with $\left|\nu_\Omega(X)\right|=1$ for $\mathcal{H}^{n-1}\lfloor(\partial ^*\Omega)$-a.e., $X \in U$ and
	\[
		\int_{ \Omega} \mathrm{div} \varphi d\mathcal{H}^n=\int_{\partial ^*\Omega } \nu_\Omega \cdot \varphi d\mathcal{H}^{n-1},
	\]
	for any $\varphi \in \mathfrak{X}^1_{c}(U)$.
This leads us to
	\begin{equation}
		\delta^g|\Omega|(\varphi)=
		\int_{ \Omega} \sum_{i =1}^{n}\frac{\partial }{\partial x_i}(\sqrt{\mathrm{det}g}\varphi^i)d\mathcal{H}^n=\int_{ \partial ^*\Omega} \nu_\Omega \cdot \varphi \sqrt{\mathrm{det}g} d\mathcal{H}^{n-1}. 
		\label{eq:pf1stCacci}
	\end{equation}
	where we write $\varphi=\sum_{i =1}^{n}\varphi^i e_i$ and recall that $\mathrm{det}g:= \mathrm{det}(\left< e_i,e_j \right>_g )_{i,j=1}^n$.
	For any $X \in \partial^*\Omega$, we define $\nu_\Omega^g(X)$ to be the unit normal vector of $T_X\partial^*\Omega$ at $X$ under metric $g(X)$ and pointing to the same side as $\nu_\Omega(X)$.
	An elementary calculation in linear algebra yields
	\begin{equation}
		\sqrt{\mathrm{det}(g(X))}=\left< \nu_\Omega(X),\nu_\Omega^g(X) \right>_{g(X)}\sqrt{\mathrm{det}(g_{T_X\partial^*\Omega}(X))}.
		\label{eq:pfDetHyper}
	\end{equation}
	We decompose $\varphi=\varphi^\top+(\varphi \cdot \nu_\Omega)\nu_\Omega$ where $\varphi^\top(X)$ is the vector in $T_X\partial^*\Omega$ for any $X \in \partial^*\Omega$, and then,
	\[
		\left< \varphi(X), \nu_\Omega^g(X) \right> = 
		0+\left< (\varphi(X)\cdot \nu_\Omega(X))\nu_\Omega(X),\nu_\Omega^g(X) \right> = \varphi(X)\cdot \nu_\Omega(X) \left< \nu_\Omega(X),\nu_\Omega^g(X) \right>_{g(X)}.
	\]
	Together with \eqref{eq:pf1stCacci} and \eqref{eq:pfDetHyper}, we arrive at
	\[
		\delta^g|\Omega|(\varphi)=\int_{ \partial ^*\Omega} \left< \nu_\Omega^g, \varphi \right> \sqrt{\mathrm{det}g_{T_X\partial^*\Omega}}  d\mathcal{H}^{n-1}.
	\]
	This equation directly leads to the conclusion that $\Omega$ is a Caccioppoli set in $U$ under metric $g$.

	For the "only if" part,it suffices to verify that De Giorgi's structure theorem is applicable under the metric $g$.
	Subsequently, we can use the same argument as above to obtain $\Omega$ is a Caccioppoli set in $U$ under Euclidean metric $\delta$.
\end{proof}

\begin{remark}
	\label{rmk:Cacci}
	Suppose $\Omega$ is a Caccioppoli set in $U$ and let $V=|\partial^*\Omega|$ be the $(n-1)$-rectifiable varifold associated to $\partial^*\Omega$.
	Then, we can find
	\[
		\delta^g|\Omega|(\varphi)=\int_{ } \left< \nu_\Omega^g, \varphi \right> d\|V\|_g
	\]
	for some vector field $\nu_\Omega^g$ with $\left|\nu_\Omega^g\right|_g=1$ for $\|V\|_g$-a.e. $X \in U$ by the proof of Proposition \ref{prop_equivCacci}.
\end{remark}

For the monotonicity formula, we need to consider the geodesic balls instead of Euclidean balls.
We define the \textit{geodesic ball} $B_\rho^g(X)$ as
\[
	B_\rho^g(X):=\left\{ Y\in U: \mathrm{dist}^g(X,Y)<\rho \right\},
\]
where $\mathrm{dist}_g$ is the distance function under metric $g$.
We consider a point $X \in U$, along with two constants, $\rho_0 \in (0,\mathrm{dist}_g(X,\partial U))$ and $\Lambda \ge 0$ such that
\begin{equation}
	\|\delta^gV\|(B_\rho^g(X))\le \Lambda \|V\|_g(B_\rho^g(X)),\quad \forall \rho \in (0,\rho_0).
	\label{eq:conditionFirstVar}
\end{equation}

At first, we derive the monotonicity formula at point $g=\delta$.
\begin{theorem}
	[Monotonicity Formula]
	Suppose $V$ has a locally bounded first variation in $U$ under metric $g$.
	We also assume $0 \in U$ and \eqref{eq:conditionFirstVar} holds with $X=0$.
	Furthermore, we assume $g(0)=\delta$.
	Then, there exists $\rho_1=\rho_1(n,k,\Lambda,g) \in (0,\rho_0)$ small enough such that
	\begin{equation}
		\frac{\|V\|(B_\sigma)}{\sigma^n}\le \left( 1+C(\mu+\Lambda)\rho \right) \frac{\|V\|(B_\rho)}{\rho^n}
		\label{eq:thmMonotonicityFormula}
	\end{equation}
	for any $0<\sigma<\rho\le \rho_1$.
	Here $C=C(n,k) \in (0,+\infty)$ is a constant and $\mu:=\|Dg\|_{L^\infty(U)}$	\label{thm_monotonicity_formula}.
\end{theorem}

\begin{proof}
	Given that $g(0)=\delta$, we know
	\[
		|g_{ij}(X)-\delta|\le \mu|X|.
	\]
	By choosing $\rho_1=\rho_1(n,k)$ small enough, we get
	\[
		|g^{ij}_S-\delta^{ij}|\le C\mu,\quad |\mathrm{det} g_S -1|\le C\mu
	\]
	for some constant $C=C(n,k) \in (0,+\infty)$.
	Here, $g^{ij}_S$ is the inverse matrix of $g_{ij}^S$ and $g_{ij}^S=g(\tau_i,\tau_j)$.
	Consequently, for any $\varphi \in \mathfrak{X}_c(B_{\rho_1})$, the following inequality holds,
	\[
		\left|\int_{ } \left( \mathrm{div}_S^g \varphi \sqrt{\mathrm{det} g_S}-\mathrm{div}_S \varphi  \right)dV(X,S)
		\right|\le C \mu \int_{ } \left|\varphi\right|+|X|\left|D_S \varphi\right| dV(X,S),
	\]
	for some $C=C(n,k) \in (0,+\infty)$.
	Together with the definition of the first variation, we get
	\begin{align*}
		\left|\int_{ } \mathrm{div}_S \varphi dV(X,S)\right|\le{} & \sup_{B_{\rho}}|\varphi|_g\|\delta^g V\|(B_\rho) +C \mu \int_{ } |\varphi|+|X||D_S \varphi| dV(X,S)\\
		\le{} & C\Lambda\sup_{B_{\rho}}|\varphi|\|V\|_g(B_\rho) +C \mu \int_{ } |\varphi|+|X||D_S \varphi| dV(X,S),
	\end{align*}
	for any $\varphi \in \mathfrak{X}_c(B_\rho)$ with $\rho\le \rho_1$ where we used the condition \eqref{eq:conditionFirstVar} and the fact that $|\varphi|_g\le C|\varphi|$ if we choose $\rho_1=\rho_1(n,k,g)$ small enough.
	Now, we choose
	\[
		\varphi(X)=
		\begin{cases}
		\left( \frac{1}{\sigma^n}-\frac{1}{\rho^n} \right)X, & X \in B_\sigma,\\
		\left( \frac{1}{|X|^n}-\frac{1}{\rho^n} \right)X, & X \in B_\rho\backslash B_\sigma,\\
		0, & \text{otherwise},
		\end{cases}
	\]
	and denote $f(\rho)=\frac{1}{\rho^n}\|V\|(B_\rho)$, then the direct computation (cf. \cite[Page. 778-779]{SchoenSimon1981Regularity}) implies,
	\begin{align}
		&(1-C\mu \sigma)f(\sigma)-(1+C\mu \rho)f(\rho)+ \int_{B_\rho\backslash B_\sigma} \frac{|S^\bot  X|^2}{|X|^{n+2}}dV(X,S)\nonumber \\
		\le{}&C\Lambda \left( \frac{1}{\sigma^n}-\frac{1}{\rho^n} \right) \rho^{n+1}f(\rho)+
		C\mu \int_{ \sigma} ^\rho f(\tau)d\tau,
		\label{eq:pfMonoWithMu}
	\end{align}
	for $0<\sigma<\rho\le \rho_1$ and $C=C(n,k) \in (0,+\infty)$.
	Note that we can assume $C\mu \rho\le \frac{1}{2}$ by choosing $\rho_1$ small enough.
	Consequently, from \eqref{eq:pfMonoWithMu}, for any $0< \frac{1}{2}\rho\le \sigma< \rho\le \rho_1$, it follows that,
	\begin{equation}
		f(\sigma)\le (1+C(\mu+\Lambda)\rho)f(\rho)+C\mu \int_{\sigma}^\rho f(\tau)d\tau.
		\label{eq:pfMonWithIntegral}
	\end{equation}
	We claim that we have
	\begin{equation}
		f(\sigma) \le (1+C(\mu+\Lambda)\rho)f(\rho),\quad \forall 0<\frac{1}{2}\rho\le \sigma<\rho\le \rho_1,
		\label{eq:pfClaimMono}
	\end{equation}
	for some constant $C=C(n,k) \in (0,+\infty)$ if we choose $\rho_1=\rho_1(n,k,\Lambda,g)$ small enough.

	Suppose $\xi \in [\sigma,\rho]$ such that $f(\xi)\ge \frac{1}{2} \sup_{\tau \in [\sigma,\rho]}f(\tau)$.
	Substituting $\sigma$ with $\xi$ or $\rho$ with $\xi$ in \eqref{eq:pfMonWithIntegral} yields two inequalities involving $f(\sigma)$, $f(\xi)$, and $f(\rho)$.
	By choosing $\rho_1=\rho_1(n,k,\Lambda,g)$ small enough, we can eliminate the term containing $f(\xi)$ to get \eqref{eq:pfClaimMono}.

	Iteratively applying \eqref{eq:pfClaimMono} by setting $\sigma$ to $\frac{\rho}{2^{j+1}}$ and $\rho$ to $\frac{\rho}{2^j}$ for $j = 0, 1, 2, \ldots$, we obtain
	\begin{equation}
		f(\sigma)\le \prod_{i=0}^j \left( 1+C(\mu+\Lambda)\frac{\rho}{2^i} \right)f(\rho)\le (1+C(\mu+\Lambda)\rho)f(\rho),
		\label{eq:pfMonoLocallyBoundedFirst}
	\end{equation}
	for any $\sigma \in \left[ \left.\frac{\rho}{2^{j+1}},\frac{\rho}{2^{j}}\right. \right)$ by choosing $\rho_1$ small enough if necessary.
	Here, the constant $C$ depends on $n,k$ only.
	This is what we want to prove.
\end{proof}

\begin{corollary}
	\label{cor_MonoFormula}
	Suppose $V$ is a varifold in $U$ with locally bounded first variation under the metric $g$, and condition \eqref{eq:conditionFirstVar} is satisfied for some point $X \in U$.
	Then, there exists $\rho_1=\rho_1(n,k,g,\Lambda) \in (0,\rho_0)$ such that for any $0<2\sigma<\rho\le \rho_1$, we have
	\begin{equation}
		\frac{\|V\|_g(B_\sigma^g(X))}{\sigma^n}\le
		(1+C\rho) \frac{\|V\|_g(B_\rho^g(X))}{\rho^n},
		\label{eq:corMonoFormula}
	\end{equation}
	for $C=C(n,k,g,\Lambda) \in (0,+\infty)$.
\end{corollary}
\begin{proof}
	We only need to consider the case that $0 \in U$ and $g(0)=\delta$ up to an affine transformation.

	Let $\mu=\|Dg\|_{L^{\infty}(U)}$.
	At first, we choose $\rho_1$ such that Theorem \ref{thm_monotonicity_formula} holds and adjust $\rho_1$ if necessary.
	As $\left|g-\delta\right|\le C\rho \mu$ in $B_\rho$, for any $C^1$ curve $\gamma$ in $B_\rho$, we have
	\[
		\left|L(\gamma)-L^g(\gamma)\right|\le C\rho \mu L(\gamma).
	\]
	With $\rho_1$ small enough, it can be shown that
	\begin{align*}
		B_{(1-C\mu\rho)\rho}^g\subset B_\rho \subset B_{(1+C\mu\rho)\rho}^g.
	\end{align*}
	Again, using $\left|g-\delta\right|\le C\mu \rho$ in $B_\rho$ and the definition of $\|V\|_g$ in \eqref{eq:defWeightMeasure}, we deduce that
	\[
		(1-C\mu\rho)\|V\|_g(B_\rho)\le \|V\|(B_\rho)\le (1+C\mu\rho)\|V\|_g(B_\rho).
	\]
	Together with \eqref{eq:thmMonotonicityFormula}, we get 
	\[
		(1-C\mu \sigma) \frac{\|V\|_g(B_{(1-C\mu\sigma)\sigma}^g)}{\sigma^n}\le
		(1+C\mu\rho) \frac{\|V\|_g(B_{(1+C\mu\rho)\rho}^g)}{\rho^n}.
	\]

	Now, we replace $(1-C\mu \sigma)\sigma$ by $\sigma$ and $(1+C\mu\rho)\rho$ by $\rho$ to get \eqref{eq:corMonoFormula} with $X=0$ by choosing $\rho_1=\rho_1(n,k,g,\Lambda)$ if necessary.
\end{proof}

For $\|V\|_g$-a.e. $X \in U$, there exists $\rho_0 \in (0, \mathrm{dist}_g(X,\partial U))$ and $\Lambda$ such that \eqref{eq:conditionFirstVar} holds by Differentiation Theorem (cf. \cite[Lemma 40.5]{simon1983lectures}). Note that we should replace Euclidean balls with geodesic balls in the Differentiation Theorem.

\begin{proposition}
	\label{prop_Density}
	Suppose $V$ has a locally bounded first variation in $U$ under $C^1$ metric $g$.
	The density function $\Theta^g(\|V\|_g,X)$ defined by
	\[
		\Theta^g(\|V\|_g,X):=\lim_{\rho\rightarrow 0^+} \frac{\|V\|_g(B_\rho^g(X))}{\omega_n \rho^n}
	\]
	exists for $\|V\|_g$-a.e. $X\in U$.
	Furthermore, $\Theta^g(\|V\|_g,X)$ is a $\|V\|_g$-measurable function.
\end{proposition}

\begin{remark}
	When $V$ is rectifiable, we can verify that
	\begin{equation}
		\Theta^g(\|V\|_g,X)=\Theta(\|V\|,X),\quad \text{ for }\|V\|\text{-a.e. }X \in U.
		\label{eq:rmkSameDensity}
	\end{equation}
	However, it is worth noting that \eqref{eq:rmkSameDensity} may not hold in general.
\end{remark}

\begin{theorem}
	[Semi-continuity of $\Theta^g$ under varifold convergence]
	Let $\left\{ V_i \right\}$ be a sequence of varifolds in $U$ with locally bounded first variation under $C^1$ metric $g_i$ such that $V_i \rightarrow V$ as Radon measures.
	We assume $\Theta ^{g_i}(\|V_i\|_{g_i},X)\ge 1$ except on a set $B_i \subset U$ with $\|V_i\|_{g_i}(B_i\cap W)\rightarrow 0$ and $\liminf_{i\rightarrow +\infty} \|\delta^{g_i}V_i\|(W)<\infty$ for each $W \subset \subset U$.
	We also assume the metrics $g_i$ satisfies $\sup_{i}(|g_i|+|Dg_i|)<+\infty$ and converges to a $C^1$ metric $g$ in $C^1$ sense in $U$ as $i\rightarrow +\infty$.
	Under these conditions, $V$ is a $n$-varifold with locally bounded first variation under metric $g$ and it satisfies $\|\delta^g V\|(W)\le \liminf_{i\rightarrow +\infty} \|\delta^{g_i}V_i\|(W)$ for all $W \subset \subset U$ and $\Theta^g(\|V\|,X)\ge 1$ for $\|V\|_g$-a.e. $X \in U$.
	\label{thm_semi_continuity_of_theta_g_under_varifold_convergence}
\end{theorem}

\begin{proof}
	The detailed proof is omitted here, as it parallels the approach used for Theorem 40.6 in \cite{simon1983lectures}, adjusted for the $C^1$ metrics set.
\end{proof}

The Rectifiability theorem for varifolds under metric $g$ is a fundamental part of the theory of varifolds.
We can establish the following theorem for locally bounded first variation under metric $g$.
\begin{theorem}
	[Rectifiability Theorem]
	Suppose $V$ has locally bounded first variation in $U$ under $C^1$ metric $g$ and $\Theta^g(\|V\|_g,X)>0$ for $\|V\|_g$-a.e. $X \in U$.
	Then, $V$ is an $n$-rectifiable varifold.
	\label{thm_rectifiability}
\end{theorem}
\begin{proof}
	This proof can be viewed as a straightforward adaptation of Theorem 42.4 in \cite{simon1983lectures}.
\end{proof}

Another important result is the compactness theorem for varifolds with locally bounded first variation.
\begin{theorem}
	[Compactness theorem for $n$-varifolds under $C^1$ metric]
	Suppose $V_i$ is a rectifiable $n$-varifolds that has locally bounded first variation in $U$ under metric $g_i$.
	We assume $\sup_i |g_i|+|Dg_i|<+\infty$ and there exists a $C^1$ metric $g$ such that $g_i\rightarrow g$ in $C^1$ sense, along with
	\[
		\sup_{i}\left( \|V_i\|_{g_i}(W)+\|\delta^{g_i}V_i\|(W) \right)<\infty\quad 
		\forall W \subset \subset U,
	\]
	and $\Theta^g(\|V_i\|_{g_i},X)\ge 1$ on $U \backslash A_i$ where $\|V_i\|_{g_i}(A_i\cap W)\rightarrow 0$ as $i\rightarrow 0, \forall W \subset \subset U$.
	Then, up to a subsequence, there exists an $n$-varifold $V$ in $U$ with locally bounded first variation under metric $g$ such that $V_i\rightarrow V$ in the sense of Radon measures, $\Theta^g(\|V\|_g,X)\ge 1$ for $\|V\|_g$-a.e. $X \in U$, and $\|\delta^g V\|(W)\le \liminf_{i\rightarrow +\infty} \|\delta^{g_i}V_i\|(W)$ for each $W\subset \subset U$.
	
	In particular, if all the $V_i$ has integer multiplicity, then so does $V$.
	\label{thm_compactness_theorem_for_n_varifolds_under_lipschitz_metric}
\end{theorem}
\begin{proof}
	The argument is similar to the proof of Theorem 42.7 in \cite{simon1983lectures}.
	See the proof of Theorem \ref{thm_compactnessRIV} for the necessary adjustments.
\end{proof}

\subsection{$\mu$-stationary and regularity}%
\label{sub:_mu_stationary_and_regularity}

We now introduce the notion of $\mu$-stationary condition under a $C^1$ metric.

\begin{definition}
	Given an $n$-varifold $V$ and a $C^1$ metric $g$ on $U\subset \mathbb{R}^{n+k}$, we say $(V,g)$ is \textit{$\mu$-stationary in $U$} if the metric $g$ satisfies $\|g-\delta\|_{C^1(U)}\le \mu$, and for any vector field $\varphi \in \mathfrak{X}^1_c(U)$, we have
	\[
		\left|\delta^g V(\varphi)\right|\le \mu \int_{ } |\varphi|_g d\|V\|_g.
	\]
\end{definition}

For simplicity, we write
\[
	\mu(V,g)=\inf\{ \mu: (V,g)\text{ is }\mu \text{-stationary} \}.
\]
Notably, $(V,g)$ is $\mu$-stationary 
can imply $V$ has locally bounded first variation in $U$ under metric $g$,
enabling the application of results from the previous subsection to $\mu$-stationary varifolds.
In particular, for any $\mu$-stationary $(V,g)$, if we require $\mu=\mu(n,k,U)$ small enough instead of $\rho$ small enough in the proof of \eqref{eq:thmMonotonicityFormula} (Note that $\Lambda$ is 0 here.), we can derive the following monotonicity formula,
\[
	\frac{\|V\|_g(B_\sigma^g(X))}{\sigma^n}\le (1+C\mu \rho) \frac{\|V\|_g(B_\rho^g(X))}{\rho^n}, \forall X \in U, \text{ and } \forall 0<2\sigma<\rho<\mathrm{dist}_g(X,\partial U)
\]
for $C=C(n,k,U)$.

The final significant theorem is due to Allard \cite{Allard1972}.
By adapting the proof to $C^1$ metric case, we can establish the following version of Allard regularity theorem.

\begin{theorem}
	[Allard Regularity]
	\label{thm_allard}
	Let $V$ be an integral rectifiable $n$-varifold in $B_1^{n+k}(0)$ and $g$ a $C^1$ metric on $B_1^{n+k}(0)$.
	Suppose $(V,g)$ is a $\mu$-stationary varifold in $B_1^{n+k}(0)$ and $P$ is an $n$-dimensional plane in $\mathbb{R}^{n+k}$ with $0 \in P$.
	Given $\gamma ,\delta\in (0,1)$, then there exists a number $\varepsilon=\varepsilon(n,k,\gamma,\delta) \in (0,1)$ such that if the following condition holds,
	\begin{enumerate}[\normalfont(a)]
		\item $\mu<\varepsilon$,
			\label{itAllThm1}
		\item $\mathrm{spt}\|V\|\cap B_{\frac{1}{2}}(0)\neq \emptyset$ and $\frac{1}{\omega_n}\|V\|(B_1(0))<1+\delta$,
			\label{itAllThm2}
		\item $\int_{ B_1(0)} \mathrm{dist}^2(X,P)d\|V\|(X)< \varepsilon$.
			\label{itAllThm3}
	\end{enumerate}
	Then, we have
	\[
		\mathrm{spt}\|V\|\cap B_{\gamma}^{n+k}(0)\subset \mathrm{graph}u
	\]
	with $u \in C^{1,\beta}(B_{\gamma}^{n+k}(0)\cap P,P^\bot)$ for some $\beta \in (0,1)$ and
	\[
		\|u\|^2_{C^{1,\beta}(B_{\gamma}(0)\cap P,P^\bot )}\le C\left( \int_{B_1(0)} \mathrm{dist}^2(X,P)d\|V\|(X)+\mu \right).
	\]

	Here, the constant $C=C(n,k,\gamma,\delta) \in (0,+\infty)$ and $\beta=\beta(n,k,\gamma,\delta)$. $P(X)$ denotes the orthogonal projection of $X$ to plane $P$.
\end{theorem}

\section{Stationary triple-junction networks on half sphere}%
\label{sec:append_stationaryNetwork}

In this section, we prove that the total length of each stationary triple-junction network on a half-sphere is bounded below by $\pi$, and this minimum is attained exclusively when the network forms a half-great circle.

For simplicity, we write $S_+^2:=\partial B^3_1 \cap \mathbb{H}^{3}$ to be the half-sphere in $\mathbb{H}^3$.

We denote $\mathcal{N}$ a collection of smooth regular curves $\{ \gamma_i \}_{i \in \mathcal{I}}$, and $\mathcal{P}$ the set of all possible endpoints of $\gamma_i$.
We say $\mathcal{N}=\{\gamma_i\}_{i \in \mathcal{I}}$ is a \textit{stationary triple-junction network} on $S_+^2$ if it satisfies the following conditions,
\begin{enumerate}[(a)]
	\item Each $\gamma_i$ is an embedded geodesic (spherical arc) in $S_+^2$, and $\gamma_i$ does not lies in $\partial S_+^2$,
		\label{it:statGeodesic}
	\item Intersection between $\gamma_i$ and $\gamma_j$ only occurs at their endpoints for distinct $i, j \in \mathcal{I}$,
		\item At each point $X \in \mathcal{P}$, we have either $X \in \partial \mathbb{H}^3$ or there exists three geodesic arcs $\gamma_{i_1},\gamma_{i_2},\gamma_{i_3}$ such that $X$ is the common endpoint of $\gamma_{i_1},\gamma_{i_2},\gamma_{i_3}$ and the contact angle between $\gamma_{i_p},\gamma_{i_q}$ is $\frac{\pi}{3}$ for $1\le p\neq q\le 3$ at $X$.
			\label{it:statNetwork}
\end{enumerate}

We define the total length of the network as
\[
	L(\mathcal{N}):=\sum_{i \in \mathcal{I}} L(\gamma_i).
\]
where $L(\gamma_i)$ is the length of $\gamma_i$.
From the definition of stationary triple-junction network, we know $\mathcal{P} \cap \{ x_1>0 \}$ is discrete and hence at most countable by \ref{it:statGeodesic}.
This indicates that $\mathcal{I}$ is at most countable, ensuring that $L(\mathcal{N})$ is well-defined.
\begin{theorem}
	
	Suppose $\mathcal{N}=\{\gamma_i\}_{i \in \mathcal{I}}$ is a stationary triple-junction network on $S_+^2$
	Then, we have 
	\[
		L(\mathcal{N})\ge \pi
	\]
	and equality holds precisely when if $\mathcal{N}=\left\{ \gamma \right\}$ where $\gamma$ is a half-great circle.
\end{theorem}
\begin{proof}
	For simplicity, we identify $S_+^2$ with the unit disk $\mathbb{D}:=\left\{ X=(x_1,x_2):|X|\le 1 \right\}$ via stereographic projection.
	We denote $g$ the metric on $\mathbb{D}$ induced by this projection.
Accordingly, each $\gamma_i$ is viewed as a smooth regular curve in $\mathbb{D}$.
	Let $N=\cup_{i \in \mathcal{I}}\gamma_i$ be the union of all $\gamma_i$ and denote $\mathbb{\mathring D}$ the interior of $\mathbb{D}$.
	It is easy to see that $\mathbb{\mathring D}\backslash N$ is a disjoint union of open sets.
	Defining $\mathcal{O}$ as the set of the connected component of $\mathbb{\mathring D}\backslash N$, it is straightforward to observe that each $U \subset \mathcal{O}$ is geodesic convex.
	\begin{lemma}
		If $L(\mathcal{N})\le \pi$, then for any $U \in \mathcal{O}$, we have $\overline{U}\cap \partial \mathbb{D}\neq \emptyset$.
		\label{lem_appendNonEmptyBoundary}
	\end{lemma}
	\begin{proof}
		We argue by contradiction.
		Suppose $L(\mathcal{N})\le \pi$ and $\overline{U}\cap \partial \mathbb{D}=\emptyset$.
		By the definition of stationary triple-junction network, we know $U$ is a spherical $m$-polygon for some $m\ge 3$ and its interior angle at each vertex is $\frac{2\pi}{3}$.
		Gauss-Bonnet formula tells us $m\le 5$, and it implies $|U|_g\ge \frac{\pi}{3}$, where $|U|_g$ is the area of $U$ under metric $g$.
		Invoking the isoperimetric inequality on sphere yields
		\[
			L_g(\partial U)\ge \frac{\pi \sqrt{11}}{3}>\pi.
		\]
		where $L_g(\partial U)$ refers to the length of $\partial U$ under metric $g$.
		This conclusion contradicts with $L(\mathcal{N})\le \pi$.
	\end{proof}

	Now, we select $U \subset \mathcal{O}$ such that $0 \in \overline{U}$ and denote $X_0$ to be the point on $\partial U$ such that $\mathrm{dist}_g(0,\partial U)=\mathrm{dist}_g(0,X_0)$.
	Next, we select $\gamma \in \mathcal{N}$ such that $X_0 \in \gamma$.
	Notably, if $0 \notin \partial U$, we know $X_0$ is in the interior of $\gamma$ by \ref{it:statNetwork} for the definition of stationary triple-junction network.
	
	Let $P,P'$ represent two endpoints of $\gamma$.
	If $P,P' \in \partial\mathbb{D}$, the proof is trivial as $L_g(\gamma)=\pi$.

	Now, we assume $P \in \mathbb{\mathring D}$.
	We denote $\tilde{\gamma}$ the great half-circle containing $\gamma$, $Q,Q'$ the two endpoints as its endpoints, arranged with $Q,P,X_0,P',Q'$ in consecutive order.
	See Figure \ref{fig:choice-of-u} for an illustration of the selection for $U$ and $\gamma$.
\begin{figure}[ht]
    \centering
	\begingroup
	\fontsize{7pt}{12pt}
	\def\svgwidth{0.8\columnwidth}
	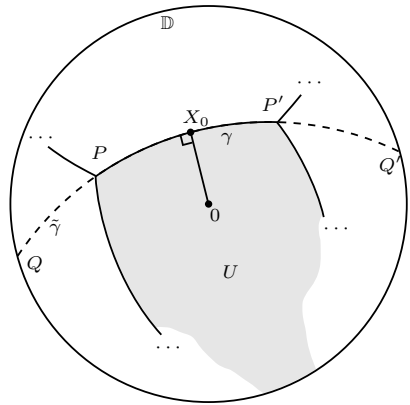
	\endgroup

    \caption{Choice of $U$ and $\gamma$}
    \label{fig:choice-of-u}
\end{figure}

	It is easy to see that
	\begin{equation}
		\mathrm{dist}_g(0,\tilde{\gamma})=\mathrm{dist}_g(0,\gamma)=\mathrm{dist}_g(0,X_0).
		\label{eq:pfDistDisk}
	\end{equation}
	Let $\tilde{U}$ be the connected component of $\mathbb{\mathring D}\backslash \tilde{\gamma}$ such that $U\cap \tilde{U}\neq \emptyset$, where it is evident that $U \subset \tilde{U}$ due to the geodesic convexity of $U$ and the fact that $0$ is in the closure of $\tilde{U}$.

Consequently, it can be verified that
	\begin{equation}
		\mathrm{dist}_g(P,X)\ge \mathrm{dist}_g(P,Q),\quad 
	\mathrm{dist}_g(P',X)\ge \mathrm{dist}_g(P',Q'),\quad \forall X \in \partial \tilde{U}\cap \partial\mathbb{D},
	\end{equation}
	by \eqref{eq:pfDistDisk} and the fact $X_0 \in \gamma$.

	We write $A$ to be the connected component of $\partial U \cap \mathbb{\mathring D}\backslash \gamma$ whose closure contains $X_0$.
	By Lemma \ref{lem_appendNonEmptyBoundary}, there exists a point $X_1\in \partial \tilde{U}\cap \partial\mathbb{D}$ in the closure of $A$ as $U \in \tilde{U}$.
	Therefore, we have $L_g(A)\ge \mathrm{dist}_g(P,X_1)\ge \mathrm{dist}_g(P,Q)$.

	Similarly, if $P'\neq Q'$, we can find another component $A'$ of $\partial U \backslash \mathbb{\mathring D}$ with $L_g(A')\ge \mathrm{dist}_g(P',Q')$.

	This implies $L_g(\partial U)\ge L_g(\tilde{\gamma})=\pi$ and equality holds if and only if $\gamma=\tilde{\gamma}$ and $\mathcal{N}=\left\{ \gamma \right\}$.
	\end{proof}
	\begin{remark}
		Even when $\overline{U} \cap \partial \mathbb{D} \neq \emptyset$, it is straightforward to deduce that $U$ is a spherical $m$-polygon for some $3\le m\le 6$ as well.
	\end{remark}
\bibliographystyle{alpha}
\bibliography{../references.bib}

\end{document}